\documentclass[twoside,12pt]{article}
\usepackage[utf8]{inputenc}
\usepackage[T1]{fontenc}
\usepackage[english]{babel} 
\usepackage{amsmath,amssymb,amsthm,amsfonts,dsfont,mathtools}
\usepackage{stmaryrd}
\usepackage{mathabx}
\usepackage{bbm}
\usepackage{graphicx}
\usepackage[left=2.5cm,right=2.5cm,top=2cm,bottom=2cm]{geometry}
\usepackage{hyperref}
\usepackage{color}
\usepackage{enumitem}

\theoremstyle{plain}
\newtheorem{theorem}{Theorem}
\newtheorem{lemma}[theorem]{Lemma}

\newtheorem{proposition}[theorem]{Proposition}

\newtheorem*{theorem*}{Theorem}
\newtheorem*{claim*}{Claim}

\theoremstyle{definition}

\theoremstyle{remark}
\newtheorem{remark}[theorem]{Remark}
\newtheorem*{remark*}{Remark}

\numberwithin{equation}{section}
\numberwithin{theorem}{section}

\def\N{\ensuremath{\mathbb{N}}}

\def\R{\ensuremath{\mathbb{R}}}

\def\ep{\varepsilon}
\def\eps{\varepsilon}
\def\Var{\textup{Var}}

\def\E{\ensuremath{\mathbf{E}}}
\def\P{\ensuremath{\mathbf{P}}}
\def\F{\ensuremath{\mathcal{F}}}

\def\Ind{\ensuremath{\mathbbm{1}}}

\def\to{\rightarrow}

\def\tif{\ensuremath{\text{ if }}}

\renewcommand{\L}{\mathcal{L}}

\begin{document}

\title{Yaglom-type limit theorems for branching Brownian motion with absorption}
\author{Pascal Maillard\thanks{Institut de Mathématiques de Toulouse, CNRS UMR5219. \emph{Postal address:} Institut de Mathématiques de Toulouse, Université de Toulouse 3 Paul Sabatier, 118 route de Narbonne, 31062 Toulouse cedex 9, France} \\ Université de Toulouse \and Jason Schweinsberg\thanks{\emph{Postal address:} Department of Mathematics 0112; 9500 Gilman Drive; La Jolla, CA 92093-0112} \\ University of California San Diego}

\maketitle

\begin{abstract}
We consider one-dimensional branching Brownian motion in which particles are absorbed at the origin.  We assume that when a particle branches, the offspring distribution is supercritical, but the particles are given a critical drift towards the origin so that the process eventually goes extinct with probability one.  We establish precise asymptotics for the probability that the process survives for a large time $t$, building on previous results by Kesten (1978) and Berestycki, Berestycki, and Schweinsberg (2014).  We also prove a Yaglom-type limit theorem for the behavior of the process conditioned to survive for an unusually long time, providing an essentially complete answer to a question first addressed by Kesten (1978).  An important tool in the proofs of these results is the convergence of a certain observable to a continuous state branching process.  Our proofs incorporate new ideas which might be of use in other branching models.
\end{abstract}



\section{Introduction and main results}

We will consider one-dimensional branching Brownian motion with absorption, which evolves according to the following rules.  At time zero, all particles are in $(0, \infty)$.  Each particle independently moves according to one-dimensional Brownian motion with a drift of $-\mu$.  Particles are absorbed when they reach zero.  Each particle independently branches at rate $\beta$.  When a branching event occurs, the particle dies and is replaced by a random number of offspring.  We denote by $p_k$ the probability that an individual has $k$ offspring and assume that the numbers of offspring produced at different branching events are independent.  We define $m$ so that $m + 1 = \sum_{k=1}^{\infty} kp_k$ is the mean of the offspring distribution, and we assume $m > 0$.  We also assume that $\sum_{k=1}^{\infty} k^2 p_k < \infty$, so the offspring distribution has finite variance.  Finally, we assume that $\beta = 1/2m$, which by scaling arguments entails no real loss of generality because speeding up time by a factor of $r$ while scaling space by a factor of $1/\sqrt{r}$ multiplies the branching rate by $r$ and the drift by $\sqrt{r}$.

Kesten \cite{kesten} showed that if $\mu < 1$, the process survives forever with positive probability, while if $\mu \geq 1$, the process eventually goes extinct almost surely.  In this paper, we will assume that $\mu = 1$, so we are considering the case of critical drift.  Our aim is to establish some new results for this process, focusing on the question of how the process behaves when conditioned to survive for an unusually long time.  These results sharpen some of the results in the seminal paper of Kesten \cite{kesten} and build on more recent work of Berestycki, Berestycki, and Schweinsberg \cite{bbs, bbs3}.

There are several reasons to study branching Brownian motion with absorption:
\begin{enumerate}
\item To study branching Brownian motion without absorption, for example its extremal particles, it is often useful to kill the particles at certain space-time barriers, as pioneered by Bramson~\cite{Bramson1978}. It is therefore natural to study these processes for their own sake.

\item Branching Brownian motion with absorption, or more complicated models that build on it, can be interpreted as a model of a biological population under the influence of evolutionary selection.  In this setting, particles represent individuals in a population, the position of a particle represents the fitness of the individual, and the absorption at zero models the deaths of individuals with low fitness.  See, for example, the work of Brunet, Derrida, Mueller, and Munier \cite{bdmm1, bdmm2}.

\item There are close connections between branching Brownian motion and partial differential equations, going back to the early work of McKean \cite{mckean}.  Branching Brownian motion with absorption was used in \cite{hhk06} to study the equation $\frac{1}{2} f'' - \rho f' + \beta (f^2 - f)$, which describes traveling wave solutions to the FKPP equation, under the boundary conditions $\lim_{x \downarrow 0} f(x) = 1$ and $\lim_{x \rightarrow \infty} f(x) = 0$.

\item The branching random walk with absorption, a discrete-time analogue of branching Brownian motion with absorption, appears directly or indirectly in other mathematical models such as infinite urn models 
\cite{mr2016} or in the study of algorithms for finding vertices of large labels in a labelled tree generated by a branching random walk \cite{Aldous1992}. Also, branching Brownian motion with absorption is a toy model for certain noisy travelling wave equations (see again \cite{bdmm1,bdmm2}).

\item Branching Brownian motion with absorption can be regarded as a non-conservative Markov process living in an infinite-dimensional and unbounded state space (the space of finite collections of points on $\R_+$). As such, it is an interesting testbed for quasi-stationary distributions and Yaglom limits, which have seen a great deal of attention in the last decade \cite{cclmms,mv12,cv16}, particularly regarding approximating particle systems \cite{afgj,delmoral}
\end{enumerate}

\subsection{Some notation}

We introduce here some notation that we will use throughout the paper.  When the branching Brownian motion starts with a single particle at the position $x$, we denote probabilities and expectations by $\P_x$ and $\E_x$ respectively.  More generally, we may start from a fixed or random initial configuration of particles in $(0, \infty)$, which we represent by the measure $\nu$ consisting of a unit mass at the position of each particle.  We then denote probabilities and expectations by $\P_{\nu}$ and $\E_{\nu}$.  To avoid trivialities, we will always assume that the initial configuration of particles is nonempty.  When the initial configuration $\nu$ is random, $\P_{\nu}$ and $\E_{\nu}$ refer to unconditional probabilities and expectations, not conditional probabilities and expectations given the random measure $\nu$.  In particular, if $A$ is an event, then $\P_{\nu}(A)$ is a number between 0 and 1, not a random variable whose value depends on the realization of $\nu$.  That is, using the language of random walks in random environments, ${\bf P}_{\nu}$ represents the “annealed” law rather than the “quenched” law.  We will denote by $({\cal F}_t, t \geq 0)$ the natural filtration of the process.

We will denote by $N_s$ the set of particles that are alive at time $s$, meaning they have not yet been absorbed at the origin.  If $u \in N_s$, we denote by $X_u(s)$ the position at time $s$ of the particle $u$.  We also define the critical curve
\begin{equation}\label{Lcdef}
L_t(s) = c(t - s)^{1/3}, \hspace{.5in} c = \bigg( \frac{3 \pi^2}{2} \bigg)^{1/3}.
\end{equation}
This critical curve appeared in the original paper of Kesten \cite{kesten}.  As will become apparent later, it can be interpreted, very roughly, as the position where a particle must be at time $s$ in order for it to be likely to have a descendant alive in the population at time $t$.  We will also define
\begin{equation}\label{Zdef}
Z_t(s) = \sum_{u \in N_s} z_t(X_u(s), s), \hspace{.5in}z_t(x,s) = L_t(s) \sin \bigg( \frac{\pi x}{L_t(s)} \bigg) e^{x - L_t(s)} \Ind_{x \in [0, L_t(s)]}.
\end{equation}
The process $(Z_t(s), 0 \leq s \leq t)$ will be extremely important in what follows.  Lemma \ref{lem:Zexp} below shows that, in some sense, this process is very close to being a martingale.  Let $M(s)$ be the number of particles at time $s$, and let $R(s)$ denote the position of the right-most particle at time $s$.  In symbols, we define
\begin{equation}\label{MRdef}
M(s) = \#N_s, \hspace{.4in} R(s) = \sup\{X_u(s): u \in N_s\}.
\end{equation}
Finally, let $$\zeta = \inf\{s: M(s) = 0\}$$ be the extinction time for the process.

We will often be working to prove asymptotic results as $t \rightarrow \infty$ where, for each $t$, we are working under a different probability measure such as ${\bf P}_{\nu_t}$ or the conditional probability $\P_{\nu_t}(\:\cdot \:|\zeta > t)$.  We use $\Rightarrow$ to denote convergence in distribution and $\rightarrow_p$ to denote convergence in probability.  If $X_t$ is a random variable for each $t$, then by $X_t \rightarrow_p c$ under $\P_{\nu_t}$ we mean that $\P_{\nu_t}(|X_t - c| > \eps) \rightarrow 0$ as $t \rightarrow \infty$ for all $\eps > 0$, while $X_t \rightarrow_p \infty$ under $\P_{\nu_t}$ means $\P_{\nu_t}(X_t > K) \rightarrow 1$ as $t \rightarrow \infty$ for all $K \in \R_+$.  We also write $f(t) \sim g(t)$ if $\lim_{t \rightarrow \infty} f(t)/g(t) = 1$ and $f(t) \ll g(t)$ if $\lim_{t \rightarrow \infty} f(t)/g(t) = 0$.

Throughout the paper, $C$ denotes a positive constant whose value may change from line to line.  Numbered constants $C_k$ keep the same value from one occurrence to the next.

\subsection{The probability of survival until time \texorpdfstring{$t$}{t}}

For branching Brownian motion started with a single particle at $x > 0$, we are interested in calculating the probability that the process survives at least until time $t$.  Kesten \cite{kesten} showed that there exists a positive constant $K_1$ such that for every fixed $x > 0$, we have for sufficiently large $t$,
$$x e^{x - L_t(0) - K_1(\log t)^2} \leq \P_x(\zeta > t) \leq (1 + x)e^{x - L_t(0) + K_1(\log t)^2}.$$
Berestycki, Berestycki, and Schweinsberg \cite{bbs} tightened these bounds by showing that there are positive constants $C_1$ and $C_2$ such that for all $x > 0$ and $t > 0$ such that $x \leq L_t(0) - 1$, we have
\begin{equation}\label{oldbound}
C_1 z_t(x,0) \leq \P_x(\zeta > t) \leq C_2 z_t(x,0).
\end{equation}
Note that the results in \cite{bbs} are stated in the case when $p_2 = 1$, which means two offspring are produced at each branching event.  However, the proof of (\ref{oldbound}) can be extended, essentially without change, to the case of the more general supercritical offspring distributions considered here.  Also, a slightly different scaling, with $\beta = 1$ and $\mu = \sqrt{2}$, was used in \cite{bbs}.  Theorem \ref{survival} below is our main result regarding survival probabilities.

\begin{theorem}\label{survival}
Suppose that for each $t > 0$, we have a possibly random initial configuration of particles $\nu_t$.  Then there is a constant $\alpha > 0$ such that the following hold:
\begin{enumerate}
\item Suppose that, under $\P_{\nu_t}$, we have $Z_t(0) \Rightarrow Z$ and $L_t(0) - R(0) \rightarrow_p \infty$ as $t \rightarrow \infty$.  Then $$\lim_{t \rightarrow \infty} \P_{\nu_t}(\zeta > t) = 1 - \E[e^{-\alpha Z}].$$

\item Suppose that each $\nu_t$ is deterministic, and that, under $\P_{\nu_t}$, we have $Z_t(0) \rightarrow 0$ and $L_t(0) - R(0) \rightarrow \infty$ as $t \rightarrow \infty$.  Then as $t \rightarrow \infty$, we have $$\P_{\nu_t}(\zeta > t) \sim \alpha Z_t(0).$$  In particular, if $x > 0$ is fixed, then
\begin{equation}\label{scsurvival}
\P_x(\zeta > t) \sim \alpha \pi x e^{x-L_t(0)}.
\end{equation}
\end{enumerate}
%We have $\alpha = \pi^{-1}e^{-a_{\ref{eq:abelian}}-3/4}$, where $a_{\ref{eq:abelian}}$ is the constant from Lemma~\ref{neveuW} below.
\end{theorem}

\begin{remark}
The constant $\alpha$ in the statement of Theorem~\ref{survival} has the expression $\alpha = \pi^{-1}e^{-a_{\ref{eq:abelian}}-3/4}$, where $a_{\ref{eq:abelian}}$ is a constant related to the tail of the \emph{derivative martingale} of the branching Brownian motion and defined in Lemma~\ref{neveuW} below. 
\end{remark}

Note that (\ref{scsurvival}) improves upon (\ref{oldbound}) when the initial configuration has only a single particle.  Derrida and Simon \cite{ds07} had previously obtained (\ref{scsurvival}) by nonrigorous methods.

Theorem \ref{survival} applies when there is no particle at time zero that is close to $L_t(0)$.  This condition is important, here and throughout much of the paper, because it ensures that no individual particle at time zero has a high probability of having descendants alive at time $t$.  Theorem~\ref{survivex} below applies when the process starts with one particle near $L_t(0)$.  Here and throughout the rest of the paper, we denote by $q$ the extinction probability for a Galton-Watson process with offspring distribution $(p_k)_{k=0}^{\infty}$.  Note that Theorem \ref{survivex} implies that when $q = 0$ and the process starts from one particle near $L_t(0)$, the fluctuations in the extinction time are of the order $t^{2/3}$, which can also be seen from Theorem 2 in \cite{bbs}.

\begin{theorem}\label{survivex}
There is a function $\phi: \R \rightarrow (0,1)$ such that for all $x \in \R$, 
\begin{equation}\label{convphi}
\lim_{t \rightarrow \infty} \P_{L_t(0) + x}(\zeta \leq t) = \phi(x)
\end{equation}
and, more generally, for all $x \in \R$ and $v \in \R$, 
\begin{equation}\label{convphi2}
\lim_{t \rightarrow \infty} \P_{L_t(0)+x}(\zeta \leq t + vt^{2/3}) = \phi\Big(x - \frac{cv}{3}\Big).
\end{equation}
We have $\lim_{x \rightarrow -\infty} \phi(x) = 1$ and $\lim_{x \rightarrow \infty} \phi(x) = q$.  The function $\phi$ also satisfies $$\frac{1}{2} \phi'' - \phi' = \beta(\phi - f \circ \phi),$$ where $f(s) = \sum_{k=0}^{\infty} p_k s^k$ is the generating function for the offspring distribution $(p_k)_{k=1}^{\infty}$.  In fact, $\phi(x) = \psi(x+\log(\alpha\pi))$, where $\psi$ is the function from Lemma~\ref{neveuW} below and $\alpha$ is the constant from Theorem \ref{survival}.

\end{theorem}

\subsection{The process conditioned on survival}

Our main goal in this paper is to understand the behavior of branching Brownian motion conditioned to survive for an unusually long time.  The results in this section can be viewed as the analogs of the theorem of Yaglom \cite{yaglom} for critical Galton-Watson processes conditioned to survive for a long time.

Proposition \ref{extinctiondist}, which is a straightforward consequence of Theorem \ref{survival}, gives the asymptotic distribution of the survival time for the process, conditional on $\zeta > t$.  We see that the amount of additional time for which the process survives is of the order $t^{2/3}$, and has approximately an exponential distribution.

\begin{proposition}\label{extinctiondist}
Suppose that for each $t > 0$, we have a deterministic initial configuration of particles $\nu_t$.  Suppose that, under $\P_{\nu_t}$, we have
\begin{equation}\label{maininitial}
\lim_{t \rightarrow \infty} Z_t(0) = 0, \hspace{.5in}\lim_{t \rightarrow \infty} \big(L_t(0) - R(0)\big) = \infty.
\end{equation}
Let $V$ have an exponential distribution with mean 1.  Then, under the conditional probability measures $\P_{\nu_t}( \: \cdot \:| \zeta > t)$, we have $t^{-2/3}(\zeta - t) \Rightarrow \frac{3}{c} V$ as $t \rightarrow \infty$.
\end{proposition}

We are also able to get rather precise information regarding what the configuration of particles looks like at or near time $t$, conditional on the process surviving until time $t$.  Recall the definitions of $M(s)$ and $R(s)$ from (\ref{MRdef}).  Kesten proved (see (1.12) of \cite{kesten}) that for fixed $x > 0$, there is a positive constant $K_2$ such that
\begin{equation}\label{KestenM}
\lim_{t \rightarrow \infty} \P_x(M(t) > e^{K_2 t^{2/9} (\log t)^{2/3}} \, | \, \zeta > t) = 0.
\end{equation}
Kesten also showed (see (1.11) of \cite{kesten}) that there is a positive constant $K_3$ such that
\begin{equation}\label{KestenR}
\lim_{t \rightarrow \infty} \P_x(R(t) > K_3 t^{2/9} (\log t)^{2/3} \,| \, \zeta > t) = 0.
\end{equation}

Theorem \ref{larges} below provides sharper results regarding the behavior of the number of particles in the system and the position of the right-most particle near time $t$, when the process is conditioned to survive until time $t$.  Note that the time $s$ depends on $t$.
%, and $\Rightarrow$ denotes convergence in distribution as $t \rightarrow \infty$.  Also, 
%$f(t) \sim g(t)$ means $\lim_{t \rightarrow \infty} f(t)/g(t) = 1$, and 
%$f(t) \ll g(t)$ means $\lim_{t \rightarrow \infty} f(t)/g(t) = 0$.

\begin{theorem}\label{larges}
Suppose that for each $t > 0$, we have a deterministic initial configuration of particles $\nu_t$ such that (\ref{maininitial}) holds under $\P_{\nu_t}$.
Suppose $s \in [0, t]$.  Let $V$ have an exponential distribution with mean 1.  Under the conditional probability measures $\P_{\nu_t}(\: \cdot\:|\,\zeta > t)$, the following hold:
\begin{enumerate}
\item If $t^{-2/3}(t - s) \to \sigma \geq 0$, then
%\begin{equation}\label{M1as}
%t^{-2/9} \log M(s) \Rightarrow c \Big(a + \frac{3V}{c}\Big)^{1/3}
%\end{equation}
%and
%\begin{equation}\label{R1as}
%t^{-2/9} R(s) \Rightarrow c \Big(a + \frac{3V}{c}\Big)^{1/3}.
%\end{equation}
%Moreover, if $M^*(s)$ and $R^*(s)$ denote the left-hand sides of (\ref{M1as}) and (\ref{R1as}) respectively, then we have the joint convergence
\begin{equation}\label{J1as}
\big(t^{-2/3}(\zeta - t), t^{-2/9} \log M(s), t^{-2/9} R(s) \big) \Rightarrow \bigg( \frac{3V}{c}, \: c \Big(\sigma + \frac{3V}{c}\Big)^{1/3}, \: c \Big(\sigma + \frac{3V}{c}\Big)^{1/3} \bigg).
\end{equation}

\item If $t^{2/3} \ll t - s \ll t$, then letting $a(s,t) = ((t-s)/t)^{2/3}$ and $b(s,t) = c(t-s)^{1/3} - \log(t-s)$, we have
%\begin{equation}\label{M2as}
%\bigg( \frac{t-s}{t} \bigg)^{2/3} \big(\log M(s) - c (t - s)^{1/3} + \log (t - s) \big) \Rightarrow V
%\end{equation}
%and
%\begin{equation}\label{R2as}
%\bigg( \frac{t-s}{t} \bigg)^{2/3} \big(R(s) - c (t - s)^{1/3} + \log (t - s) \big) \Rightarrow V.
%\end{equation}
%Moreover, if $M^*(s)$ and $R^*(s)$ denote the left-hand sides of (\ref{M2as}) and (\ref{R2as}) respectively, then we have the joint convergence
\begin{equation}\label{J2as}
\big(t^{-2/3}(\zeta - t), a(s,t)(\log M(s) - b(s,t)), a(s,t)(R(s) - b(s,t)) \big) \Rightarrow \bigg( \frac{3V}{c}, V, V \bigg).
\end{equation}
\end{enumerate}
\end{theorem}

Part 1 of Theorem \ref{larges} with $\sigma = 0$ implies that if $t - s \ll t^{2/3}$, and in particular if $s = t$, then conditional on $\zeta > t$, we have $t^{-2/9} \log M(s) \Rightarrow (3c^2)^{1/3} V^{1/3}$ and $t^{-2/9} R(s) \Rightarrow (3c^2)^{1/3} V^{1/3}$.  When we start with one particle at $x$, these results improve upon (\ref{KestenM}) and (\ref{KestenR}).  Note also that when $t^{2/3} \ll t - s \ll t/(\log t)^{3/2}$, the $\log(t - s)$ term can be dropped from $b(s,t)$.

\begin{remark}
The assumption in Proposition~\ref{extinctiondist} and Theorem~\ref{larges} that the initial configurations are deterministic is important.  Suppose we allow the $\nu_t$ to be random and assume, similar to part 1 of Theorem \ref{survival}, that under $\P_{\nu_t}$, we have $Z_t(0) \rightarrow_p 0$ and $L_t(0) - R(0) \rightarrow_p 0$.  To see that the conclusions of Proposition~\ref{extinctiondist} and Theorem~\ref{larges} can fail, consider the example in which $\nu_t$ consists of a single particle at 1 with probability $1 - 1/t$ and a single particle at $2L_t$ with probability $1/t$.  Conditional on the initial particle being at 1, the probability that the process survives until time $t$ is approximately $\alpha \pi e^{1 - ct^{1/3}}$ by (\ref{scsurvival}), while conditional on the initial particle being at $2L_t(0)$, the probability of survival until time $t$ is approximately $1 - q$.  Therefore, conditional on survival, with overwhelming probability the initial particle was at $2L_t(0)$, and on this event, the configuration of particles at time $t$ will be quite different from what is predicted by Theorem~\ref{larges}.
\end{remark}

\begin{remark}
It is possible to define the process conditioned to survive \emph{for all time}, through a certain \emph{spine decomposition}, which is classical for branching processes. One can easily convince oneself that in this process the number of particles at time $t$ is of the order $\exp(O(t^{1/2}))$, which is of a very different magnitude from the $\exp(O(t^{2/9}))$ obtained in the above theorems. This is in stark contrast to the classical case of (critical) Galton-Watson processes conditioned to survive until time $t$ or forever, where the number of particles at time $t$ is of the order of $t$ in both cases (see e.g.~\cite{lpp}); in fact, both are related through a certain change of measure.
%It might be interesting to continue this study with other branching Markov processes or more generally, other Markov processes in infinite-dimensional, unbounded space.
\end{remark}

\section{Tools, heuristics, and further results}

In this section, we describe some of the tools required to prove the main results stated in the introduction, along with the heuristics that allow us to see why these results are true.  We also state some further results (Theorems \ref{CSBPcond}, \ref{meds}, and \ref{condconfigprop}) which provide information about the behavior of the branching Brownian motion during the time interval $[\delta t, (1 - \delta)t]$, where $\delta > 0$ is small, conditioned on survival of the process until time $t$.

Theorems \ref{survival} and \ref{survivex} and Proposition \ref{extinctiondist} depend heavily on a connection between branching Brownian motion with absorption and continuous-state branching processes.  This connection is explained in Section~\ref{CSBPintro}, where Theorems~\ref{CSBPthm} and \ref{CSBPcond} are stated. To prepare for the proof of Theorem~\ref{larges}, we present in Section~\ref{sec:particle_configurations} a slight generalization of a result on particle configurations from \cite{bbs3}, which is Proposition~\ref{configpropnew}. We also state in that section two more results complementing Theorem~\ref{larges}, namely Theorems \ref{meds} and \ref{condconfigprop}. To be able to apply Proposition~\ref{configpropnew} for proving Theorem~\ref{larges}, we develop a method which allows us to predict the extinction time starting from an arbitrary initial configuration. This method is outlined in Section~\ref{sec:predicting}. Section~\ref{sec:descendants} recalls a result on the number of descendants of a single particle in branching Brownian motion with absorption, which is used in the proof of Theorem~\ref{CSBPthm}. Finally, Section~\ref{sec:organization} explains the organization of the rest of the paper.

\subsection{Connections with continuous-state branching processes}\label{CSBPintro}

The primary tool that allows us to improve upon previous results is a connection between branching Brownian motion with absorption and continuous-state branching processes.  This connection is a variation of a result of Berestycki, Berestycki, and Schweinsberg \cite{bbs2}, who considered branching Brownian motion with absorption in which the drift $\mu$ was slightly supercritical and was chosen so that the number of particles in the system remained approximately stable over the longest possible time.  They showed that under a suitable scaling, the number of particles in branching Brownian motion with absorption converges to a continuous-state branching process with jumps.  The intuition behind why we get a jump process in the limit is that, on rare occasions, a particle will move unusually far to the right.  Many descendants of this particle will then be able to survive, because they will avoid being absorbed at zero.  Such events can lead to a large rapid increase in the number of particles, and such events become instantaneous jumps in the limit as $t \rightarrow \infty$. To prove the main results of the present paper, we will need to establish a version of this result when $\mu = 1$, so that the drift is critical. 

\paragraph{Continuous-state branching processes.}
A continuous-state branching process is a $[0, \infty]$-valued Markov process $(\Xi(t), t \geq 0)$ whose transition functions satisfy the branching property $p_t(x + y, \: \cdot) = p_t(x, \: \cdot) * p_t(y, \: \cdot),$ which means that the sum of two independent copies of the process started from $x$ and $y$ has the same finite-dimensional distributions as the process started from $x + y$.  It is well-known that continuous-state branching processes can be characterized by their branching mechanism, which is a function $\Psi: [0, \infty) \rightarrow \R$.  If we exclude processes that can make an instantaneous jump to $\infty$, the function $\Psi$ is of the form $$\Psi(q) = \gamma q + \beta q^2 + \int_0^{\infty} (e^{-qx} - 1 + qx \Ind_{x \leq 1}) \: \nu(dx),$$ where $\gamma \in \R$, $\beta \geq 0$, and $\nu$ is a measure on $(0, \infty)$ satisfying $\int_0^{\infty} (1 \wedge x^2) \: \nu(dx) < \infty$.  If $(\Xi(t), t \geq 0)$ is a continuous-state branching process with branching mechanism $\Psi$, then for all $\lambda \geq 0$,
\begin{equation}\label{csbpLaplace}
E[e^{-\lambda \Xi(t)} \,| \, \Xi_0 = x] = e^{-x u_t(\lambda)},
\end{equation}
where $u_t(\lambda)$ can be obtained as the solution to the differential equation
\begin{equation}\label{diffeq}
\frac{\partial}{\partial t} u_t(\lambda) = -\Psi(u_t(\lambda)), \hspace{.5in} u_0(\lambda) = \lambda.
\end{equation}
We will be interested here in the case
\begin{equation}\label{Psidef}
\Psi(q) = \Psi_{a,b}(q) = aq + bq \log q
\end{equation}
for $a \in \R$ and $b > 0$. 
It is not difficult to solve (\ref{diffeq}) to obtain
\begin{equation}\label{utlambda}
u_t(\lambda) = \lambda^{e^{-bt}} e^{a(e^{-bt} - 1)/b}.
\end{equation} 
This process was first studied by Neveu \cite{nev92} when $a = 0$ and $b = 1$. It is therefore also called \emph{Neveu's continuous state branching process.}

\paragraph{Relation with branching Brownian motion.}

The following result is the starting point in the study of branching Brownian motion with absorption at critical drift. Note that in contrast to the case of weakly supercritical drift considered in \cite{bbs2}, a nonlinear time change appears here.

\begin{theorem}\label{CSBPthm}
Suppose that for each $t > 0$, we have a possibly random initial configuration of particles $\nu_t$.  Suppose that, under $\P_{\nu_t}$, we have $Z_t(0) \Rightarrow Z$ and $L_t(0) - R(0) \rightarrow_p \infty$ as $t \rightarrow \infty$.
%Set $a=\frac{2}{3}(a_{\ref{eq:abelian}}+\log \pi)+\frac{1}{2}$, where $a_{\ref{eq:abelian}}$ is the constant from Lemma~\ref{neveuW} below. Then
Then there exists $a \in \R$ such that
the finite-dimensional distributions of the processes $$(Z_t((1 - e^{-u}) t), u \geq 0),$$ under $\P_{\nu_t}$, converge as $t \rightarrow \infty$ to the finite-dimensional distributions of a continuous-state branching process $(\Xi(u), u \geq 0)$ with branching mechanism $\Psi_{a,2/3}(q) = aq + \frac{2}{3} q \log q$, whose distribution at time zero is the distribution of $Z$.
\end{theorem}

The strategy for proving Theorem~\ref{CSBPthm} will be similar to the one followed in \cite{bbs2}, but the proof is more involved due to the time inhomogeneity emerging in the analysis as a result of the non-linear time change. Yet, thanks to the introduction of several new ideas, we were able to significantly reduce the length of the proof. 

\begin{remark}
The constant $a$ in the statement of Theorem~\ref{CSBPthm} has the expression $a=\frac{2}{3}(a_{\ref{eq:abelian}}+\log \pi)+\frac{1}{2}$, with $a_{\ref{eq:abelian}}$ the constant defined in Lemma~\ref{neveuW} below. 
\end{remark}

\begin{remark} To understand the time change, let $s$ denote the original time scale on which the branching Brownian motion is defined, and let $u$ denote the new time parameter under which the process will converge to a continuous-state branching process.  From \cite{bbs2}, we know that the jumps in the process described above will happen at a rate proportional to $L_t(s)^{-3}$ or, equivalently, proportional to $(t - s)^{-1}$.  This corresponds to the time scaling by $(\log N)^3$ in \cite{bbs2}.  Therefore, to get a time-homogeneous limit, we need to set $du = (t - s)^{-1} \: ds.$  Integrating this equation gives
$$u = \int_0^u \: dv = \int_0^s (t - r)^{-1} \: dr = \log \bigg( \frac{t}{t-s} \bigg).$$  Rearranging, we get $s = (1 - e^{-u}) t$, which is the time change that appears in Theorem \ref{CSBPthm}.
\end{remark}

\paragraph{The probability of survival.}
Let $(\Xi(u), u \geq 0)$ be the continuous-state branching process that appears in Theorem \ref{CSBPthm}.  It follows from well-known criteria due to Grey \cite{grey} that $(\Xi(u), u \geq 0)$ neither goes extinct nor explodes in finite time.  That is, if $\Xi(0) \in (0, \infty)$, then $P(\Xi(u) \in (0, \infty) \mbox{ for all }u \geq 0) = 1$.
Let \begin{equation}\label{alphadef}
\alpha = e^{-3a/2}.
\end{equation}
The process $((e^{-\alpha \Xi(u)}), u \geq 0)$ is a martingale taking values in $(0, 1)$, as can be seen either by observing that $u_t(\alpha) = \alpha$ for all $t \geq 0$ and making a direct calculation using (\ref{csbpLaplace}), or by observing that $\Psi(\alpha) = 0$ and following the discussion on p.~716 of \cite{bfm08}.  By the Martingale Convergence Theorem, this martingale converges to a limit, and it is not difficult to see that the only possible values for the limit are $0$ and $1$.  Therefore, using $P_x$ to denote probabilities when $\Xi(0) = x$, we have, as noted in \cite{bfm08},
\begin{equation}\label{CSBPsurvival}
P_x \Big( \lim_{u \rightarrow \infty} \Xi(u) = \infty \Big) = 1 - e^{-\alpha x}, \hspace{.3in} P_x \Big( \lim_{u \rightarrow \infty} \Xi(u) = 0 \Big) = e^{-\alpha x}.
\end{equation}
As can be guessed from Theorem \ref{CSBPthm}, the event that $\lim_{u \rightarrow \infty} \Xi_u = \infty$ corresponds to the event that the branching Brownian motion survives until time $t$, and this correspondence leads to Theorem \ref{survival}. Note that the constant $\alpha$ in Theorem \ref{survival} and the constant $a$ in the definition of the continuous-state branching process in Theorem \ref{CSBPthm} are related by the formula (\ref{alphadef}).

\paragraph{Conditioning on survival.}
To make a connection between continuous-state branching processes and branching Brownian motion conditioned on survival until time $t$, we need to consider the continuous-state branching process conditioned to go to infinity.  Let $(\Xi(u), u \geq 0)$ be a continuous-state branching process with branching mechanism $\Psi(q) = aq + \frac{2}{3} q \log q$, started from $\Xi(0) = x$.  Bertoin, Fontbona, and Martinez \cite{bfm08} interpreted this process as describing a population in which a random number (possibly zero) of so-called prolific individuals have the property that their number of descendants in the population at time $u$ tends to infinity as $u \rightarrow \infty$.  The number $N$ of such prolific individuals at time zero has a Poisson distribution with mean $\alpha x$, which is consistent with Theorem \ref{survival}.  As noted in Section~3 of \cite{bfm08}, the branching property entails that $(\Xi(u), u \geq 0)$ can be decomposed as the sum of $N$ independent copies of a process $(\Phi(u), u \geq 0)$, which describes the number of descendants of a prolific individual, plus a copy of the original process conditioned to go to zero as $u \rightarrow \infty$, which accounts for the descendants of the non-prolific individuals.  Conditioning on the event $\lim_{u \rightarrow \infty} \Xi(u) = \infty$ is the same as conditioning on $N \geq 1$.  Furthermore, as $x \rightarrow 0$, the conditional probability that $N = 1$ given $N \geq 1$ tends to one.  Consequently, if we condition on $\lim_{u \rightarrow \infty} \Xi(u) = \infty$ and then let $x \rightarrow 0$, we obtain in the limit the process $(\Phi(u), u \geq 0)$.  Therefore, the process $(\Phi(u), u \geq 0)$ can be interpreted as the continuous-state branching process started from zero but conditioned to go to infinity as $u \rightarrow \infty$.  See \cite{bkm11, fm19} for further developments in this direction.  The following result, which we will deduce from Theorem~\ref{CSBPthm}, describes the finite-dimensional distributions of the branching Brownian motion with absorption, conditioned to survive for an unusually long time.

\begin{theorem}\label{CSBPcond}
Suppose that for each $t > 0$, we have a deterministic initial configuration of particles $\nu_t$ such that (\ref{maininitial}) holds under $\P_{\nu_t}$. Then the finite-dimensional distributions of $(Z_t((1 - e^{-u})t), u \geq 0),$ under the conditional probability measures $\P_{\nu_t}(\:\cdot\:|\,\zeta > t)$, converge as $t \rightarrow \infty$ to the finite-dimensional distributions of $(\Phi(u), u \geq 0)$.
\end{theorem}

\begin{remark}
Theorem \ref{CSBPcond} provides another way of understanding Proposition \ref{extinctiondist}.  It is known that
\begin{equation}\label{csbpasymp}
\lim_{u \rightarrow \infty} e^{-2u/3} \log \Xi(u) = - \log W \hspace{.2in} \mbox{a.s.},
\end{equation}
where $W$ has an exponential distribution with rate parameter $\alpha x$.  This result was stated for the case when the branching mechanism is $\Psi(q) = q \log q$ in \cite{nev92} by Neveu, who attributed the result as being essentially due to Grey \cite{grey77}.  A complete proof is given in Appendix A of \cite{fs04}, and by using (\ref{utlambda}), this proof can be adapted to give (\ref{csbpasymp}) when $\Psi(q) = aq + \frac{2}{3} q \log q$.  By conditioning on the event $\lim_{u \rightarrow \infty} \Xi(u) = \infty$, which is equivalent to conditioning on $-\log W > 0$, and then letting $x \rightarrow 0$, we obtain
\begin{equation}\label{CSBPV}
\lim_{u \rightarrow \infty} e^{-2u/3} \log \Phi(u) = V \hspace{.2in} \mbox{ a.s.},
\end{equation}
where $V$ has the exponential distribution with mean $1$.  This exponential limit law was derived also in Proposition 7 of \cite{fm19}. It turns out that the random variable $V$ in (\ref{CSBPV}) is the same random variable that appears in Proposition \ref{extinctiondist} and Theorem \ref{larges} above.  To see this, note that (\ref{CSBPV}) combined with Theorem \ref{CSBPcond} implies that when $u$ is large, we can write $Z_t((1 - e^{-u})t) \approx \exp(e^{2u/3} V)$.  Using the Taylor approximation $c(t + s)^{1/3} - ct^{1/3} \approx \frac{c}{3} s t^{-2/3}$ when $s \ll t$, we have
\begin{align}\label{CSBPheuristic}
Z_{t + vt^{2/3}}((1 - e^{-u})t) &\approx \exp \Big( e^{2u/3} V - L_{t + vt^{2/3}}((1 - e^{-u})t) + L_t((1 - e^{-u})t) \Big) \nonumber \\
&\approx \exp \Big( e^{2u/3} V - \frac{c}{3}(v t^{2/3}) (e^{-u} t)^{-2/3} \Big) \nonumber \\ 
&= \exp\Big( e^{2u/3} V - \frac{vc}{3} e^{2u/3} \Big).
\end{align}
The process should survive until approximately time $t + vt^{2/3}$, where $v$ is chosen so that $Z_{t + vt^{2/3}}((1 - e^{-u})t)$ is neither too close to zero nor too large.  This will happen when the expression inside the exponential in (\ref{CSBPheuristic}) is close to zero, which occurs when $v = \frac{3}{c}V$. That is, conditional on survival until at least time $t$, the process should survive for approximately time $t + \frac{3}{c} V t^{2/3}$, consistent with Proposition \ref{extinctiondist}.
\end{remark}

\subsection{Particle configurations}
\label{sec:particle_configurations}

After branching Brownian motion with absorption has been run for a sufficiently long time, the particles will settle into a fairly stable configuration.  Specifically, as long as $Z_t(s)$ is neither too small nor too large, the “density” of particles near $y$ at time $s$ is likely to be roughly proportional to
\begin{equation}\label{roughdensity}
\sin \bigg( \frac{\pi y}{L_t(s)} \bigg) e^{-y}.
\end{equation}
Berestycki, Berestycki, and Schweinsberg \cite{bbs3} obtained some results that made this idea precise, in the case of binary branching when the branching Brownian motion starts from a single particle that is far from the origin.  The proposition below extends the results in \cite{bbs3} to more general initial configurations and more general offspring distributions.

\begin{proposition}\label{configpropnew}
Consider a possibly random sequence of initial configurations $(\nu_n)_{n=1}^{\infty}$, along with possibly random times $(t_n)_{n=1}^{\infty}$, where $t_n$ may depend only on $\nu_n$ and $t_n \rightarrow_p \infty$ as $n \rightarrow \infty$.  Suppose that, under $\P_{\nu_n}$, the sequences $(Z_{t_n}(0))_{n=1}^{\infty}$ and $(Z_{t_n}(0)^{-1})_{n=1}^{\infty}$ are tight, and $L_{t_n}(0) - R(0) \rightarrow_p \infty$ as $n \rightarrow \infty$.  
Let $0 < \delta < 1/2$.  Then the following hold:
\begin{enumerate}
\item For all $\eps > 0$, there exist positive constants $C_3$ and $C_4$, depending on $\delta$ and $\eps$, such that if $\delta t_n \leq s \leq (1 - \delta)t_n$ and $n$ is sufficiently large, then
\begin{equation}\label{configconc1}
\P_{\nu_n} \bigg( \frac{C_3}{L_{t_n}(s)^3} e^{L_{t_n}(s)} \leq M(s) \leq \frac{C_4}{L_{t_n}(s)^3} e^{L_{t_n}(s)} \bigg) > 1 - \eps.
\end{equation}

\item For all $\eps > 0$, there exist positive constants $C_5$ and $C_6$, depending on $\delta$ and $\eps$, such that if $\delta t_n \leq s \leq (1 - \delta)t_n$ and $n$ is sufficiently large, then
\begin{equation}\label{configconc2}
\P_{\nu_n} \big( L_{t_n}(s) - \log t_n - C_5 \leq R(s) \leq L_{t_n}(s) - \log t_n + C_6 \big) > 1 - \eps.
\end{equation}

\item Let $N_{s,n}$ denote the set of particles alive at time $s$ for branching Brownian motion started from the initial configuration $\nu_n$.  Let $(s_n)_{n=1}^{\infty}$ be a sequence of times such that $\delta t_n \leq s_n \leq (1 - \delta) t_n$ for all $n$.  Define the probability measures $$\chi_n = \frac{1}{M(s_n)} \sum_{u \in N_{s_n,n}} \delta_{X_u(s_n)}, \hspace{.1in} \eta_n = \bigg( \sum_{u \in N_{s_n,n}} e^{X_u(s_n)} \bigg)^{-1} \sum_{u \in N_{s_n,n}} e^{X_u(s_n)} \delta_{X_u(s_n)/L_{t_n}(s_n)}.$$  Let $\mu$ be the probability measure on $(0, \infty)$ with density $g(y) = ye^{-y}$, and let $\xi$ be the probability measure on $(0,1)$ with density $h(y) = \frac{\pi}{2} \sin(\pi y)$.  Then $\chi_n \Rightarrow \mu$ and $\eta_n \Rightarrow \xi$ as $n \rightarrow \infty$, where $\Rightarrow$ denotes convergence in distribution for random elements in the Polish space of probability measures on $(0, \infty)$, endowed with the weak topology.
\end{enumerate}
\end{proposition}

\begin{remark}
Parts 1 and 2 of Proposition \ref{configpropnew} give estimates on the number of particles at time $s$ and the position of the right-most particle at time $s$.  Part 3 of Proposition \ref{configpropnew} states two limit theorems which together make precise the idea described in (\ref{roughdensity}).  If we choose a particle at random from the particles alive at time $s$, then most likely we will choose a particle near the origin.  Using the $\sin(x) \approx x$ approximation for small $x$, we get that the density of the position of this randomly chosen particle is approximately $g$.  If instead we choose a particle at random such that a particle at $y$ is chosen with probability proportional to $e^y$, and then we scale the location of the chosen particle such that the right-most particle is located near $1$, then the density of the chosen particle is approximately $h$. 
\end{remark}

\begin{remark}
Proposition \ref{configpropnew} also allows us to see why Theorem \ref{larges} should be true.  For simplicity, we focus on the case when $s = t$.  Consider a branching Brownian motion that has already survived for time $t$ and will ultimately survive until time $t + v$.  We expect $Z_{t+v}(t)$ to be neither too close to zero (in which case the process would most likely die out before time $t + v$) nor too large (in which case the process would most likely survive beyond time $t + v$).  Furthermore, because the process has evolved for a long time, we expect the density of particles at time $t$ to follow approximately (\ref{roughdensity}).  It follows that the position of the right-most particle at time $t$ should be close to $L_{t+v}(t) = cv^{1/3}$, while the number of particles at time $t$ should be within a constant multiple of $v^{-1} e^{cv^{1/3}}$.  The key to proving Theorem \ref{larges} is to argue that as long as $t - s \ll t$, the extinction time can be predicted fairly accurately from the configuration of particles at time $s$, so that we can apply Proposition \ref{configpropnew} with the predicted extinction time of the process in place of $t_n$.  Proposition \ref{extinctiondist} tells us that conditional on survival until time $t$, the amount of additional time for which the process survives can be approximated by $\frac{3}{c}V t^{2/3}$, where $V$ has an exponential distribution with mean one.  Therefore, using $\frac{3}{c}Vt^{2/3}$ in place of $v$, we expect $\log M(t) \approx R(t) \approx c(\frac{3}{c} V t^{2/3})^{1/3} = (3c^2 V)^{1/3} t^{2/9}$, consistent with Theorem \ref{larges}. 
\end{remark}

\paragraph{More results conditioned on survival.} The following two results complement Theorem~\ref{larges} and will be proved using the same methods, explained in Section~\ref{sec:predicting}. As in Theorem~\ref{larges}, the time $s$ depends on $t$. 

\begin{theorem}\label{meds}
Suppose that for each $t > 0$, we have a deterministic initial configuration of particles $\nu_t$ such that (\ref{maininitial}) holds under $\P_{\nu_t}$.  Let $0 < \delta < 1/2$, and suppose $s \in [\delta t, (1 - \delta) t]$.  For all $\eps > 0$, there exist positive constants $C_3$, $C_4$, $C_5$, and $C_6$ such that if $t$ is sufficiently large, then
$$\P_{\nu_t} \bigg( \frac{C_3}{L_t(s)^3} e^{L_t(s)} \leq M(s) \leq \frac{C_4}{L_t(s)^3} e^{L_t(s)} \,\Big|\, \zeta > t \bigg) > 1 - \eps$$ and $$\P_{\nu_t}\big(L_t(s) - \log t - C_5 \leq R(s) \leq L_t(s) - \log t + C_6 \,\big|\, \zeta > t\big) > 1 - \eps.$$ 
\end{theorem}

\begin{theorem}\label{condconfigprop}
Suppose that for each $t > 0$, we have a deterministic initial configuration of particles $\nu_t$ such that (\ref{maininitial}) holds under $\P_{\nu_t}$.
Suppose $s \in [0, t]$, and suppose $$\liminf_{t \rightarrow \infty} \frac{s}{t} > 0.$$  Define the probability measures $$\chi_s = \frac{1}{M(s)} \sum_{u \in N_s} \delta_{X_u(s)}, \hspace{.2in} \eta_s = \bigg( \sum_{u \in N_s} e^{X_u(s)} \bigg)^{-1} \sum_{u \in N_s} e^{X_u(s)} \delta_{X_u(s)/R(s)}.$$  Then, under the conditional probability measures $\P_{\nu_t}(\:\cdot\:|\,\zeta > t)$, we have $\chi_s \Rightarrow \mu$ and $\eta_s \Rightarrow \xi$ as $t \rightarrow \infty$, where $\mu$ and $\xi$ are defined as in Proposition \ref{configpropnew}.  If $\limsup_{t \rightarrow \infty} s/t < 1$, then we may replace $R(s)$ by $L_t(s)$ in the formula for $\eta_s$.
\end{theorem}

\subsection{Predicting the extinction time}
\label{sec:predicting}
Our strategy for proving Theorem \ref{larges} will be to use Proposition \ref{configpropnew} to deduce results about the configuration of particles at time $s$, where $t - s \ll t$, by allowing the configuration of particles at some time $r \leq s$ to play the role of the initial configuration of particles.  To do this, we will need to show that the configuration of particles at time $r$ satisfies the hypotheses of Proposition \ref{configpropnew}.  However, because the number of particles near time $t$ is highly variable, there is no deterministic choice of $t_n$ that will allow the tightness criterion in Proposition \ref{configpropnew} to be satisfied.

Consequently, we will develop a method for associating with an arbitrary configuration of particles a random time, which represents approximately how long the branching Brownian motion is likely to survive, starting from that configuration.  This technique may be of independent interest.  For all $s \geq 0$, let
\begin{equation}\label{Tdef}
T(s) = \inf\{t: L_{s+t}(s) \geq R(s) + 2 \mbox{ and }Z_{s+t}(s) \leq 1/2\}.
\end{equation}
For any fixed $s \geq 0$, as have $\lim_{t \rightarrow \infty} L_{s+t}(s) = \infty$, and for any fixed $s \geq 0$ and $x > 0$, we have $\lim_{t \rightarrow \infty} z_{s+t}(x,s) = 0$.  Therefore, $T(s)$ is well-defined and finite.  
The following result allows us to interpret $T(s)$ as being approximately the amount of additional time we expect the process to survive, given what the configuration of particles looks like at time $s$, provided that no particle at time $s$ is too close to $L_{T(s)}(0)$.

\begin{lemma}\label{tightextinct}
Let $\eps > 0$.  There exist positive constants $k'$, $t'$, and $a'$ such that for all initial configurations $\nu$ such that $T(0) \geq t'$ and $L_{T(0)}(0) - R(0) \geq a'$, we have $$\P_{\nu}(|\zeta - T(0)| \leq k' T(0)^{2/3}) > 1 - \eps.$$  
\end{lemma}

To apply Proposition \ref{configpropnew} to the configuration of particles at time $r$, we will need to know that with high probability, no particle at time $r$ is too close to $L_{T(r)}(0)$.  The key to this argument will be Lemma \ref{mainLRlemma}, which says that starting from any configuration of particles at time zero, there will typically be no particle close to this right boundary a short time later.

\begin{lemma}\label{mainLRlemma}
Let $\eps > 0$ and $A > 0$.  There exist positive real numbers $t_0 > 0$ and $d > 0$, depending on $\eps$ and $A$, such that if $\nu$ is any initial configuration of particles, then $$\P_{\nu}(\{R(d) \geq L_{T(d)}(0) - A\} \cap \{T(d) \geq t_0\}) < \eps.$$
\end{lemma}

\subsection{Descendants of a single particle}
\label{sec:descendants}

Recall that $(p_k)_{k=1}^{\infty}$ denotes the offspring distribution when a particle branches.  Let $L$ be a random variable such that $P(L = k) = p_k$. Recall that we suppose that $\E[L^2]<\infty$. Let $f(s) = \E[s^L]$ be the probability generating function of the offspring distribution, and let $q$ be the smallest root of $f(s) = s$, which is the extinction probability for a Galton-Watson process with offspring distribution $(p_k)_{k=1}^{\infty}$.  We record the following lemma, which is a consequence of results in Chapter 4 of \cite{thesis}.

\begin{lemma}\label{neveuW}
Suppose the branching Brownian motion is started with a single particle at zero, and there is no absorption at the origin.  For each $y \geq 0$, let $K(y)$ be the number of particles that reach $-y$ if particles are killed upon reaching $-y$.  Then there exists a random variable $W$ such that $$\lim_{y \rightarrow \infty} y e^{-y} K(y) = W \hspace{.2in}\textup{a.s.}$$  We have $\P(W > 0) = 1 - q$ and ${\bf E}[e^{-e^x W} ] = \psi(x)$, where $\psi$ is a solution to the equation $$\frac{1}{2} \psi'' - \psi' = \beta(\psi - f \circ \psi)$$ with $\lim_{x \rightarrow -\infty} \psi(x) = 1$, $\lim_{x \rightarrow \infty} \psi(x) = q$ and $1-\psi(-x)\sim xe^{-x}$ as $x\to\infty$. In fact, there exists $a_{\ref{eq:abelian}}\in\R$ such that as $\lambda \to 0$,
\begin{equation}
\label{eq:abelian}
 \E[e^{-\lambda W}] = \exp\left(\Psi_{a_{\ref{eq:abelian}},1}(\lambda) + o(\lambda)\right),
\end{equation}
where $\Psi_{a,b}(\lambda) = a\lambda + b\lambda \log \lambda$ is the function from \eqref{Psidef}.
\end{lemma}

In the case of binary branching, the existence of the random variable $W$ in Lemma~\ref{neveuW} goes back to the work of Neveu \cite{nev87}.  Proposition 4.1 in Chapter 2 of \cite{thesis} establishes that
\begin{equation}\label{Wasymp1}
\P(W > x) \sim \frac{1}{x} \hspace{.3in}\textup{as }x \rightarrow \infty
\end{equation}
and
\begin{equation}\label{Wasymp2}
\E[W \Ind_{\{W \leq x\}}] - \log x \rightarrow C \hspace{.3in}\textup{as }x \rightarrow \infty.
\end{equation}
The results (\ref{Wasymp1}) and (\ref{Wasymp2}) were proved earlier in \cite{bbs2} for binary branching.
As indicated in \cite{thesis}, the result (\ref{eq:abelian}) follows from (\ref{Wasymp1}) and (\ref{Wasymp2}) 
by de Haan's Tauberian Theorem (see Theorem 2 of \cite{dh76}).

\begin{remark}
Lemma \ref{neveuW} holds under weaker assumptions on the offspring distribution; see \cite{thesis}. Also, an analogous result for branching random walk has been proven recently in \cite{bim20}.
The random variable $W$ appearing in Lemma~\ref{neveuW} is equal to the limit of the so-called \emph{derivative martingale} \cite{nev87}, but we will not use this fact explicitly. 
\end{remark}

\subsection{Organization of the paper}
\label{sec:organization}

In Sections \ref{survivalsec} and \ref{18sec}, we prove the main results of the paper, assuming Theorem \ref{CSBPthm} and Proposition \ref{configpropnew}.  The most novel arguments in the paper are in these two sections.  In Section~\ref{survivalsec}, we prove Theorems \ref{survival}, \ref{survivex}, and \ref{extinctiondist}, all of which pertain to survival times for the process, as well as Theorem \ref{CSBPcond}, whose proof requires similar ideas.  In Section \ref{18sec}, we consider the process conditioned to survive until a large time $t$.  We prove Theorem \ref{larges} and Theorems \ref{meds} and \ref{condconfigprop}, along with Lemmas \ref{tightextinct} and \ref{mainLRlemma}.

The last four sections of the paper are devoted to the proofs of Theorem \ref{CSBPthm} and Proposition \ref{configpropnew}.  In Section \ref{momsec}, we establish some preliminary heat kernel and moment estimates that will be needed to prove those results.  In Section~\ref{configsec}, we show how to use results from \cite{bbs3} to deduce Proposition~\ref{configpropnew}.  Finally, Theorem \ref{CSBPthm} is proved in Sections \ref{CSBPsec} and \ref{sec:csbp_proof}.
%, using a strategy similar to the one followed in \cite{bbs2}.

\section{The probability of survival until time \texorpdfstring{$t$}{t}}\label{survivalsec}

Let $(\Xi(u), u \geq 0)$ denote a continuous-state branching process with branching mechanism $\Psi(q) = aq + \frac{2}{3} q \log q$, where $a$ is the constant from Theorem \ref{CSBPthm}.  Use $P_x$ and $E_x$ to denote probabilities and expectations for this process started from $\Xi(0) = x$.  Recall (\ref{CSBPsurvival}), and let ${\cal E}$ be the event that $\lim_{u \rightarrow \infty} \Xi(u) = 0$, so that
\begin{equation}\label{CSBPext}
P_x({\cal E}) = e^{-\alpha x},
\end{equation}
where $\alpha = \exp(-3a/2)$ as defined in (\ref{alphadef}).  Throughout this section, we also use the notation $$\phi_t(u) = (1 - e^{-u})t.$$

We begin with the following lemma, which can be deduced from (\ref{oldbound}) and gives an initial rough estimate of the survival probability.

\begin{lemma}\label{survivalgen}
There exist positive constants $C_2$ and $C_7$ such that for all $t > 0$ and all initial configurations $\nu$ such that $R(0) \leq L_t(0) - 1$, we have
\begin{equation}\label{roughsurvival}
1 - e^{-C_7 Z_t(0)} \leq \P_{\nu}(\zeta > t) \leq C_2 Z_t(0),
\end{equation}
and the lower bound holds even if the condition $R(0) \leq L_t(0) - 1$ is removed.
\end{lemma}

\begin{proof}
Recall that (\ref{oldbound}) implies that if $0 \leq x \leq L_t(0) - 1$, then
\begin{equation}\label{zsurvive}
C_1 z_t(x,0) \leq \P_x(\zeta > t) \leq C_2 z_t(x,0).
\end{equation}
One easily checks that there exists $C>0$ such that $z_t(x,0) \le C$ and $z_t(L_t(0)-1,0) \ge C^{-1}$ for $t$ sufficiently large. Furthermore, $\P_x(\zeta > t)$ is an increasing function of $x$. Hence, the lower bound in (\ref{zsurvive}) holds even if $x > L_t(0) - 1$, with the constant $C_1$ replaced by a different constant $C_7$.  Now consider a general initial configuration of particles $\nu$.  It follows from Boole's Inequality and (\ref{zsurvive}) that $$\P_{\nu}(\zeta > t) \leq \sum_{u \in N_0} \P_{X_u(0)}(\zeta > t) \leq C_2 Z_t(0),$$ which is the upper bound in (\ref{roughsurvival}).  To see the lower bound, note that by the inequality $1-x\le e^{-x}$ for $x\in [0,1]$,
\begin{align*}
\P_{\nu}(\zeta > t) 
&= 1 - \prod_{u \in N_0} (1 - \P_{X_u(0)}(\zeta > t)) \\
&\geq 1 - \exp \bigg( - \sum_{u \in N_0} \P_{X_u(0)}(\zeta > t) \bigg) \geq 1 - e^{-C_7Z_t(0)},
\end{align*}
as claimed.
\end{proof}

\begin{remark}
Once we prove Theorem \ref{survivex}, we will know that the condition $R(0) \leq L_t(0) - 1$ keeps the probabilities $\P_{X_u(0)}(\zeta > t)$ bounded away from one.  This means there is a positive constant $C$ for which $1 - \P_{X_u(0)}(\zeta > t) \geq \exp\left(-C \P_{X_u(0)}(\zeta > t)\right)$ for all $u \in N_0$.  Therefore, letting $C_8 = C C_2$, it will follow as in the above proof that
\begin{equation}\label{strongupper}
\P_{\nu}(\zeta > t) \leq 1 - \exp \bigg(- C \sum_{u \in N_0} \P_{X_u(0)}(\zeta > t) \bigg) \leq 1 - e^{-C_8 Z_t(0)}.
\end{equation}
This stronger form of the upper bound will be used in the proof of Lemma \ref{mainLRlemma} below.
\end{remark}

\begin{lemma}\label{zrare}
Suppose that, for each $t > 0$, we have a deterministic configuration of particles $\nu_t$.  Suppose that, under $\P_{\nu_t}$, we have $L_t(0) - R(0) \rightarrow \infty$ and $Z_t(0) \rightarrow z \in (0, \infty)$ as $t \rightarrow \infty$.  Let $\delta > 0$ and $r \in (0, 1)$.  Then there exist $\eps > 0$ and $y > 0$, depending on $\delta$ but not on $r$, such that for sufficiently large $t$, we have
\begin{equation}\label{smallzeq}
\P_{\nu_t}(\{\zeta > t\} \cap \{Z_t(rt) \leq \eps\}) < \delta
\end{equation}
and 
\begin{equation}\label{bigzeq}
\P_{\nu_t}(\{\zeta \leq t\} \cap \{Z_t(rt) \geq y\}) < \delta.
\end{equation}
\end{lemma}

\begin{proof}
Write $s = rt$, and let $A_{s,t}$ be the event that all particles at time $s$ are in the interval $[0, L_t(s) - 1]$.  By applying the Markov property at time $s$ along with the upper bound in Lemma \ref{survivalgen}, and noting that $L_t(s) = L_{t-s}(0)$, we get that on the event $A_{s,t}$, we have $\P_{\nu_t}(\zeta > t \,|\, {\cal F}_s) \leq C_2 Z_t(s)$.  Therefore,
$$\P_{\nu_t}(\{\zeta > t\} \cap \{Z_t(s) \leq \eps\} \cap A_{s,t}) \leq \P_{\nu_t}(\zeta > t \,|\, A_{s,t} \cap \{Z_t(s) \leq \eps\}) \leq C_2 \eps.$$  Also, it follows from the conclusion (\ref{configconc2}) of Proposition \ref{configpropnew} that $\P_{\nu_t}(A_{s,t}^c) < \delta/2$ for sufficiently large $t$.  The result (\ref{smallzeq}) follows by choosing $\eps < \delta/(2C_2)$.  Likewise, the lower bound in Lemma \ref{survivalgen}, in combination with the Markov property applied at time $s$, gives $\P_{\nu_t}(\zeta \leq t \,|\, {\cal F}_s) \leq e^{-C_7 Z_t(s)}$.  Therefore,
$$\P_{\nu_t}(\{\zeta \leq t\} \cap \{Z_t(s) \geq y\}) \leq \P_{\nu_t}(\zeta \leq t \,|\, Z_t(s) \geq y)\leq e^{-C_7 y},$$
and thus (\ref{bigzeq}) holds for sufficiently large $y$.
\end{proof}

\begin{lemma}\label{trilem}
Suppose that, for each $t > 0$, we have a deterministic configuration of particles $\nu_t$.  Suppose that, under $\P_{\nu_t}$, we have $L_t(0) - R(0) \rightarrow \infty$ and $Z_t(0) \rightarrow z \in (0, \infty)$ as $t \rightarrow \infty$.   Let $\delta > 0$.  There exist $\eps > 0$, $y > 0$, and $u_0 > 0$ such that for each fixed $u \geq u_0$, we have for sufficiently large $t$,
\begin{align*}
&\P_{\nu_t}(\{Z_t(\phi_t(u)) \leq \eps\} \: \triangle \: \{\zeta \leq t\}) < 3\delta \\
&\P_{\nu_t}(\{Z_t(\phi_t(u)) > y\} \: \triangle \: \{\zeta > t\}) < 3\delta \\
&P_z(\{\Xi(u) \leq \eps\} \: \triangle \: {\cal E}) < 3\delta \\
&P_z(\{\Xi(u) > y\} \: \triangle \: {\cal E}^c) < 3\delta
\end{align*}
where $\triangle$ denotes the symmetric difference between two events.
\end{lemma}

\begin{proof}
Choose $\eps > 0$ small enough that $P_{\eps}({\cal E}) \geq 1 - \delta$ and (\ref{smallzeq}) holds.  Choose $y > 0$ large enough that $P_y({\cal E}) \leq \delta$ and (\ref{bigzeq}) holds.  Fix $u_0$ large enough that $P_z(\eps < \Xi(u) \leq y) < \delta$ for $u \geq u_0$, which is possible because the limit in (\ref{CSBPsurvival}) exists.  By Theorem \ref{CSBPthm}, for $u \geq u_0$,
$$\lim_{t \rightarrow \infty} \P_{\nu_t}(\eps < Z_t(\phi_t(u)) \leq y) = P_z(\eps < \Xi(u) \leq y) < \delta.$$  The first two statements of the lemma follow from this result and Lemma \ref{zrare}.  Likewise, it follows from the Markov property of $(\Xi(u), u \geq 0)$ that $P_z(\{\Xi(u) \leq \eps\} \cap {\cal E}^c) \leq P_{\eps}({\cal E}^c) < \delta$ and $P_z(\{\Xi(u) > y\} \cap {\cal E}) \leq P_y({\cal E}) \leq \delta$.  The third and fourth statements of the lemma follow.
\end{proof}

\begin{proof}[Proof of Theorem \ref{survival}]
The proof is similar to the proof of Proposition 6 in \cite{bbs4}.  Suppose the initial configuration $\nu_t$ is deterministic, and, under $\P_{\nu_t}$, we have $Z_t(0) \rightarrow z \in (0, \infty)$ and $L_t(0) - R(0) \rightarrow \infty$ as $t \rightarrow \infty$.  Let $\delta > 0$.  Choose $\eps > 0$, $y > 0$, and $u_0 > 0$ as in Lemma \ref{trilem}.  By Theorem~\ref{CSBPthm}, for each fixed $u \geq u_0$, we have $$\lim_{t \rightarrow \infty} \P_{\nu_t}(Z_t(\phi_t(u)) \leq \eps) = P_z(\Xi(u) \leq \eps).$$  Therefore, using the first and third statements in Lemma \ref{trilem}, we obtain for each fixed $u \geq u_0$,
$$\limsup_{t \rightarrow \infty} |\P_{\nu_t}(\zeta \leq t) - P_z({\cal E})| \leq 6 \delta + \limsup_{t \rightarrow \infty} |\P_{\nu_t}(Z_t(\phi_t(u)) \leq \eps) - P_z(\Xi(u) \leq \eps)| = 6 \delta.$$  Since $\delta > 0$ was arbitrary, it follows that
\begin{equation}\label{zpositive}
\lim_{t \rightarrow \infty} \P_{\nu_t}(\zeta \leq t) = P_z({\cal E}) = e^{-\alpha z},
\end{equation}
which gives part~1 of Theorem \ref{survival} when each $\nu_t$ is deterministic and $z > 0$.

Next, suppose $\nu_t$ is deterministic and, under $\P_{\nu_t}$, we have $Z_t(0) \rightarrow 0$ and $L_t(0) - R(0) \rightarrow \infty$ as $t \rightarrow \infty$.  We may consider $t$ large enough that $0 < Z_t(0) < 1$.  Let $\nu_t^*$ denote the initial configuration with $\lfloor 1/Z_t(0) \rfloor$ particles at the location of each particle in the configuration $\nu_t$.  Then, adding a star to the notation when referring to the process started from $\nu_t^*$, we have $Z_t^*(0) \rightarrow 1$ as $t \rightarrow \infty$.  Also, we have $L_t^*(0) - R^*(0) \rightarrow \infty$.  Thus, we can apply (\ref{zpositive}) to get $$\lim_{t \rightarrow \infty} \P_{\nu_t^*}(\zeta \leq t) = e^{-\alpha}.$$  Because the process started from $\nu_t^*$ goes extinct by time $t$ if and only if each of the $\lfloor 1/Z_t(0) \rfloor$ independent copies of the process started from $\nu_t$ goes extinct by time $t$, we have $$\P_{\nu_t^*}(\zeta \leq t) = (1 - \P_{\nu_t}(\zeta > t))^{\lfloor 1/Z_t(0) \rfloor}.$$  It follows that $\P_{\nu_t}(\zeta > t) \sim \alpha Z_t(0)$, which establishes part~2 of Theorem \ref{survival}.  It follows that $\lim_{t \rightarrow \infty} \P_{\nu_t}(\zeta \leq t) = 1,$ so (\ref{zpositive}) also holds when $z = 0$.

It remains only to establish part~1 of Theorem~\ref{survival} when the initial configuration of particles may be random.  Consider an arbitrary subsequence of times $(t_n)_{n=1}^{\infty}$ tending to infinity.  Because, under $\P_{\nu_{t_n}}$, we have $Z_{t_n}(0) \Rightarrow Z$ and $L_{t_n}(0) - R(\infty) \rightarrow_p \infty$, we can use Skorohod's Representation Theorem to construct the sequence of random initial configurations $(\nu_{t_n})_{n=1}^{\infty}$ on one probability space $(\Omega, {\cal F}, \P)$ so that $Z_{t_n}(0) \rightarrow Z$ and $L_{t_n}(0) - R(0) \rightarrow \infty$ almost surely.  Then, for $\P$-almost every $\omega \in \Omega$, we can apply the result (\ref{zpositive}) for deterministic initial configurations to get $$\lim_{n \rightarrow \infty} \P_{\nu_{t_n}(\omega)}(\zeta \leq t) = e^{-\alpha Z(\omega)}.$$  Taking expectations of both sides and applying the Dominated Convergence Theorem gives $\lim_{n \rightarrow \infty} \P_{\nu_{t_n}}(\zeta \leq t) = \E[e^{-\alpha Z}]$, which implies part~1 of Theorem~\ref{survival}.
\end{proof}

\begin{proof}[Proof of Theorem \ref{survivex}]
The proof is similar to the proof of Theorem 1 in \cite{bbs4}.  Recalling Lemma~\ref{neveuW}, we first start a branching Brownian motion with a single particle at zero and stop particles when they reach $-y$.  Let $T_y$ be the time at which the last particle is killed at $-y$.  Let $g: (0, \infty) \rightarrow (0, \infty)$ be an increasing function such that
\begin{equation}\label{Tygy}
\lim_{y \rightarrow \infty} \P(T_y > g(y)) = 0.
\end{equation}
Fix $x \in \R$, and let $t \mapsto y(t)$ be an increasing function which tends to infinity slowly enough that the following three conditions hold:
\begin{equation}\label{ycond}
\lim_{t \rightarrow \infty} y(t) = \infty, \hspace{.5in} \lim_{t \rightarrow \infty} \frac{y(t)}{L_t(0)} = 0, \hspace{.5in} \lim_{t \rightarrow \infty} t^{-2/3} g(y(t)) = 0.
\end{equation}

Now we begin a branching Brownian motion with a single particle at $L_t(0) + x$.  Let $K_t$ denote the number of particles that reach $L_t(0) + x - y(t)$ before time $t$, if particles are stopped upon reaching this level.  For the process to go extinct before time $t$, the descendants of each of these $K_t$ particles must go extinct before time $t$.  
Let $w_1, \dots, w_{K_t}$ denote the times when these particles reach the level $L_t(0) + x - y(t)$.  Then, $$\P_{L_t(0) + x}(\zeta \leq t) \leq \E \bigg[ \prod_{i=1}^{K_t} \P_{L_t(0) + x - y(t)}(\zeta \leq t - w_i) \bigg].$$  Let $\nu_t$ denote the random configuration with $K_t$ particles at the position $L_t(0) + x - y(t)$.  Recall that $\P_{\nu_t}$ is an unconditional probability measure, and does not refer to conditional probability given the value of $\nu_t$.  Then for $t$ large enough that $g(y(t)) < t$,
\begin{equation}\label{extsqueeze}
\P_{\nu_t}(\zeta \leq t - g(y(t))) - \P(T_{y(t)} > g(y(t))) \leq \P_{L_t(0) + x}(\zeta \leq t) \leq  \P_{\nu_t}(\zeta \leq t).
\end{equation}

For the initial configuration $\nu_t$, we have
$$Z_t(0) = K_t L_t(0) \sin \bigg( \frac{\pi (L_t(0) + x - y(t))}{L_t(0)} \bigg) e^{x - y(t)}.$$
In view of the first two conditions in (\ref{ycond}), we have $$\sin \bigg( \frac{\pi (L_t(0) + x - y(t))}{L_t(0)} \bigg) \sim \frac{\pi y(t)}{L_t(0)},$$
where $\sim$ means that the ratio of the two sides tends to one as $t \rightarrow \infty$.
Also, by Lemma~\ref{neveuW}, the processes for all $t$ can be constructed on one probability space in such a way that
$y(t) e^{-y(t)} K_t \rightarrow W$ a.s., where $W$ is the random variable introduced in Lemma \ref{neveuW}.  Therefore, as $t \rightarrow \infty$, we have $$Z_t(0) \rightarrow \pi e^x W \hspace{.1in}\textup{a.s.}$$  Also, $L_t(0) - R(0) = y(t) - x \rightarrow \infty$ as $t \rightarrow \infty$.
Thus, by Theorem \ref{survival},
\begin{equation}\label{sq1}
\lim_{t \rightarrow \infty} \P_{\nu_t}(\zeta \leq t) = \E[e^{-\alpha \pi e^x W}].
\end{equation}

For the lower bound, let $t' = t - g(y(t))$.  By the third condition in (\ref{ycond}), we have $L_t(0) - L_{t'}(0) = c t^{1/3} - c (t - g(y(t)))^{1/3} \rightarrow 0$ as $t \rightarrow \infty$.  Therefore, by repeating the arguments above, we see that as $t \rightarrow \infty$, we have $Z_{t'}(0) \rightarrow \pi e^x W$ and $L_{t'}(0) - R(0) \rightarrow \infty$ almost surely.  Therefore,
\begin{equation}\label{sq2}
\lim_{t \rightarrow \infty} \P_{\nu_t}(\zeta \leq t - g(y(t))) = \E[e^{-\alpha \pi e^x W}].
\end{equation}
It follows from (\ref{Tygy}), (\ref{extsqueeze}), (\ref{sq1}), and (\ref{sq2}) that $$\lim_{t \rightarrow \infty} \P_{L_t(0) + x}(\zeta \leq t) = \E[e^{-\alpha \pi e^x W}],$$ which gives (\ref{convphi}).  Finally, if we define $\psi$ as in Lemma \ref{neveuW} and $\phi(x) = \E[e^{-\alpha \pi e^x W}]$, then $\phi(x) = \psi(x + \log(\alpha \pi))$, so the properties of $\phi$ claimed in the statement of the theorem follow from Lemma \ref{neveuW}.

To prove (\ref{convphi2}), write $t'' = t + v t^{2/3}$.  By differentiating, we get
\begin{equation}\label{tpt}
\lim_{t \rightarrow \infty} \big( L_{t''}(0) - L_t(0) \big) = \lim_{t \rightarrow \infty} \big( c(t + v t^{2/3})^{1/3} - c t^{1/3} \big) = \frac{cv}{3}.
\end{equation}
Using (\ref{convphi}), it follows that
$$\lim_{t \rightarrow \infty} \P_{L_t(0) + x}(\zeta \leq t'') = \lim_{t \rightarrow \infty} \P_{L_{t''}(0) + x - cv/3}(\zeta \leq t'') = \phi(x - cv/3),$$ 
as claimed.
\end{proof}

\begin{proof}[Proof of Proposition \ref{extinctiondist}]
Let $v > 0$.  By Theorem \ref{survival}, as $t \rightarrow \infty$ we have $$\P_{\nu_t}(\zeta > t + v t^{2/3}\,|\,\zeta > t) = \frac{\P_{\nu_t}(\zeta > t + v t^{2/3})}{\P_{\nu_t}(\zeta > t)} \sim \frac{Z_{t + vt^{2/3}}(0)}{Z_t(0)}.$$ 
Note that here both $Z_{t+vt^{2/3}}(0)$ and $Z_t(0)$ are being evaluated under the same initial measure $\P_{\nu_t}$.  Therefore, by (\ref{tpt}),
$$\lim_{t \rightarrow \infty} \frac{Z_{t + vt^{2/3}}(0)}{Z_t(0)} = \lim_{t \rightarrow \infty} e^{L_t(0) - L_{t + vt^{2/3}}(0)} = e^{-cv/3},$$
which gives the result.
\end{proof}

\begin{proof}[Proof of Theorem \ref{CSBPcond}]
We begin by following a similar strategy to the proof of part 2) of Theorem~\ref{survival}.  Let $z > 0$.  Let $\nu_t^*$ denote the initial configuration with $\lfloor z/Z_t(0) \rfloor$ particles at the location of each particle in the configuration $\nu_t$.  Adding the star to the notation when considering the process started from $\nu_t^*$, we have $Z_t^*(0) \rightarrow z$ and $L_t(0) - R^*(0) \rightarrow \infty$ as $t \rightarrow \infty$.  Equation (\ref{CSBPext}) and Theorem \ref{survival} give
\begin{equation}\label{starprob}
\lim_{t \rightarrow \infty} \P_{\nu^*_t}(\zeta > t) = 1 - e^{-\alpha z} = P_z({\cal E}^c).
\end{equation}
Also, by Theorem \ref{CSBPthm}, the finite-dimensional distributions of $(Z_t^*((1 - e^{-u})t), u \geq 0)$ converge as $t \rightarrow \infty$ to the finite-dimensional distributions of $(\Xi(u), u \geq 0)$ started from $\Xi_0 = z$.  

Fix $k \in \N$ and times $0 \leq u_1 < \dots < u_k$.  Let $\delta > 0$.  Choose $\eps > 0$, $y > 0$, and $u_0 > 0$ as in Lemma \ref{trilem}, and then fix $u \geq u_0$.  Let $g: \R^k \rightarrow \R$ be bounded and uniformly continuous, and let $h: \R^+ \rightarrow [0,1]$ be a continuous nondecreasing function such that $h(x) = 0$ if $x \leq \eps$ and $h(x) = 1$ if $x \geq y$. 
By the convergence result stated at the end of the previous paragraph,
\begin{equation}\label{ghlim}
\lim_{t \rightarrow \infty} \E_{\nu_t^*}[g(Z_t^*(\phi_t(u_1)), \dots, Z_t^*(\phi_t(u_k)) )h(Z_t^*(\phi_t(u)))] = E_z[g(\Xi(u_1), \dots, \Xi(u_k)) h(\Xi(u))].
\end{equation}
Lemma \ref{trilem} implies that for sufficiently large $t$, we have
\begin{equation}\label{happ1}
\P_{\nu_t^*}(h(Z_t^*(\phi_t(u))) \neq \Ind_{\{\zeta > t\}}) < 6 \delta
\end{equation}
and
\begin{equation}\label{happ2}
P_z(h(\Xi(u)) \neq \Ind_{{\cal E}^c}) < 6 \delta.
\end{equation}
By combining (\ref{starprob}), (\ref{ghlim}), (\ref{happ1}), and (\ref{happ2}), we get
\begin{equation}\label{condlim}
\lim_{t \rightarrow \infty} \frac{\E_{\nu_t^*}[g(Z_t^*(\phi_t(u_1)), \dots, Z_t^*(\phi_t(u_k))) \Ind_{\{\zeta > t\}}]}{\P_{\nu_t^*}(\zeta > t)} = \frac{E_{z}[g(\Xi(u_1), \dots, \Xi(u_k)) \Ind_{{\cal E}^c}]}{P_{z}({\cal E}^c)},
\end{equation}
which means the finite-dimensional distributions of $(Z_t^*((1 - e^{-u})t), u \geq 0)$ conditional on $\zeta > t$ converge as $t \rightarrow \infty$ to the finite-dimensional distributions of $(\Xi(u), u \geq 0)$ started from $\Xi(0) = z$ and conditioned to go to infinity.

We now take a limit as $z \rightarrow 0$.   We can write the branching Brownian motion started from $\nu_t^*$ as the sum of $\lfloor z/Z_t(0) \rfloor$ independent branching Brownian motions started from $\nu_t$.  Let $N_{t,z}$ denote the number of these independent branching Brownian motions that have a descendant alive at time $t$.  Conditioning on survival of the process until time $t$ is the same as conditioning on $N_{t,z} \geq 1$.  Therefore, the process conditioned on survival until time $t$ can be constructed by summing three processes, in the following way.
\begin{enumerate}
\item The first process is branching Brownian motion started from $\nu_t$ conditioned on survival until time $t$.

\item Choose a random variable $M_{t,z}$ whose distribution is the conditional distribution of $N_{t,z}$ given $N_{t,z} \geq 1$.  The second process is the sum of $M_{t,z} - 1$ independent branching Brownian motions started from $\nu_t$ conditioned on survival until time $t$.

\item The third process is the sum of $\lfloor z/Z_t(0) \rfloor - M_{t,z}$ independent branching Brownian motions conditioned to go extinct before time $t$.
\end{enumerate}
We will denote the contributions from these three processes by $Z_t^{(1)}$, $Z_t^{(2)}$, and $Z_t^{(3)}$ and let $Z_t' = Z_t^{(1)} + Z_t^{(2)} + Z_t^{(3)}$.  This means that the law of $(Z_t'(s), 0 \leq s < t)$ is the same as the conditional law of $(Z_t^*(s), 0 \leq s < t)$ given $\zeta > t$.  Therefore, for all $t \geq 0$, we have
\begin{align}\label{ZZprime}
&\E[g(Z_t^{(1)}(\phi_t(u_1)), \dots, Z_t^{(1)}(\phi_t(u_k)))] \nonumber \\
&\hspace{.4in}= \frac{\E_{\nu_t^*}[g(Z_t^*(\phi_t(u_1)), \dots, Z_t^*(\phi_t(u_k))) \Ind_{\{\zeta > t\}}]}{\P_{\nu_t^*}(\zeta > t)} \nonumber \\
&\hspace{.8in}+ \E\big[g(Z_t^{(1)}(\phi_t(u_1)), \dots, Z_t^{(1)}(\phi_t(u_k))) - g(Z_t'(\phi_t(u_1)), \dots, Z_t'(\phi_t(u_k)))\big]. 
\end{align}
Define $\|g\| = \sup_x |g(x)|$ and
$$w_g(\delta) = \sup\big\{|g(x_1, \dots, x_k) - g(y_1, \dots, y_k)|: |x_i - y_i| < \delta \mbox{ for all }i \in \{1, \dots, k\} \big\}.$$  Let $$p(z,t) = \P(Z_t^{(2)}(s) > 0 \mbox{ for some }s \geq 0)$$ and $$q(z,t,\delta) = \P(Z_t^{(3)}(\phi_t(u_i)) > \delta \mbox{ for some }i \in \{1, \dots, k\}).$$  Then, the absolute value of the second term on the right-hand side of (\ref{ZZprime}) is bounded above by
$$2 \|g\| (p(z,t) + q(z,t,\delta)) + w_g(\delta).$$
It is easy to see, for example by splitting the initial population into two groups of approximately equal size and applying (\ref{starprob}) with $z/2$ in place of $z$, that there is a positive constant $C$ such that for each $z > 0$, we have $\P_{\nu_t^*}(N_{t,z} \geq 2) \leq C z^2$ for sufficiently large $t$.  Therefore,
\begin{equation}\label{pzlim}
\lim_{z \rightarrow 0} \lim_{t \rightarrow \infty} p(z,t) = \lim_{z \rightarrow 0} \lim_{t \rightarrow \infty} \P_{\nu_t^*}(N_{t,z} \geq 2 \,|\, N_{t,z} \geq 1) = 0.
\end{equation}
By Theorem \ref{CSBPthm}, the finite-dimensional distributions of $(Z_t^{(3)}((1 - e^{-u})t), u \geq 0)$, if the process were not being conditioned to go extinct, would converge as $t \rightarrow \infty$ to the finite-dimensional distributions of $(\Xi(u), u \geq 0)$ started from $\Xi(0) = z$.  As $z \rightarrow 0$, the limiting extinction probability for the branching Brownian motion as $t \rightarrow \infty$ tends to one, while the process $(\Xi(u), u \geq 0)$ started from $\Xi(0) = z$ converges to the zero process.  These observations imply that for all $\delta > 0$, we have
\begin{equation}\label{qzlim}
\lim_{z \rightarrow 0} \lim_{t \rightarrow \infty} q(z,t,\delta) = 0.
\end{equation}
From (\ref{pzlim}), (\ref{qzlim}), and the fact that $w_g(\delta) \rightarrow 0$ as $\delta \rightarrow 0$ by the uniform continuity of $g$, we obtain
$$\lim_{z \rightarrow 0} \lim_{t \rightarrow \infty} \E\big[g(Z_t^{(1)}(\phi_t(u_1)), \dots, Z_t^{(1)}(\phi_t(u_k))) - g(Z_t'(\phi_t(u_1)), \dots, Z_t'(\phi_t(u_k)))\big] = 0.$$
Finally, as noted in Section \ref{CSBPintro}, the finite-dimensional distributions of $(\Xi(u), u \geq 0)$ started from $\Xi(0) = z$ and conditioned on ${\cal E}^c$ converge as $z \rightarrow 0$ to the finite-dimensional distributions of $(\Phi(u), u \geq 0)$.  Thus, by taking limits in (\ref{ZZprime}), observing that the left-hand side of (\ref{ZZprime}) does not depend on $z$, and applying (\ref{condlim}), we obtain
\begin{align*}
\lim_{t \rightarrow \infty} \E[g(Z_t^{(1)}(\phi_t(u_1)), \dots, Z_t^{(1)}(\phi_t(u_k)))] &= \lim_{z \rightarrow 0} \lim_{t \rightarrow \infty} \frac{\E_{\nu_t^*}[g(Z_t^*(\phi_t(u_1)), \dots, Z_t^*(\phi_t(u_k))) \Ind_{\{\zeta > t\}}]}{\P_{\nu_t^*}(\zeta > t)} \\
&= \lim_{z \rightarrow 0} \frac{E_{z}[g(\Xi(u_1), \dots, \Xi(u_k)) \Ind_{{\cal E}^c}]}{P_{z}({\cal E}^c)}  \\
&= E[g(\Phi(u_1), \dots, \Phi(u_k))].
\end{align*}
The result follows.
\end{proof}

\section{Conditioning on Survival}\label{18sec}

In this section, we prove our main results concerning the behavior of branching Brownian motion conditioned to survive for an unusually large time $t$, namely Theorem \ref{larges} and Theorems \ref{meds} and \ref{condconfigprop}.

We will frequently need estimates on $z_t(x,0)$.  Because $2x/\pi \leq \sin(x) = \sin(\pi - x) \leq x$ for all $x \in [0, \pi/2]$, we have
\begin{equation}\label{detzbound}
2 \min\{x, L_t(0) - x\} e^{x - L_t(0)} \leq z_t(x,0) \leq \pi \min\{x, L_t(0) - x\} e^{x - L_t(0)}
\end{equation}
for all $t > 0$ and $x \in [0, L_t(0)]$.

Recall the definition of $T(s)$ from (\ref{Tdef}).  The following result shows that $Z_{T(0)}(0)$ will be exactly $1/2$ as long as $T(0)$ is sufficiently large, and will allow us to prove Lemma \ref{tightextinct}.

\begin{lemma}\label{monotone}
Given any initial configuration of particles, the function $t \mapsto Z_t(0)$ is monotone decreasing on $\{t\ge0:L_t(0) \ge R(0) + 2\}$.  Also, there is a positive number $t^*$ such that if $T(0) \geq t^*$, then $T(0)$ is the unique positive real number $t$ such that $L_t(0) \geq R(0) + 2$ and $Z_t(0) = 1/2$.
\end{lemma}

\begin{proof}
To prove the first claim, note that $$\frac{d}{dL} L e^{x-L} \sin \bigg( \frac{\pi x}{L} \bigg) = e^{x-L} \bigg[ (1 - L) \sin \bigg( \frac{\pi x}{L} \bigg) - \frac{\pi x}{L} \cos \bigg( \frac{\pi x}{L} \bigg) \bigg].$$  If $0 \leq x < L/2$, then both terms inside the brackets are negative when $L > 1$.  Suppose instead $L/2 \leq x \leq L - 2$.  Then $\sin(\pi x/L) \geq \sin(2 \pi/L) \geq 4/L$, so
$$\frac{d}{dL} L e^{x-L} \sin \bigg( \frac{\pi x}{L} \bigg) \leq e^{x-L} \bigg( \frac{4 (1 - L)}{L} + \frac{(L-2) \pi}{L} \bigg) < 0.$$
It follows that $t \mapsto Z_t(0)$ is a monotone decreasing function on $\{t\ge0:L_t(0) \ge R(0) + 2\}$.  Therefore, either $T(0) = R(0) + 2$, or $T(0)$ is the unique positive real number $t$ such that $L_t(0) \geq R(0) + 2$ and $Z_t(0) = 1/2$.  Because $\lim_{t \rightarrow \infty} z_t(L_t(0) - 2, 0) = 2 \pi/e^2 > 1/2$, the first possibility can be ruled out if $T(0)$ is sufficiently large, which completes the proof of the lemma.
\end{proof}

\begin{proof}[Proof of Lemma \ref{tightextinct}]
It suffices to show that for any deterministic sequence of initial configurations $(\nu_n)_{n=1}^{\infty}$ such that $T(0) \rightarrow \infty$ and $L_{T(0)}(0) - R(0) \rightarrow \infty$ as $n \rightarrow \infty$, we have
\begin{align}
\lim_{k \rightarrow \infty} \limsup_{n \rightarrow \infty} \P_{\nu_n}(\zeta \leq T(0) - kT(0)^{2/3}) = 0, \label{limk1} \\
\lim_{k \rightarrow \infty} \liminf_{n \rightarrow \infty} \P_{\nu_n}(\zeta \leq T(0) + k T(0)^{2/3}) = 1. \label{limk2}
\end{align}
For $k \geq 0$, let $t_n$, $t_n^-(k)$ and $t_n^+(k)$ denote the values of $T(0)$, $T(0) - kT(0)^{2/3}$ and $T(0) + kT(0)^{2/3}$ respectively under $\P_{\nu_n}$. 
Recall by (\ref{tpt}) that for every fixed $k$,
\begin{equation}
\label{tptadapted}
L_{t_n^-(k)}(0) = L_{t_n}(0) + O(1) = L_{t_n^+(k)}(0).
\end{equation}
Furthermore, by Lemma \ref{monotone}, we have $Z_{T(0)}(0) = 1/2$ under $\P_{\nu_n}$ for sufficiently large $n$.  If $(x_n)_{n=1}^{\infty}$ is a sequence of positive numbers for which $L_{t_n}(0) - x_n \rightarrow \infty$, then using \eqref{tptadapted},
$$\lim_{n \rightarrow \infty} \frac{z_{t_n^-(k)}(x_n, 0)}{z_{t_n}(x_n, 0)} = \lim_{n \rightarrow \infty} \frac{L_{t_n^-(k)}(0) \sin\big(\frac{\pi x_n}{L_{t_n^-(k)}(0)}\big) e^{x_n - L_{t_n^-(k)}(0)}}{L_{t_n}(0) \sin\big(\frac{\pi x_n}{L_{t_n}(0)}\big) e^{x_n - L_{t_n}(0)}} = \lim_{n \rightarrow \infty} e^{L_{t_n}(0) - L_{t_n^-(k)}(0)} = e^{ck/3}.$$
From this calculation, and a similar calculation with $t_n^+(k)$ in place of $t_n^-(k)$, it follows that
\begin{equation}\label{zchange}
\lim_{n \rightarrow \infty} Z_{t_n^-(k)}(0) = \frac{e^{ck/3}}{2}, \hspace{.5in}\lim_{n \rightarrow \infty} Z_{t_n^+(k)}(0) = \frac{e^{-ck/3}}{2}.
\end{equation}
Because $L_{t_n}(0) - R(0) \rightarrow \infty$ 
%and $L_{t_n^+(k)}(0) \rightarrow \infty$ as $n \rightarrow \infty$, 
it now follows from Theorem~\ref{survival} that
$$\lim_{n \rightarrow \infty} \P_{\nu_n}(\zeta \leq t_n^-) = e^{-(\alpha/2) e^{ck/3}}, \hspace{.5in} \lim_{n \rightarrow \infty} \P_{\nu_n}(\zeta \leq t_n^+) = e^{-(\alpha/2) e^{-ck/3}},$$
which imply (\ref{limk1}) and (\ref{limk2}).
\end{proof}

\begin{lemma}\label{smallTlem}
Let $\eps > 0$ and $K > 0$.  Then there exists $t > 0$, depending on $\eps$ and $K$, such that for all initial configurations $\nu$ for which $T(0) \leq K$ under $\P_\nu$, we have $\P_{\nu}(\zeta > t) < \eps$.
\end{lemma}

\begin{proof}
Let $u \leq K \leq t$.
It follows from (\ref{detzbound}) that if $0 \leq x \leq \min\{L_u(0), L_t(0)\}$, then
$$\frac{z_t(x,0)}{z_u(x,0)} = \frac{L_t(0) \sin(\frac{\pi x}{L_t(0)}) e^{-L_t(0)}}{L_u(0) \sin(\frac{\pi x}{L_u(0)}) e^{-L_u(0)}} \leq \frac{\pi}{2} \cdot \frac{\min\{x, L_t(0) - x\}}{\min\{x, L_u(0) - x\}} \cdot e^{L_u(0) - L_t(0)}.$$  Consequently, if $x \leq L_u(0) - 2$, then
\begin{equation}\label{KZratio}
\frac{z_t(x,0)}{z_u(x,0)} \leq \frac{\pi}{2} \cdot \frac{L_t(0)}{2} \cdot e^{L_u(0) - L_t(0)} \leq \frac{\pi e^K}{4} \cdot L_t(0) e^{-L_t(0)}.
\end{equation}
By the definition of $T(0)$, we have $Z_{T(0)}(0) \leq 1/2$ and $R(0) \leq T(0) - 2$.  Therefore, we can choose $t$ sufficiently large that for all initial configurations $\nu$ for which $T(0) \leq K$ under $\P_\nu$, we have
$$Z_t(0) \leq Z_{T(0)}(0) \cdot \frac{\pi e^K}{4} \cdot L_t(0) e^{-L_t(0)} \leq \frac{\pi e^K}{8} \cdot L_t(0) e^{-L_t(0)} < \frac{\eps}{C_2},$$ with $C_2$ the constant from  Lemma \ref{survivalgen}. It follows from that lemma that the probability of survival until time $t$ is bounded above by $\eps$, as claimed.
\end{proof}

We now work towards the proof of Lemma \ref{mainLRlemma}.  To prepare for this proof, we record some bounds on the position of the right-most particle $R(t)$ in branching Brownian motion with absorption.  For branching Brownian motion without absorption, Bramson \cite{bram83} considered this problem when $q = 0$.  He showed that if $m_x(t)$ denotes the median of the distribution of $R(t)$ when we start with a single particle at $x$, then there is a positive constant $C$ such that for all $t \geq 1$, we have
\begin{equation}\label{mainBram}
\bigg|m_x(t) - \bigg(x - \frac{3}{2} \log t \bigg) \bigg| \leq C.
\end{equation}
Bramson also showed (see equation (8.17) of \cite{bram83}) that there is another positive constant $C'$ such that for all $x \in \R$, $t \geq 1$, and $y \geq 1$, we have $\P_x(R(t) > m_x(t) + y) \leq C' y e^{-y}$.
Combining this result with (\ref{mainBram}) and noting that absorption at zero can only reduce the likelihood that there is a particle above a certain level at time $t$, we get that for branching Brownian motion with absorption, there is a positive constant $C''$ such that for all $x > 0$, $t \geq 1$, and $y \geq 1$, we have
\begin{equation}\label{Bramabs}
\P_x\bigg(R(t) > x - \frac{3}{2} \log t + y \bigg) \leq C'' y e^{-y}.
\end{equation}

We now claim that (\ref{Bramabs}) holds even when $q > 0$.  To see this, we construct the branching Brownian motion process in the following way.  First, we define a branching Brownian motion process with no killing at the origin.  If we ignore the spatial positions of the particles, this process is simply a continuous-time Galton-Watson process.  Next, we color particles red if they have an infinite line of descent, and blue if all of their descendants eventually die out.  It follows from results in \cite{gr} that the red particles form a continuous-time Galton-Watson process in which the offspring distribution still has finite variance but particles can never die. Furthermore, this process has the same growth rate as the original process. After coloring the particles red and blue, we again consider the spatial motion, which is independent of the branching structure, and add the killing at the origin by truncating paths once they hit the origin.  Now the red particles form a branching Brownian motion whose offspring distribution satisfies $q = 0$, and so (\ref{Bramabs}) holds.  Because, conditional on the configuration of particles at time $t$, each particle is red with probability $1-q$ and blue with probability $q$, the result (\ref{Bramabs}) must also hold for the original process that includes particles of both colors, after dividing the constant by $1 - q$.

We will also need an alternative bound when $x$ is small that allows us to take the absorption into account.  For this, let $$V(s) = \sum_{u \in N_s} X_u(s) e^{X_u(s)}.$$  It is well-known (see, for example, Lemma 2 in \cite{hh07}) that $(V(s), s \geq 0)$ is a nonnegative martingale, and its value is at least $ye^y$ when there is a particle above $y$.  It follows from Markov's Inequality that
\begin{equation}\label{derMG}
\P_x(R(t) > y) \leq \P_x(V(s) \ge ye^y) \leq \frac{x}{y} e^{x - y}.
\end{equation}

\begin{proof}[Proof of Lemma \ref{mainLRlemma}]
Consider the set $N_0$ of particles at time zero.  Rank the particles $u_1, u_2, \dots$ in decreasing order by position, so that $X_{u_1}(0) \geq X_{u_2}(0) \geq \dots$.  Now construct an extension of the process in which the absorption is suppressed, so that the trajectories of particles continue past the origin.  Let $G$ be the smallest integer $g$ such that the particle $u_g$ has descendants alive at time $d$ in this extended process.  
Note that if $q_d$ denotes the probability that a Galton-Watson process with offspring distribution $(p_k)_{k=0}^{\infty}$ dies before time $d$, then $\P(G = k|\#N_0 \geq k) = q_d^{k-1}(1 - q_d)$.
Let $\nu^*$ denote the initial configuration consisting of the particles $u_i$ with $i \geq G$.  Let ${\cal F}_0^*$ denote the $\sigma$-field generated by $N_0$ and $G$.  Note that, conditional on ${\cal F}_0^*$, the descendants of the particles $u_i$ for $i \geq G+1$ behave as they would in the original branching Brownian motion process, while the descendants of the particle $u_G$ are conditioned to survive until time $d$ in the extended process.

Let $T^*(0)$ be defined as in (\ref{Tdef}) for the configuration $\nu^*$. 
We will show that given $0 < \eps < 1$ and $A > 0$, we can choose $d$ sufficiently large and then $t_0$ sufficiently large that
\begin{equation}\label{Rd}
\P_{\nu}(R(d) \geq L_{T^*(0)}(0) - 2A) < \frac{\eps}{2}
\end{equation}
and
\begin{equation}\label{needL1}
\P_{\nu}\Big(\{L_{T(d)}(0) \leq L_{T^*(0)}(0) - A\} \cap \{T(d) \geq t_0\} \Big) < \frac{\eps}{2}.
\end{equation}
These two results immediately imply the statement of the lemma.

We first show that (\ref{Rd}) holds if $d$ is sufficiently large.  Let $N^*_0 = N_0 \setminus \{u_1, \dots, u_{G-1}\},$ and let $N^*_s$ denote the set of descendants of these particles alive at time $s$.  Let $$\alpha = \frac{e^{2A}}{\eps (1-q)}.$$  Let $S_1 = \{u \in N^*_0: L_{T^*(0)}(0) - X_u(0) \geq \alpha\}$ and $S_2 = N^*_0 \setminus S_1$.  Let $Z_t^*(s)$ be defined as in (\ref{Zdef}), but summing only over particles in $N^*_s$.  To bound the probability that some particle in $N_0^*$ has a descendant above $L_{T^*(0)} - 2A$ at time $d$, we apply (\ref{derMG}) to particles in $S_1$ and (\ref{Bramabs}) to particles in $S_2$.  The behavior of the descendants of the particle $u_G$ is affected by conditioning.  However, because the probability that a continuous-time Galton-Watson process with branching rate $\beta$ and offspring distribution $(p_k)_{k=1}^{\infty}$ survives until time $d$ is greater than $1 - q$, we can apply the results (\ref{Bramabs}) and (\ref{derMG}) to all particles in our process if we divide the upper bounds there by $1 - q$.  

Consider first the particles in $S_2$.  Assume for now that $L_{T^*(0)}(0) \geq 2 \alpha$, so that all particles in $S_2$ are above $\frac{1}{2} L_{T^*(0)}(0)$.  Using that $X_{u_G}(0) \leq L_{T^*(0)}(0) - 2$ by (\ref{Tdef}) as well as the lower bound in (\ref{detzbound}), we get $z_{T^*(0)}(X_u(0), 0) \geq 4 e^{-\alpha}$ for all $u \in S_2$.  Because $Z^*_{T^*(0)}(0) \leq 1/2$, it follows that there can be at most $e^{\alpha}/8$ particles in $S_2$.  In view of (\ref{Bramabs}), the probability that one of these particles has a descendant above $L_{T^*(0)}(0) - 2A$ at time $d$ tends to zero as $d \rightarrow \infty$.  Therefore, given $\eps$ and $A$, we can choose $d$ large enough to keep this probability below $\eps/4$.  Using also (\ref{derMG}) to handle the particles in $S_1$, we get that on $\{L_{T^*(0)}(0) \geq 2 \alpha\}$,
$$\P_{\nu}(R(d) \geq L_{T^*(0)}(0) - 2A\,|\,{\cal F}^*_0) < \frac{\eps}{4} + \frac{1}{1-q} \sum_{u \in S_1} \frac{X_u(0)  e^{X_u(0) - L_{T^*(0)}(0) + 2A}}{L_{T^*(0)}(0) - 2A}.$$
The lower bound in (\ref{detzbound}), applied separately when $x \leq \frac{1}{2} L_{T^*(0)}(0)$ and $x > \frac{1}{2}L_{T^*(0)}(0)$, yields
$$\sum_{u \in S_1} \frac{X_u(0)  e^{X_u(0) - L_{T^*(0)}(0) + 2A}}{L_{T^*(0)}(0) - 2A} \leq \frac{e^{2A}}{2} \sum_{u \in S_1} \frac{z_{T^*(0)}(X_u(0), 0)}{L_{T^*(0)}(0) - 2A} \max \bigg\{1, \frac{X_u(0)}{L_{T^*(0)}(0) - X_u(0)} \bigg\}.$$
Recall that $L_{T^*(0)}(0) - X_u(0) \geq \alpha$ for all $u \in S_1$, and therefore using that $\alpha \geq 2A$, we also have $X_u(0) \leq L_{T^*(0)}(0) - 2A$ for all $u \in S_1$ and $L_{T^*(0)}(0) - 2A \geq \alpha$ on the event $\{L_{T^*(0)}(0) \geq 2 \alpha\}$.  It follows that for all $u \in S_1$, we have
$$\frac{1}{L_{T^*(0)}(0) - 2A} \max \bigg\{1, \frac{X_u(0)}{L_{T^*(0)}(0) - X_u(0)} \bigg\} \leq \frac{1}{\alpha}.$$
Therefore,
$$\frac{1}{1-q} \sum_{u \in S_1} \frac{X_u(0)  e^{X_u(0) - L_{T^*(0)}(0) + 2A}}{L_{T^*(0)}(0) - 2A} \leq \frac{\eps Z^*_{T^*(0)}(0)}{2} \leq \frac{\eps}{4},$$ 
and thus $$\P_{\nu}(R(d) \geq L_{T^*(0)}(0) - 2A\,|\, {\cal F}^*_0) < \frac{\eps}{2}$$ on the event $\{L_{T^*(0)}(0) \geq 2 \alpha\}$.  Lemma \ref{smallTlem} implies that we can choose $d$ large enough that $\P_{\nu}(R(d) \geq L_{T^*(0)}(0) - 2A\,|\, {\cal F}^*_0) \leq \P_{\nu}(\zeta > d\,|\,{\cal F}_0^*) < \eps/2$ on the event $\{L_{T^*(0)}(0) < 2 \alpha\}$.  It follows that (\ref{Rd}) holds, when $d$ is chosen to be sufficiently large.

It remains to establish (\ref{needL1}).  Choose $\delta > 0$ small enough that
\begin{equation}\label{deltadef}
\frac{2 \delta e^{C_8/2}}{1 - q + \delta} < \frac{\eps}{2},
\end{equation}
where $C_8$ is the constant from (\ref{strongupper}).  Let $k'$, $t'$, and $a'$ be the constants from Lemma \ref{tightextinct} with $\delta$ in place of $\eps$.  Choose $a_0$ large enough that $a_0 \geq a'$, $a_0 > ck'/6$, and $\phi(a_0) \leq q + \delta/2$, where $\phi$ is the function from Theorem \ref{survivex}.  We will assume that $A \geq 2a_0$, which can be done because the statement of the lemma is weaker when $A < 2a_0$.  Next, choose $d$ large enough that (\ref{Rd}) holds, and large enough that the probability that a continuous-time Galton-Watson process with branching rate $\beta$ and offspring distribution $(p_k)_{k=1}^{\infty}$ survives until time $d$ is at most $1 - q + \delta$.  Finally, choose $t_0 > 0$ large enough that the following hold:
\begin{enumerate}
\item We have $t_0 \geq t'$.

\item We have $t - k't^{2/3} \geq d$ for all $t \geq t_0$.

\item If $x \geq a_0$ and $t \geq t_0$, then $\P_{L_t(0) + x}(\zeta \geq t + d) \geq 1 - q - \delta$.  Note that Theorem \ref{survivex} and our assumption that $\phi(a_0) \leq q + \delta/2$ imply that $t_0$ can be chosen this way.

\item If $t \geq t_0$, then $ct^{1/3} - 2a_0 \leq c(t - k' t^{2/3} - d)^{1/3}$.  Note that this is possible because $ct^{1/3} - c(t - k't^{2/3} - d)^{1/3} \sim ck'/3$ as $t \rightarrow \infty$, and $a_0 > ck'/6$.
\end{enumerate}

Let $T'$ be the time such that $L_{T'}(0) = L_{T^*(0)}(0) - A$.  Our strategy will be to show that with high probability, the process will survive until time $T' + d$, which will preclude $T(d)$ from being too small.  In particular, we claim that on $\{T' \geq t_0\}$, we have 
\begin{equation}\label{Bsurvival}
\P_{\nu}(\zeta > T' + d\,|\,{\cal F}_0^*) \geq \frac{1 - q - \delta}{1 - q + \delta}.
\end{equation}
Assume for now that (\ref{Bsurvival}) holds.  It follows that
\begin{equation}\label{intzeta1}
\P_{\nu}(\{\zeta \leq T' + d\} \cap \{T' \geq t_0\}) \leq \frac{2 \delta}{1 - q + \delta}.
\end{equation}
Because $Z_{d + T(d)}(d) \leq 1/2$ and $R(d) \leq L_{T(d)}(0) - 2$ by definition, it follows from (\ref{strongupper}) that $$\P_{\nu}(\zeta \leq T' + d\,|\,t_0 \leq T(d) \leq T') \geq \P_{\nu}(\zeta \leq T(d) + d\,|\,t_0 \leq T(d) \leq T') \geq e^{-C_8/2},$$ and therefore
\begin{align}\label{intzeta2}
\P_{\nu}(\{\zeta \leq T' + d\} \cap \{T' \geq t_0\}) &\geq \P_{\nu}(\{\zeta \leq T' + d\} \cap \{t_0 \leq T(d) \leq T'\}) \nonumber \\
&= \P_{\nu}(t_0 \leq T(d) \leq T') \P_{\nu}(\zeta \leq T' + d \,|\, t_0 \leq T(d) \leq T') \nonumber \\
&\geq e^{-C_8/2} \P_{\nu}(t_0 \leq T(d) \leq T').
\end{align}
From (\ref{intzeta1}), (\ref{intzeta2}), and (\ref{deltadef}), we get
$$\P_{\nu}(t_0 \leq T(d) \leq T') \leq \frac{2 \delta e^{C_8/2}}{1 - q + \delta} < \frac{\eps}{2},$$
which by the definition of $T'$ is precisely (\ref{needL1}).

It remains to prove (\ref{Bsurvival}).
Let $B = \{X_{u_G}(0) \geq L_{T^*(0)}(0) - A/2\} \in {\cal F}_0^*$.  On the event $B$, the particle $u_G$ begins above $L_{T'}(0) + A/2$.
Our choices of $a_0$ and $t_0$ ensure that as long as $A \geq 2 a_0$ and $T' \geq t_0$, the probability that a particle started at the position $X_{u_G(0)}$ has descendants alive at time $T' + d$ is at least $1 - q - \delta$.  Also, our choice of $d$ ensures that the probability that, without absorption at zero, such a particle would have descendants alive until time $d$ is at most $1 - q + \delta$.  Because our definition of $G$ entails conditioning on the latter event, and because the presence of other particles in the initial configuration can only increase the probability that the process survives beyond time $T' + d$, the inequality (\ref{Bsurvival}) holds on the event $B \cap \{T' \geq t_0\}$.
On the event $B^c$, the configuration $\nu^*$ has no particles above $L_{T^*(0)}(0) - A/2$.  Then we can apply Lemma~\ref{tightextinct}, which implies that on the event $B^c \cap \{T^*(0) \geq t_0\}$ we have
\begin{equation}\label{Bcpre}
\P_{\nu}(\zeta \geq T^*(0) - k' T^*(0)^{2/3} \,|\, {\cal F}_0^*) > 1 - \delta.
\end{equation}
Note that this result holds even though, as noted at the beginning of the proof, conditioning on ${\cal F}_0^*$ means the descendants of the particle at $u_G$ are conditioned to survive until time $d$ in the extended process.  This conditioning can only increase the chance that descendants of the particle at $u_G$ survive beyond time $T^*(0) - k' T^*(0)^{2/3}$ because, by our choice of $t_0$, particles can not survive this long if they die out before time $d$ even in the extended process.  The fourth condition above on our choices of $a_0$ and $t_0$ guarantees that on the event $\{T^*(0) \geq t_0\}$, we have $T' + d < T^*(0) - k' T^*(0)^{2/3}$.  Also, $(1 - q - \delta)/(1 - q + \delta) \leq 1 - \delta$, so (\ref{Bcpre}) implies that (\ref{Bsurvival}) holds also on $B^c \cap \{T^*(0) \geq t_0\}$, and therefore on $\{T' \geq t_0\}$.
\end{proof}

\begin{lemma}\label{MGsurvival}
Let $(\nu_n)_{n=1}^{\infty}$ be a sequence of deterministic initial configurations.  Let $(s_n)_{n=1}^{\infty}$ and $(t_n)_{n=1}^{\infty}$ be sequences of times such that:
\begin{equation}\label{sntn}
\mbox{1) } 0 \leq s_n \leq t_n \mbox{ for all } n, \hspace{.3in}\mbox{2)} \lim_{n \rightarrow \infty} (t_n - s_n) = \infty, \hspace{.3in}\mbox{3)}\lim_{n \rightarrow \infty} s_n/t_n = 1.
\end{equation}
Suppose that, under $\P_{\nu_n}$, we have $Z_{t_n}(0) \rightarrow 0$ and $L_{t_n}(0) - R(0) \rightarrow \infty$ as $n \rightarrow \infty$.  For $0 \leq u \leq t_n$, define
\begin{equation}\label{Wdef}
W_n(u) = \P_{\nu_n}(\zeta > t_n \,|\, {\cal F}_u).
\end{equation}
Under the conditional probability measure $\P_{\nu_n}( \: \cdot \:|\,\zeta > t_n)$, we have $W_n(s_n) \rightarrow_p 1$ as $n \rightarrow \infty$.  Moreover, for all $\eps > 0$ and $a \in (0, 1)$, there exists $\delta > 0$ such that for sufficiently large $n$,
\begin{equation}\label{infW}
\P_{\nu_n}\Big( \inf_{at_n \leq u \leq t_n} W_n(u) \leq \delta \,\big|\, \zeta > t_n \Big) < \eps.
\end{equation}
\end{lemma}

\begin{proof}
Suppose conditions 1), 2), and 3) hold.
Let $\eps > 0$.  Choose $m$ sufficiently large that $e^{-C_7 m} < \eps^2$, where $C_7$ is the constant from Lemma \ref{survivalgen}.  By Theorem \ref{CSBPcond}, conditional on $\zeta > t_n$, the finite-dimensional distributions of the processes $(Z_{t_n}((1 - e^{-u}) t_n), u \geq 0)$ converge as $n \rightarrow \infty$ to the finite-dimensional distributions of $(\Phi(u), u \geq 0)$, which is a continuous-state branching process started at zero and conditioned to go to infinity as $u \rightarrow \infty$.  Therefore, we can choose $v \in (0, 1)$ sufficiently close to $1$ that 
\begin{equation}\label{vsurvival}
\P_{\nu_n}(Z_{t_n}(vt_n) > m \,|\, \zeta > t_n) > 1 - \eps
\end{equation}
for sufficiently large $n$.  Lemma \ref{survivalgen} implies that $\P_{\nu_n}(\zeta > t_n\,|\,{\cal F}_{vt_n}) \geq 1 - e^{-C_7 m} > 1 - \eps^2$ on the event $\{Z_{t_n}(vt_n) > m\}$ for sufficiently large $n$.  That is, we have $W_n(vt_n) > 1 - \eps^2$ on $\{Z_{t_n}(vt_n) > m\}$ for sufficiently large $n$.  Therefore, (\ref{vsurvival}) implies that for sufficiently large $n$, we have
\begin{equation}\label{Weq1}
\P_{\nu_n}(W_n(vt_n) > 1 - \eps^2 \,|\, \zeta > t_n) > 1 - \eps.
\end{equation}
Since $(W_n(u), 0 \leq u \leq t_n)$ is a $[0,1]$-valued martingale, it follows from the Optional Sampling Theorem that
\[
\P_{\nu_n}\Big( \inf_{vt_n \leq u \leq t_n} W_n(u) > 1 - \eps \,\big| \, W_n(vt_n) > 1 - \eps^2\Big) \geq 1 - \eps.
\]
We claim that we also have,
\begin{equation}\label{Weq2}
\P_{\nu_n}\Big( \inf_{vt_n \leq u \leq t_n} W_n(u) > 1 - \eps \,\big| \, \{W_n(vt_n) > 1 - \eps^2\} \cap \{\zeta > t_n\}\Big) \geq 1 - \eps.
\end{equation}
To see this, note that the further conditioning on the event $\{\zeta > t_n\} = \{W_n(t_n) = 1\}$ can only increase the probability that the martingale stays above $1 - \eps$ because the martingale can not stay above $1 - \eps$ between times $vt_n$ and $t_n$ on the event $\{\zeta > t_n\}^c = \{W_n(t_n) = 0\}$.  From (\ref{sntn}), (\ref{Weq1}), and (\ref{Weq2}), we get that for sufficiently large $n$, $$\P_{\nu_n}(W_n(s_n) > 1 - \eps \,|\, \zeta > t_n) \geq \P_{\nu_n}\Big(\inf_{vt_n \leq u \leq t_n} W_n(u) > 1 - \eps \,\big|\, \zeta > t_n \Big) > (1 - \eps)^2,$$ which immediately gives the first conclusion of the lemma when conditions 1), 2), and 3) hold.

It remains to prove (\ref{infW}).  There exists $b > 0$ such that $P(\Phi(-\log(1-a)) > b) > 1 - \eps/2$.  Then Theorem \ref{CSBPcond} implies that
$$\P_{\nu_n}(Z_{t_n}(a t_n) > b \,|\, \zeta > t_n) > 1 - \frac{\eps}{2}$$ for sufficiently large $n$.  It follows from Lemma \ref{survivalgen} that, for sufficiently large $n$, we have $W_n(at_n) > 1 - e^{-C_7 b}$ on the event $\{Z_{t_n}(at_n) > b\}$, and therefore, writing $d = 1 - e^{-C_7 b} > 0$, we have
\begin{equation}\label{infW1}
\P_{\nu_n}(W_n(at_n) > d \,|\, \zeta > t_n) > 1 - \frac{\eps}{2}.
\end{equation}
Let $\delta = d \eps/2$, and let $D$ be the event that $\inf_{at_n \leq u \leq t_n} W_n(u) \leq \delta$.  Using Bayes' Rule followed by the Optional Sampling Theorem, along with the trivial bound $\P_{\nu_n}( D \,|\, W_n(at_n) > d ) \leq 1$, we get
\begin{align}\label{infW2}
\P_{\nu_n} ( D \,|\, \{W_n(at_n) > d\} \cap \{\zeta > t_n\}) &= \frac{\P_{\nu_n}( D \,|\, W_n(at_n) > d ) \P_{\nu_n}(\zeta > t_n\,|\, D \cap \{W_n(at_n) > d\})}{\P_{\nu_n}(\zeta > t_n\,|\,W_n(at_n) > d)} \nonumber \\
&\leq \frac{\delta}{d}.
\end{align}
It follows from (\ref{infW1}) and (\ref{infW2}) that for sufficiently large $n$,
$$\P_{\nu_n}(D^c\,|\,\zeta > t_n) \geq \P_{\nu_n}(W_n(at_n) > d\,|\,\zeta > t_n) \P_{\nu_n}(D^c\,|\,\{W_n(at_n) > d\} \cap \{\zeta > t_n\}) \geq \bigg(1 - \frac{\eps}{2} \bigg)^2,$$
which implies (\ref{infW}).
\end{proof}

\begin{lemma}\label{gensequence}
Let $(\nu_n)_{n=1}^{\infty}$ be a sequence of deterministic initial configurations.  Let $(s_n)_{n=1}^{\infty}$ and $(t_n)_{n=1}^{\infty}$ be sequences of times such that
\begin{equation}\label{sntn2}
\mbox{1) } 0 \leq s_n \leq t_n \mbox{ for all } n, \hspace{.3in}\mbox{2)} \lim_{n \rightarrow \infty} (t_n - s_n) = \infty, \hspace{.3in}\mbox{3)}\liminf_{n \rightarrow \infty} s_n/t_n > 0.
\end{equation}
Suppose, under $\P_{\nu_n}$, we have $Z_{t_n}(0) \rightarrow 0$ and $L_{t_n}(0) - R(0) \rightarrow \infty$ as $n \rightarrow \infty$.  Then, under the conditional probability measure $\P_{\nu_n}( \: \cdot \:|\,\zeta > t_n)$, we have $T(s_n) \rightarrow_p \infty$ and $L_{T(s_n)}(0) - R(s_n) \rightarrow_p \infty$ as $n \rightarrow \infty$.
\end{lemma}

\begin{proof}
Let $\eps > 0$ and $A > 0$.  Define the martingale $(W_n(u), 0 \leq u \leq t_n)$ as in Lemma \ref{MGsurvival}.  Choose $a > 0$ such that $\liminf_{n \rightarrow \infty} s_n/t_n > 2a$, and choose $\delta > 0$ such that (\ref{infW}) holds for sufficiently large $n$.  It follows from (\ref{infW}) that
\begin{equation}\label{Wsn}
\P_{\nu_n} (W_n(s_n) \leq \,\delta \,|\, \zeta > t_n) < \eps.
\end{equation}
By Lemma \ref{smallTlem} and the fact that $t_n - s_n \rightarrow \infty$, for any fixed $K > 0$, we have $W_n(s_n) < \delta$ on the event $\{T(s_n) \leq K\}$ for sufficiently large $n$.  Therefore, for sufficiently large $n$, we have $\P_{\nu_n}(T(s_n) \leq K\,|\,\zeta > t_n) < \eps$.  It follows that $T(s_n) \rightarrow_p \infty$ as $n \rightarrow \infty$ under $\P_{\nu_n}( \: \cdot \:|\,\zeta > t_n)$.

Choose $d$ and $t_0$ as in Lemma \ref{mainLRlemma}, with $\delta \eps$ playing the role of $\eps$.  Because $s_n - d > at_n$ for sufficiently large $n$, the reasoning that led to (\ref{Wsn}) also gives
\begin{equation}\label{Wsnd}
\P_{\nu_n} (W_n(s_n - d) \leq \,\delta \,|\, \zeta > t_n) < \eps.
\end{equation}
By applying Lemma~\ref{mainLRlemma} with the configuration of particles at time $s_n - d$ playing the role of $\nu$, we get
$$\P_{\nu_n}(\{R(s_n) \geq L_{T(s_n)}(0) - A\} \cap \{T(s_n) \geq t_0\}\,|\,{\cal F}_{s_n - d}) < \delta \eps.$$
In particular, because $\{W_n(s_n - d) > \delta\} \in {\cal F}_{s_n - d}$, we have
\begin{equation}\label{condW}
\P_{\nu_n}(\{R(s_n) \geq L_{T(s_n)}(0) - A\} \cap \{T(s_n) \geq t_0\}\,|\,W_n(s_n - d) > \delta) < \delta \eps.
\end{equation}
Elementary probability results imply that if $B$, $C$, $D$, and $E$ are events, then
\begin{align*}
P(B|E) &\leq P(B \cap C \cap D|E) + P(C^c|E) + P(D^c|E) \\
&= P(B \cap C \cap E|D) \cdot \frac{P(D|E)}{P(E|D)} + P(C^c|E) + P(D^c|E).
\end{align*}
Now write $B = \{R(s_n) \geq L_{T(s_n)}(0) - A\}$, $C = \{T(s_n) \geq t_0\}$, $D = \{W_n(s_n - d) > \delta\}$, and $E = \{\zeta > t_n\}$.  Note that $P(E|D) > \delta$ by definition, and $P(D|E) > 1 - \eps$ by (\ref{Wsnd}).  Also, $P(B \cap C \cap E|D) \leq \delta \eps$ by (\ref{condW}), and $P(C^c|E) < \eps$ for sufficiently large $n$ because we already know that $T(s_n) \rightarrow_p \infty$ as $n \rightarrow \infty$ under $\P_{\nu_n}( \: \cdot \: |\, \zeta > t_n)$.  Thus, for sufficiently large $n$,
$$\P_{\nu_n}(R(s_n) \geq L_{T(s_n)}(0) - A\,|\, \zeta > t_n) < \delta \eps \cdot \frac{1}{\delta} + 2 \eps = 3 \eps.$$  Because $\eps > 0$ and $A > 0$ were arbitrary, it follows that $L_{T(s_n)}(0) - R(s_n) \rightarrow_p \infty$ under the conditional probability measure $\P_{\nu_n}( \: \cdot \:|\,\zeta > t_n)$ as $n \rightarrow \infty$.
\end{proof}

\begin{lemma}\label{Ttcompare}
Let $(\nu_n)_{n=1}^{\infty}$ be a sequence of deterministic initial configurations.  Let $(t_n)_{n=1}^{\infty}$ be a sequence of times tending to infinity.  Let $\delta > 0$, and let $(s_n)_{n=1}^{\infty}$ be a sequence of times such that $\delta t_n \leq s_n \leq (1 - \delta) t_n$ for all $n$.  Suppose, under $\P_{\nu_n}$, we have $Z_{t_n}(0) \rightarrow 0$ and $L_{t_n}(0) - R(0) \rightarrow \infty$ as $n \rightarrow \infty$.  Then, under the conditional probability measure $\P_{\nu_n}( \: \cdot \:|\,\zeta > t_n)$, we have $L_{t_n}(s_n) - R(s_n) \rightarrow_p \infty$.
\end{lemma}

\begin{proof}
We will show that for all $\eps > 0$, there is a positive constant $C$, depending on $\delta$ and $\eps$, such that
\begin{equation}\label{Ttnts}
\P_{\nu_n}(|T(s_n) - (t_n - s_n)| > C t_n^{2/3} \,|\, \zeta > t_n) < \eps.
\end{equation}
Because $L_{T(s_n)}(0) - R(s_n) \rightarrow_p \infty$ under $\P_{\nu_n}( \: \cdot \:|\,\zeta > t_n)$ by Lemma \ref{gensequence}, we can see from (\ref{tpt}) that (\ref{Ttnts}) implies the result of the lemma.

By Lemma \ref{MGsurvival}, there exists $\eta > 0$ such that $\P_{\nu_n}(W_n(s_n) \leq \eta\,|\,\zeta > t_n) < \eps/4$ for sufficiently large $n$.  By Lemmas \ref{tightextinct} and \ref{gensequence} there is a constant $k'$ such that, if $H_n$ denotes the random variable $\P_{\nu_n}(|(\zeta - s_n) - T(s_n)| \leq k' T(s_n)^{2/3}\,|\,{\cal F}_{s_n})$, then $\P_{\nu_n}(H_n > 1 - \eta \eps/4 \,|\, \zeta > t_n) \rightarrow 1$ as $n \rightarrow \infty$.  Elementary probability results imply that if $B$, $C$, and $D$ are events, then
$$P(B|D) \leq P(C^c|D) + P(B|C \cap D) \leq P(C^c|D) + \frac{P(B|C)}{P(D|C)}.$$  By taking $B = \{|(\zeta - s_n) - T(s_n)| > k'T(s_n)^{2/3}\}$, $C = \{H_n > 1 - \eta \eps/4\} \cap \{W_n(s_n) > \eta\} \in {\cal F}_{s_n}$, and $D = \{\zeta > t_n\}$, we get that for sufficiently large $n$,
\begin{equation}\label{Tt1}
\P_{\nu_n}(|(\zeta - s_n) - T(s_n)| > k'T(s_n)^{2/3} \,|\, \zeta > t_n) \leq \frac{\eps}{4} + \frac{(\eta \eps/4)}{\eta} = \frac{\eps}{2}.
\end{equation}
Proposition \ref{extinctiondist} implies that there is another positive constant $k$ such that
\begin{equation}\label{Tt2}
\P_{\nu_n}(\zeta > t_n + kt_n^{2/3} \,|\, \zeta > t_n) < \eps/4.
\end{equation}
Now (\ref{Tt1}) and (\ref{Tt2}) imply
\begin{equation}\label{Tt3}
\P_{\nu_n}(T(s_n) > (t_n - s_n) + kt_n^{2/3} + k'T(s_n)^{2/3} \,|\, \zeta > t_n) < 3 \eps/4.
\end{equation}

To obtain the necessary lower bound on $T(s_n)$, first note that by Theorem \ref{CSBPcond} and the assumptions on $s_n$, there exists $\delta > 0$ such that
\begin{equation}\label{Tt4}
\P_{\nu_n}(Z_{t_n}(s_n) < \delta\,|\,\zeta > t_n) < \eps/4
\end{equation}
for sufficiently large $n$.  Choose $k$ large enough that $e^{-ck/3}/2 < \delta$.  Lemmas \ref{monotone} and \ref{gensequence} imply that under $\P_{\nu_n}(\: \cdot \: |\, \zeta > t_n)$, with probability tending to one as $n \rightarrow \infty$, we have $Z_{s_n + T(s_n)}(s_n) = 1/2$ and therefore, in view of (\ref{zchange}), we also have $Z_{s_n + T(s_n) + kT(s_n)^{2/3}}(s_n) < \delta$ for sufficiently large $n$.  On this event, by the monotonicity established in Lemma \ref{monotone}, if $t_n > s_n + T(s_n) + kT(s_n)^{2/3}$ then $Z_{t_n}(s_n) < \delta$.  Combining this observation with (\ref{Tt4}), we see that for sufficiently large $n$,
\begin{equation}\label{Tt5}
\P_{\nu_n}(T(s_n) < (t_n - s_n) - kT(s_n)^{2/3} \,|\, \zeta > t_n) < \eps/4.
\end{equation}
Now (\ref{Ttnts}) can be deduced from (\ref{Tt3}) and (\ref{Tt5}).
\end{proof}

\begin{proof}[Proof of Theorem \ref{larges}]
If $t^{-2/3}(t - s) \to \sigma \geq 0$, then let $r = s - t^{2/3}$.  If $t^{2/3} \ll t - s \ll t$, then let $r = 2s - t$, so that $s - r = t - s$.  Throughout the proof, we will work under the conditional distribution $\P_{\nu_t}( \: \cdot \:|\, \zeta > t)$. We will repeatedly make use of the fact that $\P_{\nu_t}(\zeta > t\,|\,{\cal F}_r) \rightarrow_p 1$ as $t \rightarrow \infty$, by Lemma \ref{MGsurvival}. Indeed, this allows us to remove the conditioning when applying results (namely, Lemma \ref{tightextinct} and Proposition \ref{configpropnew})  with the particle configuration at time $r$ playing the role of the initial configuration.

We first claim that
\begin{equation}\label{Tzeta}
t^{-2/3}(T(r) - (\zeta - r)) \rightarrow_p 0 \hspace{.2in}\mbox{as }t \rightarrow \infty.
\end{equation}
Our choice of $r$ ensures that $1 \ll t - r \ll t$, and Proposition \ref{extinctiondist} states that 
\begin{equation}
\label{eq:zetaV}
t^{-2/3}(\zeta - t) \Rightarrow \frac{3}{c}V.
\end{equation}
Combining these facts, we get
\begin{equation}\label{zetasmall}
t^{-1}(\zeta - r) \rightarrow_p 0 \hspace{.2in}\mbox{as }t \rightarrow \infty.
\end{equation}
On the other hand, Lemma \ref{gensequence} implies that $T(r) \rightarrow_p \infty$ and $L_{T(r)}(0) - R(r) \rightarrow_p \infty$ as $t \rightarrow \infty$. Then Lemma \ref{tightextinct} implies that for all $\eps > 0$, there is a constant $k'$ such that
\begin{equation}\label{zetaTdiff}
\P_{\nu_t}(|(\zeta - r) - T(r)| \leq k' T(r)^{2/3} \,|\, \zeta > t) > 1 - \eps
\end{equation}
for sufficiently large $t$. Now we see from (\ref{zetasmall}) and (\ref{zetaTdiff}) that $t^{-1}T(r) \rightarrow_p 0$ as $t \rightarrow \infty$, and then another application of (\ref{zetaTdiff}) yields (\ref{Tzeta}) as claimed.

Now let $V_t = t^{-2/3}(T(r) - (t - r))$.  It follows from (\ref{Tzeta}) and \eqref{eq:zetaV} that
\begin{equation}\label{Tudist}
V_t = t^{-2/3}(T(r) - (t - r)) \Rightarrow \frac{3}{c} V
\end{equation}
and
\begin{equation}\label{simcon1}
t^{-2/3}(\zeta - t) - V_t \rightarrow_p 0.
\end{equation}

We now apply Proposition~\ref{configpropnew} with the configuration of particles at time $r$ playing the role of the initial configuration of particles and the time $T(r)$ playing the role of $t_n$.  The assumptions of Proposition~\ref{configpropnew} are satisfied, because, as mentioned above, $T(r) \rightarrow_p \infty$ and $L_{T(r)}(0) - R(r) \rightarrow_p \infty$ as $t \rightarrow \infty$, which implies that $Z_{r+T(r)}(r) \rightarrow_p 1/2$ as $t\to\infty$, by the definition of $T(r)$.  If $t^{-2/3}(t - s) \to \sigma \geq 0$, then using (\ref{Tudist}), we have
\begin{equation}\label{strat1}
\frac{s - r}{T(r)} = \frac{t^{2/3}}{(t - s) + (s - r) + (T(r) - (t - r))} \Rightarrow \frac{1}{\sigma + 1 + \frac{3}{c}V}.
\end{equation}
The limiting random variable on the right-hand side is $(0,1)$-valued, so given $\eps > 0$, we can find $\delta > 0$ such that $\P_{\nu_t}(\delta T(r) \leq s - r \leq (1 - \delta)T(r) \,|\, \zeta > t) > 1 - \eps/2$.
If instead $t^{2/3} \ll t - s \ll t$, then using again \eqref{Tudist},
\begin{equation}\label{strat2}
\frac{s - r}{T(r)} = \frac{t - s}{2(t - s) + (T(r) - (t - r))} \Rightarrow \frac{1}{2}.
\end{equation}
It follows that in both cases, we can apply Proposition \ref{configpropnew} to get that if $\eps > 0$, then for sufficiently large $t$ we have
\begin{equation}\label{Me1}
\P_{\nu_t} \bigg( \frac{C_{9}}{L_{T(r)}(s-r)^3} e^{L_{T(r)}(s-r)} \leq M(s) \leq \frac{C_{10}}{L_{T(r)}(s-r)^3} e^{L_{T(r)}(s-r)} \,\Big|\, \zeta > t \bigg) > 1 - \eps
\end{equation}
and
\begin{equation}\label{Re1}
\P_{\nu_t} \big( L_{T(r)}(s-r) - \log T(r) - C_{11} \leq R(s) \leq L_{T(r)}(s-r) - \log T(r) + C_{12} \,\big|\, \zeta > t \big) > 1 - \eps.
\end{equation}

We write that $W_t$ is $O_p(1)$ if for all $\eps > 0$, there exists a positive real number $K$ such that $P(|W_t| \leq K) > 1 - \eps$ for sufficiently large $t$, and we write that $W_t$ is $o_p(1)$ if $W_t \rightarrow_p 0$.  Then, by (\ref{Tudist}) and (\ref{Me1}),
\begin{align*}
\log M(s) &= L_{T(r)}(s - r) - 3 \log L_{T(r)}(s - r) + O_p(1) \\
&= c(t - s + t^{2/3} V_t)^{1/3} - \log (t - s + t^{2/3} V_t) + O_p(1).
\end{align*}
Likewise, by (\ref{Tudist}) and (\ref{Re1}),
\begin{align*}
R(s) &= c (t - s + t^{2/3} V_t)^{1/3} - \log (t - r + t^{2/3} V_t) + O_p(1).
\end{align*}
When $t^{-2/3}(t - s) \to \sigma \geq 0$, it follows that $$t^{-2/9} \log M(s) = c(\sigma + V_t)^{1/3} + o_p(1)$$ and $$t^{-2/9} R(s) = c(\sigma + V_t)^{1/3} + o_p(1).$$  These two results, combined with (\ref{Tudist}) and (\ref{simcon1}), give (\ref{J1as}).
When instead $t^{2/3} \ll t - s \ll t$, the Mean Value Theorem implies that for some random variable $\xi_t$ such that $0 \leq \xi_t \leq t^{2/3} V_t$, we have
$$\log M(s) = c (t - s)^{1/3} - \log(t - s) + \frac{c}{3} (t - s + \xi_t)^{-2/3} t^{2/3} V_t + O_p(1).$$
Because $(t - s + \xi_t)/(t - s) \rightarrow_p 1$, it follows that
$$\bigg( \frac{t - s}{t} \bigg)^{2/3} \big( \log M(s) - c(t - s)^{1/3} + \log(t - s) \big) = \frac{c}{3} V_t + o_p(1).$$  By the same reasoning, we get
$$\bigg( \frac{t - s}{t} \bigg)^{2/3} \big( R(s) - c(t - s)^{1/3} + \log(t - s) \big) = \frac{c}{3} V_t + o_p(1).$$
These results, combined with (\ref{Tudist}) and (\ref{simcon1}), imply (\ref{J2as}).
\end{proof}

\begin{proof}[Proof of Theorem \ref{meds}]
Consider a sequence of times $(t_n)_{n=1}^{\infty}$ tending to infinity, and choose $(s_n)_{n=1}^{\infty}$ such that $\delta t_n \leq s_n \leq (1 - \delta) t_n$ for all $n$.  We will condition on $\zeta > t_n$ and then apply Proposition \ref{configpropnew} with the configuration of particles at time $\delta t_n/2$ playing the role of the initial configuration of particles.  Because $P(0 < \Phi(u) < \infty) = 1$ for all $u > 0$, it follows from Theorem \ref{CSBPcond} that, under $\P_{\nu_{t_n}}(\:\cdot\:|\,\zeta > t_n)$, the distributions of the sequences $(Z_{t_n}(\delta t_n/2))_{n=1}^{\infty}$ and $(Z_{t_n}(\delta t_n/2)^{-1})_{n=1}^{\infty}$ are tight.  Lemma \ref{Ttcompare} implies that, under $\P_{\nu_{t_n}}(\:\cdot\:|\,\zeta > t_n)$, we have $L_{t_n}(\delta t_n/2) - R(\delta t_n/2) \rightarrow_p \infty$ as $n \rightarrow \infty$.  Therefore, the hypotheses of Proposition~\ref{configpropnew} are satisfied.

To deduce the result of Theorem \ref{meds} from Proposition \ref{configpropnew}, we need to show that the conclusions are unaffected by conditioning on $\zeta > t$.  We proceed as in the proof of Lemma~\ref{Ttcompare}.  By Lemma \ref{MGsurvival}, there exists $\eta > 0$ such that $\P_{\nu_{t_n}}(W_n(\delta t_n) \leq \eta \,|\,\zeta > t_n) < \eps/2$ for sufficiently large $n$.  By Proposition \ref{configpropnew}, if we define the random variables $$H_n = \P_{\nu_{t_n}} \bigg( \frac{C_3}{L_{t_n}(s_n)^3} e^{L_{t_n}(s_n)} \leq M(s_n) \leq \frac{C_4}{L_{t_n}(s_n)^3} \,\bigg|\, {\cal F}_{\delta t_n/2} \bigg)$$ and
$$J_n = \P_{\nu_{t_n}}\big( L_t(s) - \log t - C_5 \leq R(s) \leq L_t(s) - \log t + C_6 \, \big| \, {\cal F}_{\delta t_n/2} \big),$$ then $\P_{\nu_{t_n}}(H_n > 1 - \eta \eps/2 \,|\, \zeta > t_n) \rightarrow 1$ and $\P_{\nu_{t_n}}(J_n > 1 - \eta \eps/2 \,|\, \zeta > t_n) \rightarrow 1$ as $n \rightarrow \infty$, provided that we choose the values of the constants so that (\ref{configconc1}) and (\ref{configconc2}) hold with $\eta \eps/2$ in place of $\eps$.
Following the steps in the derivation of (\ref{Tt1}) then yields the two conclusions in Theorem~\ref{meds}.
\end{proof}

\begin{proof}[Proof of Theorem \ref{condconfigprop}]
Consider any sequence of times $(t_n)_{n=1}^{\infty}$ tending to infinity, and let $s_n$ be the value of $s$ associated with the time $t_n$.  We first consider the case in which $t_n - s_n \ll t_n$.
Let $r_n = s_n - t_n^{2/3}$ if $t_n - s_n \leq t_n^{2/3}$, and let $r_n = 2s_n - t_n$ if $t_n - s_n \geq t_n^{2/3}$.  Let $A_n^{\delta}$ be the event that $\delta T(r_n) \leq s_n - r_n \leq (1 - \delta) T(r_n)$.  Using the same reasoning used to establish (\ref{strat1}) and (\ref{strat2}), we can see that for all $\eps > 0$, there is a $\delta > 0$ such that $\P_{\nu_{t_n}}(A_n^{\delta} \,|\, \zeta > t_n) > 1 - \eps$ for sufficiently large $n$.  

We apply Proposition \ref{configpropnew} with the configuration of particles at time $r_n$ playing the role of the initial configuration of particles, the time $T(r_n)$ playing the role of $t_n$, and $s_n - r_n$ playing the role of $s_n$.  The result of part 3 of Proposition \ref{configpropnew} only applies on the event $A_n^{\delta}$.  Therefore, we will define the probability measure $\chi_n^{\delta}$ to be equal to $\chi_{t_n}$ on the event $A_n^{\delta}$ and to be equal to $\mu$ otherwise.  Likewise, we will define the probability measure $\eta_n^{\delta}$ in the same way as $\eta_{t_n}$, except with $L_{T(u_n)}(s_n - r_n)$ in place of $R(s_n)$ in the definition, on the event $A_n^{\delta}$.  Otherwise, we define $\eta_n^{\delta}$ to be the probability measure $\xi$.  Define $\eta_t^*$ to be the same as $\eta_t$, except with $L_{T(r_n)}(s_n - r_n)$ in place of $R(s_n)$ in the definition.  Then part 3 of Proposition~\ref{configpropnew} implies that for all $\delta > 0$, we have $\chi_n^{\delta} \Rightarrow \mu$ and $\eta_n^{\delta} \Rightarrow \xi$ as $n \rightarrow \infty$.  Note that Lemma \ref{MGsurvival} ensures that the conditioning on $\zeta > t$ does not affect the result when we apply Proposition~\ref{configpropnew}.
Therefore, letting $\rho$ denote the Prohorov metric on the space of probability measures on $\R$, we have $$\lim_{n \rightarrow \infty} \P_{\nu_{t_n}}(\rho(\chi_n^{\delta}, \mu) > \eps\,|\,\zeta > t_n) = 0, \hspace{.4in}\lim_{n \rightarrow \infty} \P_{\nu_{t_n}}(\rho(\eta_n^{\delta}, \xi) > \eps\,|\,\zeta > t_n) = 0.$$  Because $\P_{\nu_{t_n}}(A_n^{\delta} \,|\, \zeta > t_n) > 1 - \eps$ for sufficiently large $n$, it follows that
$$\limsup_{n \rightarrow \infty} \P_{\nu_{t_n}}(\rho(\chi_{t_n}, \mu) > \eps\,|\, \zeta > t_n) \leq \eps, \hspace{.4in}\limsup_{n \rightarrow \infty} \P_{\nu_{t_n}}(\rho(\eta^*_{t_n}, \xi) > \eps\,|\,\zeta > t_n) \leq \eps.$$  Because $\eps > 0$ is arbitrary, it follows that $\chi_t \Rightarrow \mu$ and $\eta_t^* \Rightarrow \xi$.  Finally, $R(s)/L_{T(r)}(s - r) \rightarrow_p 1$ as $t \rightarrow \infty$ by part 2 of Proposition \ref{configpropnew}, so $\eta_t \Rightarrow \xi$, as claimed.

By a subsequence argument, it remains only to consider the case in which, for some $\delta > 0$, we have $\delta t_n \leq s_n \leq (1-\delta)t_n$ for all $n$.  In this case, we can apply part 3 of Proposition \ref{configpropnew} with the configuration of particles at time $\delta t_n/2$ playing the role of the initial configuration of particles, as in the proof of Theorem \ref{meds}, to obtain the result.  Because the limit distributions $\mu$ and $\eta$ are concentrated on a single measure, the result of Lemma \ref{MGsurvival} remains enough to ensure that the conditioning on $\zeta > t$ does not affect the conclusion.
\end{proof}

\section{Moment estimates}\label{momsec}

\subsection{Heat kernel estimates} \label{sec:heat_kernel}

First, consider a single Brownian particle which is killed when it reaches $0$ or $1$.  Let $w_s(x,y)$ denote the ``density'' of the position of this particle at time $s$, meaning that if the Brownian particle starts at the position $x \in (0, 1)$ at time zero, then the probability that it is in the Borel subset $U$ of $(0,1)$ at time $s$ is $$\int_U w_s(x,y) \: dy.$$  It is well-known (see, for example, p. 188 of \cite{lawler}) that
\begin{align}
\label{wfourier}
w_s(x,y) &= 2 \sum_{n=1}^{\infty} e^{-\pi^2 n^2 s/2} \sin (n \pi x) \sin (n \pi y).
%\\
%\label{wgaussian}
%&= \sum_{k\in\Z} \frac 1 {\sqrt{2\pi s}} \left[e^{-\frac{(x-y-2k)^2}{2s}} - e^{-\frac{(x+y-2k)^2}{2s}}\right],
\end{align}
%the two representations being related by the Poisson summation formula (see \cite{bellman}, §9).
Equation \eqref{wfourier} yields that for every $x\in[0,1]$ and $s\ge0$,
\begin{equation}
\label{wsin}
 \int_0^1 \sin(\pi y) w_s(x,y)\,dy = e^{-\pi^2 s/2}\sin(\pi x).
\end{equation}
Furthermore, by the reasoning in Lemma 5 of \cite{bbs2}, if we define
\begin{equation}\label{vdef}
v_s(x,y) = 2 e^{-\pi^2 s/2} \sin(\pi x) \sin(\pi y)
\end{equation}
and
\begin{equation}\label{Ddef}
D(s) = \sum_{n=2}^{\infty} n^2 e^{-\pi^2 (n^2 - 1) s/2},
\end{equation}
then
\begin{equation}\label{udef}
w_s(x,y) = v_s(x,y)(1 + D_s(x,y)),
\end{equation}
where $|D_s(x,y)| \leq D(s)$ for all $x, y \in (0, 1)$.
We further recall (see Lemma 7.1 of \cite{maillard}) that
\begin{align}
\label{wint}
 \int_0^s e^{\pi^2 r/2}w_r(x,y)\,dr = 2s\sin(\pi x)\sin(\pi y) + O\big((x\wedge y)(1-(x\vee y))\big),
\end{align}
and
\begin{align}
 \label{wprimeint}
\int_0^s e^{\pi^2 r/2}\left(-\frac 1 2 \partial_yw_r(x,1)\right)\,dr 
 &= \pi s\sin(\pi x) + O(x).
\end{align}
We will also need the following two lemmas.

\begin{lemma}
 \label{lem:wintunif}
 For all $x \in (0,1)$ and $y \in (0, 1/2]$, we have
 \[
  \int_0^s e^{\pi^2 r/2}\sup_{y'\in[0,y]} w_r(x,y')\,dr = O(y(s\sin(\pi x)+(1-x))).
 \]
\end{lemma}

\begin{proof}
 For $r \ge 1$, we have by (\ref{vdef}) and (\ref{udef}),
\begin{equation}\label{wlarges}
 \sup_{y'\in[0,y]} w_r(x,y') = O(e^{-\pi^2 r/2}\sin(\pi x)y).
\end{equation}
It therefore suffices to show that 
 \begin{equation}
  \label{letsshowit}
  \int_0^1 \sup_{y'\in[0,y]} w_r(x,y')\,dr = O(y(1-x)).
 \end{equation}
We bound $w_r(x,y)$ by the heat kernel of Brownian motion killed at $0$, i.e.
\[
w_r(x,y) \le \frac 1 {\sqrt{2\pi r}} \left(e^{-\frac{(x-y)^2}{2r}} - e^{-\frac{(x+y)^2}{2r}}\right) = \frac 1 {\sqrt{2\pi r}} e^{-\frac{(x-y)^2}{2r}}(1-e^{-\frac{2xy}r}).
\]
Using the inequality $1 - e^{-z} \leq 1 \wedge z$ for $z \geq 0$, we get
\begin{equation}\label{wrbound}
w_r(x,y) \leq \frac{1}{\sqrt{2 \pi r}} e^{-(x - y)^2/2r} \left(1 \wedge \frac{2 x y}{r} \right).
\end{equation}
 
The first step in proving (\ref{letsshowit}) is to show the weaker statement
 \begin{equation}
  \label{letsshowit_weaker}
  \int_0^1 \sup_{y'\in[0,y]} w_r(x,y')\,dr = O(y).
 \end{equation} 
To do this, we distinguish between two cases.  When $x \leq 2y$, equation (\ref{wrbound}) gives
 \[
 \sup_{y'\in[0,y]}w_r(x,y') \le \frac 1 {\sqrt{2\pi r}} \left(1\wedge \frac {4y^2}{r}\right).
 \]
 Integrating over $r$ and changing variables by $r = y^2 u$, this gives
 \begin{equation}
 \label{intbound1}
  \int_0^1 \sup_{y'\in[0,y]} w_r(x,y')\,dr \le y \int_0^\infty  \frac 1 {\sqrt{2\pi u}} \left(1\wedge \frac {4}{u}\right)\,du = O(y),
 \end{equation}
 because the last integral converges.  When $x > 2y$, we use that $x-y' \ge x/2$ for all $y'\le y$ to get 
 \[
 \sup_{y'\in[0,y]} w_r(x,y') \le \frac {2xy} {\sqrt{2\pi}r^{3/2}} e^{-x^2/8r}.
 \]
 Integrating over $r$ and changing variables by $r = x^2 u$,
 \begin{equation}
 \label{intbound2}
  \int_0^1 \sup_{y'\in[0,y]} w_r(x,y')\,dr \le 2 y \int_0^\infty \frac 1 {\sqrt{2\pi} u^{3/2}} e^{-1/8u}\,du = O(y),
 \end{equation}
 because the last integral converges.
Equations~\eqref{intbound1} and \eqref{intbound2} together yield \eqref{letsshowit_weaker}.

When $x \leq 3/4$, equation (\ref{letsshowit}) follows immediately from (\ref{letsshowit_weaker}).  Therefore, it remains to show (\ref{letsshowit}) when $x \geq 3/4$.  By symmetry, for all $x,y \in (0,1)$ and $r\ge0$, we have $w_r(x,y) = w_r(1-x, 1-y)$ and so,  using (\ref{wrbound}) for the last step,
\begin{align*}
w_r(x,y) &= \int_0^1 w_{r/2}(x,z) w_{r/2}(z,y) \: dz \\
&\leq \sup_{z \in (0,1)} w_{r/2}(x,z) w_{r/2}(z,y) \\
&= \sup_{z \in (0,1)} w_{r/2}(1-x, 1-z) w_{r/2}(z,y) \\
&\leq \sup_{z\in (0,1)} \frac{1}{\sqrt{\pi r}} e^{-(x-z)^2/r} \left( 1 \wedge \frac{4(1-x)}{r} \right) \cdot \frac{1}{\sqrt{\pi r}} e^{-(z-y)^2/r} \left(1 \wedge \frac{4y}{r} \right).
\end{align*}
Now note that when $x \geq 3/4$ and $y \leq 1/2$, for all $z \in (0,1)$ we have either $(x-z)^2 \geq 1/64$ or $(y-z)^2 \geq 1/64$. Hence, for all $x \geq 3/4$ and $y \leq 1/2$, we have
\begin{align*}
w_r(x,y) \leq \frac{y(1-x)}{r^3} e^{-1/64r}.
\end{align*}
It follows that when $y \leq 1/2$, we have
$$\int_0^1 \sup_{y' \in [0, y]} w_r(x,y') \: dr \leq y(1-x) \int_0^1 \frac{1}{r^3} e^{-1/64r} \: dr = O(y(1-x)),$$
because the integral converges.
\end{proof}

\begin{lemma}\label{wint2}
For all $x \in (0,1)$, we have $$\int_0^s e^{\pi^2 r/2} \int_0^1 w_r(x,y) \: dy \: dr = O(s \sin(\pi x) + (1-x)).$$
\end{lemma}

\begin{proof}
Exchanging integrals, this is an immediate consequence of \eqref{wint}.
%Using the bound (\ref{wlarges}), we immediately get, when $s \geq 1$,
%$$\int_1^s e^{\pi^2 r/2} \int_0^1 w_r(x,y) \: dy \: dr = O(s \sin(\pi x)).$$
%Also, note that $\int_0^1 w_r(x,y) \: dy$ is bounded above by the probability that a Brownian motion started at $x$ does not hit $1$ before time $r$, which is at most $C(1-x)/\sqrt{r}$.  It follows that
%$$\int_0^1 e^{\pi^2 r/2} \int_0^1 w_r(x,y) \: dy \: dr \leq e^{\pi^2/2} \int_0^1 \frac{C(1-x)}{\sqrt{r}} \: dr = O(1-x).$$ The result follows.
\end{proof}

We now wish to estimate the density of the position of the Brownian particle at time $s$ when the particle is killed if it reaches either $0$ or $K(s)$ at time $s$, where $K(s)$ is a smooth positive function.  That is, the right boundary at which the Brownian particle is killed moves over time.  We will need somewhat sharper estimates than those provided in \cite{bbs}.  To obtain such estimates, we will follow almost exactly the approach used by Roberts \cite{roberts}, which in turn was inspired by the work of Novikov \cite{nov}.  We will use the following general lemma.

\begin{lemma}\label{densityK}
Let $T > 0$.  Let $K: [0, T] \rightarrow (0, \infty)$ be twice differentiable. Let $x\in [0,K(0)]$.  Let $(\Omega, {\cal F}, P)$ be a probability space and 
$(B_s, s \geq 0)$ be Brownian motion started at $x$ on this space.  For $s \in [0, T]$, let
\begin{equation}\label{rhodef}
\rho_s = \bigg( \frac{K(0)}{K(s)} \bigg)^{1/2} \exp \bigg( \frac{K'(s) B_s^2}{2K(s)} - \frac{K'(0)B_0^2}{2K(0)} - \int_0^s \frac{K''(u)B_u^2}{2K(u)} \: du \bigg)
\end{equation}
and 
\begin{equation}\label{taudef}
\tau(s) = \int_0^s \frac{1}{K(u)^2} \: du.
\end{equation}
Then  $(\rho_s)_{s\in[0,T]}$ is a martingale and under the measure $Q$ defined by $dQ/dP = \rho_T$, $(B_s)_{s\in[0,T]}$ is equal in law to $(K(s) W_{\tau(s)})_{s\in[0,T]}$, where $(W_u)_{u\ge0}$ is a Brownian motion started at $x/K(0)$. In particular, for all bounded measurable 
functions $g: [0, 1] \rightarrow \R$ and all $s \in (0, T]$, we have
$$E \bigg[ \rho_s g \bigg( \frac{B_s}{K(s)} \bigg) \Ind_{\{0 < B_u < K(u) \: \forall u \in [0, s]\}} \bigg] = \int_0^1 g(y) w_{\tau(s)}(\tfrac x {K(0)}, y) \: dy.$$
\end{lemma}

\begin{proof}
%We may assume that $(B_s, s \geq 0)$ is defined on the probability space $(\Omega, {\cal F}, P)$ and is a Brownian motion with respect to the complete right-continuous filtration $({\cal G}_s, s \geq 0)$.  
%\nota{Pascal: is the previous sentence necessary/useful?}
Denote by $({\cal G}_s, s \geq 0)$ the Brownian filtration, i.e.~the smallest complete, right-continuous filtration to which $(B_s, s\ge0)$ is adapted.
For $s \in [0, T]$, let
\begin{equation}\label{Xdef}
X_s = \frac{K(s)}{K(0)} x + K(s) \int_0^s \frac{1}{K(u)} \: dB_u.
\end{equation}
A short calculation gives
\begin{equation}\label{sdeX}
dX_s = \frac{K'(s)}{K(0)} x \: ds + K'(s) \left(\int_0^s \frac{1}{K(u)} \: dB_u\right) \: ds + dB_s = \frac{K'(s)}{K(s)} X_s \: ds + dB_s.
\end{equation}
That is, $(X_s, 0 \leq s \leq u)$ is a Brownian motion with a time and space dependent drift whose drift at time $s$ is given by $K'(s) X_s/K(s)$.  For $s \in [0, T]$, let $$\gamma_s = \exp \bigg( \int_0^s \frac{K'(u) B_u}{K(u)} \: dB_u - \frac{1}{2} \int_0^s \frac{K'(u)^2 B_u^2}{K(u)^2} \: du \bigg).$$  We show below by an integration by parts argument that $\gamma_s = \rho_s$ for all $s\in[0,T]$, where $\rho_s$ is defined in (\ref{rhodef}), and assume this for the moment. Because $K'(u)/K(u)$ is bounded over $u \in [0, T]$ by assumption, it follows, for example, from Corollary 3.5.14 in \cite{ks} that the process $(\gamma_s, 0 \leq s \leq T)$ is a martingale. Therefore, we can define a new probability measure $Q$ on $(\Omega, {\cal F})$ such that for $s \in [0, T]$, we have $$\frac{dQ}{dP} \bigg|_{{\cal G}_s} = \gamma_s.$$  By Girsanov's Theorem, the law of the process $(B_s, 0 \leq s \leq T)$ under $Q$ is the same as the law of $(X_s, 0 \leq s \leq T)$ under $P$.  
Furthermore, we can see from (\ref{Xdef}) that by a standard time-change argument due to Dambis, Dubins, and Schwarz (see, for example, Theorem 3.4.6 of \cite{ks}), we can write $$\frac{X_s}{K(s)} = W_{\tau(s)},$$ where $(W_s, s \geq 0)$ is a Brownian motion under $P$ with $W_0 = x/K(0)$ and $\tau(s)$ is given by (\ref{taudef}). This proves the first part of the lemma. In particular, if $g \in [0, 1] \rightarrow \R$ is a bounded measurable function, then using $E$ to denote expectations under $P$ and $E_Q$ to denote expectations under $Q$, we have for $s \in [0, T]$,
\begin{align*}
E \bigg[ \gamma_s g \bigg( \frac{B_s}{K(s)} \bigg) \Ind_{\{0 < B_u < K(u) \: \forall u \in [0, s]\}} \bigg] &= E_Q \bigg[ g \bigg( \frac{B_s}{K(s)} \bigg) \Ind_{\{0 < B_u < K(u) \: \forall u \in [0, s]\}} \bigg] \nonumber \\
&= E \bigg[ g \bigg( \frac{X_s}{K(s)} \bigg) \Ind_{\{0 < X_u < K(u) \: \forall u \in [0, s]\}} \bigg]\\
&= E\big[g(W_{\tau(s)}) \Ind_{\{0 < W_u < 1 \: \forall u \in [0, \tau(s)]\}} \big] \\
&= \int_0^1 g(y) w_{\tau(s)}(\tfrac x {K(0)}, y) \: dy.
\end{align*}

To prove the lemma, it remains only to show that $\gamma_s = \rho_s$ for all $s \in [0, T]$.  Observe that if we write $Z_s = K'(s)B_s/2K(s)$, then $$dZ_s = \frac{K'(s)}{2K(s)} \: dB_s + \bigg( \frac{K''(s)}{2K(s)} - \frac{K'(s)^2}{2K(s)^2} \bigg) B_s \: ds,$$ and therefore $$\langle B, Z \rangle_s = \int_0^s \frac{K'(u)}{2K(u)} \: du = \frac{1}{2} \log \bigg( \frac{K(s)}{K(0)} \bigg).$$  Integrating by parts gives
\begin{align*}
&\frac{K'(s) B_s^2}{2K(s)} - \frac{K'(0)B_0^2}{2K(0)} = Z_sB_s-Z_0B_0 = \int_0^sZ_udB_u + \int_0^s B_udZ_u + \langle B, Z \rangle_s\\
&\hspace{.1in}= \int_0^s \frac{K'(u) B_u}{2K(u)} \: dB_u + \int_0^s \frac{K'(u)B_u}{2K(u)} \: dB_u + \int_0^s \bigg( \frac{K''(u)}{2K(u)} - \frac{K'(u)^2}{2K(u)^2} \bigg) B_u^2 \: du + \frac{1}{2} \log \bigg( \frac{K(s)}{K(0)} \bigg),
\end{align*}
and rearranging this equation, we get that $\gamma_s = \rho_s$, as claimed.
\end{proof}

Next, for any fixed constant $A \geq 0$, define
\begin{equation}\label{LAdef}
L_{t,A}(s) = c(t - s)^{1/3} - A,
\end{equation}
where $c$ was defined in (\ref{Lcdef}).
We now consider the case in which $K(s) = L_{t,A}(s)$. Then $L_{t,A}(s)$ is defined for $s \in [0,t_A]$, with $t_A = t - (A/c)^3$.  Suppose there is a single Brownian particle at $x \in (0, L_{t,A}(r))$, where $0 \leq r < s$, which is killed if it reaches $0$ or $L_{t,A}(u)$ at time $u \in (r, s]$.  Let $q_{r,s}^A(x,y)$ denote the ``density" for the position of this particle at time $s$, meaning that the probability that the particle is in the Borel subset $U$ of $(0, L_{t,A}(s))$ at time $s$ is $$\int_U q_{r,s}^A(x,y) \: dy.$$  Define for $0\le r\le s < t_A$, 
\begin{equation} \label{taursdef}
\tau_A(r,s) = \int_r^s \frac{1}{L_{t,A}(u)^2} \: du
\end{equation}
(we omit the parameter $t$ in the notation of $\tau_A$).

\begin{proposition}
\label{prop:qrs}
For  $0\le r\le s< t_A$, $x\in[0,L_{t,A}(r)]$ and $y\in[0,L_{t,A}(s)]$, we have
$$q_{r,s}^A(x,y) = \frac{e^{O((t-s)^{-1/3})}}{(L_{t,A}(r) L_{t,A}(s))^{1/2}} \: w_{\tau_A(r,s)} \bigg( \frac{x}{L_{t,A}(r)}, \frac{y}{L_{t,A}(s)} \bigg).$$
\end{proposition}

\begin{proof}
Let $(B_u, u \geq r)$ denote Brownian motion started at $x$ at time $r$.  Let 
$$\rho_{r,s} = \bigg( \frac{L_{t,A}(r)}{L_{t,A}(s)} \bigg)^{1/2} \exp \bigg( \frac{L_{t,A}'(s) B_s^2}{2L_{t,A}(s)} - \frac{L_{t,A}'(r)B_r^2}{2L_{t,A}(r)} - \int_r^s \frac{L_{t,A}''(u)B_u^2}{2L_{t,A}(u)} \: du \bigg).$$
By Lemma \ref{densityK}, if $h: [0, L_{t,A}(s)] \rightarrow \R$ is a bounded measurable function, then
\begin{equation}\label{LAdensity}
E \big[ \rho_{r,s} h(B_s) \Ind_{\{0 < B_u < L_{t,A}(u) \: \forall u \in [r, s]\}} \big] = \frac{1}{L_{t,A}(s)} \int_0^{L_{t,A}(s)} h(z) w_{\tau_A(r,s)}\bigg(\frac{x}{L_{t,A}(r)}, \frac{z}{L_{t,A}(s)} \bigg) \: dz.
\end{equation}
We have $$L_{t,A}'(s) = -\frac{c}{3} (t - s)^{-2/3}, \hspace{.3in}L_{t,A}''(s) = -\frac{2c}{9} (t - s)^{-5/3}.$$  On the event that $0 < B_u < L_{t,A}(u)$ for all $u \in [r, s]$, we have
\begin{align*}
&\bigg| \frac{L_{t,A}'(s) B_s^2}{2L_{t,A}(s)} - \frac{L_{t,A}'(r)B_r^2}{2L_{t,A}(r)} - \int_r^s \frac{L_{t,A}''(u) B_u^2}{2L_{t,A}(u)} \bigg| \\
&\hspace{.8in}\leq \bigg| \frac{L_{t,A}'(s) L_{t,A}(s)}{2} \bigg| + \bigg| \frac{L_{t,A}'(r) L_{t,A}(r)}{2} \bigg| + \frac{1}{2} \bigg| \int_r^s L_{t,A}''(u) L_{t,A}(u) \: du \bigg| \\
&\hspace{.8in}\leq C (t - s)^{-1/3}
\end{align*}
for some positive constant $C$.  Therefore,
\begin{equation}\label{rhors}
\rho_{r,s} = \bigg( \frac{L_{t,A}(r)}{L_{t,A}(s)} \bigg)^{1/2}e^{O((t-s)^{-1/3})}.
\end{equation}
% Also, 
% \begin{align}\label{taursasym}
% \tau(r, s) &= \int_r^s \frac{1}{c^2 (t - u)^{2/3}} \: du + \int_r^s \frac{2A}{c^3 (t - u)} \: du + O((t - s)^{-1/3}) \nonumber \\
% &= \frac{3}{c^2} \big( (t - r)^{1/3} - (t - s)^{1/3} \big) + \frac{2A}{c^3} \log \bigg( \frac{t - r}{t-s} \bigg) + O((t - s)^{-1/3}),
% \end{align}
% which tends to infinity as long as $s - r \gg t^{2/3}$ and $t - s \gg 1$.  
It now follows from (\ref{LAdensity}) and (\ref{rhors}) that $$E \big[ h(B_s) \Ind_{\{0 < B_u < L_{t,A}(u) \: \forall u \in [r,s]\}} \big] = \frac{e^{O((t-s)^{-1/3})}}{(L_{t,A}(r) L_{t,A}(s))^{1/2}} \int_0^{L_{t,A}(s)} h(z) w_{\tau_A(r,s)} \bigg( \frac{x}{L_{t,A}(r)}, \frac{z}{L_{t,A}(s)} \bigg) \: dz.$$ This implies the result.
\end{proof}

\subsection{First moment estimates}
\label{sec:first_moment_estimates}

We now return to the original setting of the paper, in which each Brownian particle drifts to the left at rate $1$ and branching events, each producing an average of $m+1$ offspring, occur at rate $\beta = 1/2m$.  Suppose there is a single particle at $x \in (0, L_{t,A}(r))$ at time $r$, where $0 \leq r < s$, and particles are killed if they reach $0$ or $L_{t,A}(u)$ at time $u \in (r, s]$.  Let $p_{r,s}^A(x,y)$ denote the ``density'' for the process at time $s$, meaning that the expected number of particles in the Borel subset $U$ of $(0, L_{t,A}(s))$ at time $s$ is $$\int_U p_{r,s}^A(x,y) \: dy.$$    By Girsanov's Theorem, the addition of the drift multiplies the density by $e^{(x - y) - t/2}$, and by the Many-to-one Lemma, the branching multiplies the density by $e^{t/2}$.  It follows that $$p_{r,s}^A(x,y) = e^{x-y} q_{r,s}^A(x,y).$$ In this section and the next one, we use this fact to estimate first and second  moments of various quantities of this process.

%Recall the definitions of $L_t(s)$ and $Z_t(s)$.  For real numbers $A \geq 0$, recall the definition of $L_{t,A}(s)$ from (\ref{LAdef}).  
Define $N_{s,A}$ to be the set particles at time $s$ that stay below the curve $L_{t,A}$ until time $s$.  We define
\begin{align*}
% L_t(s) &= c (t-s)^{1/3},\quad \text{with $c = (3\pi^2/2)^{1/3}$},\\
% L_{t,A}(s) &= c(t - s)^{1/3} - A,\quad\text{with $A \ge 1$ a fixed constant,}\\
% Z_t(s) &= \sum_{u\in N_s} z_t(X_u(s),s),\quad z_t(x,s) = L_t(s)\sin\left(\frac{\pi x}{L_t(s)}\right)e^{x-L_t(s)}\Ind_{x\in[0,L_t(s)]},\\
% Y_t(s) &= \sum_{u\in N_s} y_t(X_u(s),s),\quad y_t(x,s) = \tfrac x {L_t(r)} e^{x-L_t(s)}\\
Z_{t,A}(s) &= \sum_{u\in N_{s,A}} z_{t,A}(X_u(s),s),\quad z_{t,A}(x,s) = L_{t,A}(s)\sin\left(\frac{\pi x}{L_{t,A}(s)}\right)e^{x-L_t(s)}\Ind_{x\in[0,L_{t,A}(s)]},\\
Y_{t,A}(s) &= \sum_{u\in N_{s,A}} y_{t,A}(X_u(s),s),\quad y_{t,A}(x,s) = \tfrac x {L_{t,A}(s)} e^{x-L_t(s)}, \\
{\tilde Y}_{t,A}(s) &= \sum_{u\in N_{s,A}} {\tilde y}_{t,A}(X_u(s), s),\quad {\tilde y}_{t,A}(x,s) = e^{x-L_t(s)}.
\end{align*}
%Note that $Z_{t,0}(s) = Z_t(s)$.  Likewise, define $Y_t(s) = Y_{t,0}(s)$ and ${\tilde Y}_t(s) = {\tilde Y}_{t,0}(s)$. \nota{Pascal: this is wrong: there is no killing in $Z_t(s)$ etc.}
We also define $$y_t(x,s) = y_{t,0}(x,s), \quad {\tilde y}_t(x,s) = {\tilde y}_{t,0}(x,s).$$
Note that $Y_{t,A}(s) \leq {\tilde Y}_{t,A}(s)$.  We further define $R_{t,A}(r,s)$, for $r\le s$, to be the number of particles absorbed at the curve $L_{t,A}$ between the times $r$ and $s$.  The notation $\P_{(x,r)}$ and $\E_{(x,r)}$ denotes probabilities and expectations for our branching Brownian motion process started from a particle at the space-time point $(x,r)$.

We now collect a few estimates for $L_{t,A}(s)$ and $\tau_A(r,s)$, which were defined in \eqref{LAdef} and \eqref{taursdef} respectively.  Recall that $t_A = t - (A/c)^3$, and define 
\[
s_A = t - \left(\frac{2A}{c}\right)^3 \le t_A,
\]
so that $A/L_t(s) \le 1/2$ for every $s\le s_A$. It follows that for $s\le s_A$, we have
\begin{align}
\label{eq:LtA_Lt}
L_{t,A}(s) = L_t(s)e^{O(A(t-s)^{-1/3})}.
\end{align}
Also, a simple calculation gives, for $r\le s\le s_A$,
\begin{align}
\tau_A(r, s) &= \int_r^s \frac{1}{c^2 (t - u)^{2/3}} \: du + \int_r^s \frac{2A}{c^3 (t - u)} \: du + O(A^2(t - s)^{-1/3}) \nonumber \\
&= \frac{3}{c^2} \big( (t - r)^{1/3} - (t - s)^{1/3} \big) + \frac{2A}{c^3} \log \bigg( \frac{t - r}{t-s} \bigg) + O(A^2(t - s)^{-1/3})\nonumber\\
\label{tauasym}
&= \frac 2 {\pi^2} \left(L_t(r) - L_t(s) + \frac {2 A}{3} \log \bigg( \frac{t - r}{t-s} \bigg) + O(A^2(t - s)^{-1/3})\right).
\end{align}
It follows that for $r\le s\le s_A$,
\begin{equation}\label{tauLexpasym}
e^{-\frac{\pi^2}2\tau_A(r,s)} = e^{L_t(s) - L_t(r) + O(A^2(t - s)^{-1/3})} \bigg( \frac{t - s}{t-r} \bigg)^{\frac{2A} 3}.
\end{equation}
Furthermore, since $L_{t,A}(s) \le L_t(s)$ for every $s\le t_A$, we get by definition and a simple calculation, for every $s\le t_A$ (in particular, every $s\le s_A$),
\begin{align}
\label{taucomparison}
\tau_A(r,s) \ge \tau_0(r,s) =  \frac 2 {\pi^2} (L_t(r) - L_t(s)),
\end{align}
and also, by \eqref{eq:LtA_Lt} and the definition of $\tau_A$ from \eqref{taursdef}, for every $s\le s_A$,
\begin{align}
\label{taucomparisonweak}
\tau_A(r,s) = \tau_0(r,s)e^{O(A(t-s)^{-1/3})}.
\end{align}
%In what follows, we introduce an asymptotic regime where we let $t$ and $A$ go to infinity together in such a way that $A$ grows sufficiently slowly with $t$. More precisely, when we use the phrase \emph{as $t\to\infty$ and $A\to\infty$} we implicitly imply that the following holds:
%\begin{itemize}
% \item $A^2t^{-1/3} \to 0$.
%\end{itemize}
%Furthermore, we fix functions $\bar s(t,A)$ and $\bar (A)$, such that
%\begin{itemize}
% \item $A^2(t-\bar s(t,A))^{-1/3}\to 0$ and $\bar s(t,A)\to\infty$ as $t\to\infty$ and $A\to\infty$ and
% \item $\bar \theta(A)A \to 0$ as $A\to\infty$.
%\end{itemize}
%We also define the meaning of the statements $S \sim T$ and $S\lesssim T$, for some quantities $S$ and $T$ depending on $t,A,r,s,x,y$. The statement $S\sim T$ means that $S/T$ tends to one as $t\to\infty$ and $A\to\infty$, and uniformly in $r,s,x,y$ satisfying
%\begin{itemize}
% \item $0\le r\le s\le \bar s(t,A)$,
% \item $(t-r)/(t-s) \le e^{\bar \theta(A)}$ and
% \item $x\in [0,L_{t,A}(r)]$, $y\in [0,L_{t,A}(s)]$.
%\end{itemize}
%Similarly, the statement $S\lesssim T$ means that the upper limit as $t\to\infty$ and $A\to\infty$ of the supremum of $S/T$ over this range of parameters is at most 1.  For example, we have by \eqref{tauasym},
%\begin{equation}
% \label{tauLexpasym}
% e^{-\frac{\pi^2}2\tau_A(r,s)} = e^{L_t(s) - L_t(r) + O(A^2(t - s)^{-1/3})} \bigg( \frac{t - s}{t-r} \bigg)^{\frac{2A} 3} \sim e^{L_t(s) - L_t(r)}.
%\end{equation}

\begin{lemma}  We have for $r\le s\le s_A$ and $x\in [0,L_{t,A}(r)]$,
\label{lem:Zexp}
$$ \E_{(x,r)}[Z_{t,A}(s)] = e^{O((1 \vee A^2)(t - s)^{-1/3})} \left(\frac{t-s}{t-r}\right)^{\frac{2A} 3 + \frac 1 2} z_{t,A}(x,r).$$
%In particular,
%\begin{equation}\label{Zexp2}
%\E_{(x,r)}[Z_{t,A}(s)] \sim z_{t,A}(x,r).
%\end{equation}
\end{lemma}

\begin{proof}
By applying Proposition \ref{prop:qrs} followed by equations (\ref{wsin}) and (\ref{tauLexpasym}), we get
\begin{align*}
\E_{(x,r)}&[Z_{t,A}(s)] = \int_0^{L_{t,A}(s)} e^{x-y} q_{r,s}^A(x,y) z_{t,A}(y,s)\,dy\\
  &=e^{O((t - s)^{-1/3})} \frac{L_{t,A}(s)^{1/2}}{L_{t,A}(r)^{1/2}} e^{x-L_t(s)} \int_0^{L_{t,A}(s)} \sin\bigg(\frac{\pi y}{L_{t,A}(s)} \bigg) w_{\tau_A(r,s)} \bigg( \frac{x}{L_{t,A}(r)}, \frac{y}{L_{t,A}(s)} \bigg)\,dy \\
  &=e^{O((t - s)^{-1/3})} \frac{L_{t,A}(s)^{3/2}}{L_{t,A}(r)^{1/2}} e^{x-L_t(s)} e^{-\frac{\pi^2}2\tau_A(r,s)}\sin\bigg(\frac{\pi x}{L_{t,A}(r)} \bigg) \\
  &=e^{O((1 \vee A^2)(t - s)^{-1/3})} \frac{L_{t,A}(s)^{3/2}}{L_{t,A}(r)^{3/2}} \bigg( \frac{t - s}{t-r} \bigg)^{\frac{2A} 3} z_{t,A}(x,r).
 \end{align*}
%The result (\ref{Zexp1}) follows from the fact that $L_{t,A}(s)^3/L_{t,A}(r)^3 = e^{O(A(t - s)^{-1/3})} (t-s)/(t-r)$, and (\ref{Zexp2}) follows immediately from (\ref{Zexp1}).
The lemma follows from \eqref{eq:LtA_Lt}.
\end{proof}

\begin{lemma}\label{lem:Yexp}
Let $\gamma > 0$.  There exists a positive constant $C$, depending on $\gamma$, such that if  $r\le s\le t_A$ and $\tau_A(r,s) \geq \gamma$, then for $x\in[0,L_{t,A}(r)]$, 
$$\E_{(x,r)}[{\tilde Y}_{t,A}(s)] \leq C e^{O((t - s)^{-1/3})}  \frac{z_{t,A}(x,r)}{L_{t,A}(r)}.$$
%In particular, if $\tau_A(r,s) \geq \gamma$ then
%\begin{equation}\label{Yexp2}
%\E_{(x,r)}[{\tilde Y}_{t,A}(s)] \lesssim \frac{C z_{t,A}(x,r)}{L_t(r)}.
%\end{equation}
\end{lemma}

\begin{proof}
By Proposition \ref{prop:qrs},
\begin{align*}
\E_{(x,r)}[{\tilde Y}_{t,A}(s)] &= \int_0^{L_{t,A}(s)} e^{x-y} q^A_{r,s}(x,y) e^{y - L_t(s)} \: dy \\
&= \frac{e^{O((t - s)^{-1/3})} e^{x - L_t(s)}}{L_{t,A}(r)^{1/2} L_{t,A}(s)^{1/2}} \int_0^{L_{t,A}(s)} w_{\tau_A(r,s)} \bigg( \frac{x}{L_{t,A}(r)}, \frac{y}{L_{t,A}(s)} \bigg) \: dy.
\end{align*}
Because $\tau_A(r,s) \geq \gamma$, it follows from (\ref{vdef}) and (\ref{udef}) that
\begin{align*}
\E_{(x,r)}[{\tilde Y}_{t,A}(s)] &\leq \frac{C e^{O((t - s)^{-1/3})}e^{x - L_t(s)} e^{-\frac{\pi^2}{2} \tau_A(r,s)}}{L_{t,A}(r)^{1/2} L_{t,A}(s)^{1/2}} \int_0^{L_{t,A}(s)} \sin \bigg( \frac{\pi x}{L_{t,A}(r)} \bigg) \sin \bigg( \frac{\pi y}{L_{t,A}(s)} \bigg) \: dy \\
&\leq C e^{O((t - s)^{-1/3})}e^{x - L_t(s)} e^{-\frac{\pi^2}{2} \tau_A(r,s)} \bigg( \frac{L_{t,A}(s)}{L_{t,A}(r)} \bigg)^{1/2} \sin \bigg( \frac{\pi x}{L_{t,A}(r)}\bigg).
\end{align*}
Therefore, using \eqref{taucomparison} and the fact that $L_{t,A}$ is decreasing, we get
$$\E_{(x,r)}[{\tilde Y}_{t,A}(s)] \leq C e^{O((t - s)^{-1/3})} e^{x - L_t(r)} \sin \bigg( \frac{\pi x}{L_{t,A}(r)} \bigg),$$
as claimed. % The result (\ref{Yexp2}) follows immediately from (\ref{Yexp1}).
\end{proof}

To calculate the first moment of $R_{t,A}$, we will use the following well-known result on the hitting time of a curve by a Brownian motion. 

\begin{lemma}
\label{lem:heat_flow}
Let $b_+,b_-:\R_+\to\R$ be smooth functions and let $y \in (b_-(0),b_+(0))$. Let $u(y,s)$ be the density of Brownian motion started at $x$ and killed when hitting one of the curves $b_+$ and $b_-$. Let $H_+$ and $H_-$ denote the hitting times of the curves $b_+$ and $b_-$, respectively. Then
\[
\P_x(H_+ \in ds,\ H_+ < H_-) = - \frac 1 2 \partial_y u(y,s)\Big|_{y = b_+(s)}\, ds
\]
\end{lemma}

In words, Lemma~\ref{lem:heat_flow} says that the density at time $s$ of the hitting time of the boundary $b_+$ is equal to the heat flow of $u$ out of the boundary at time $s$. This result is so classical that it is difficult to find a complete proof of it in the literature. See e.g. \cite[p.~154, eq.~32]{ito_mckean} for an early appearance (without proof) in the case of constant boundaries and note that in our one-dimensional setting, one can easily reduce to this case by a suitable change of variables. For two different proof ideas, one more elegant, the other one more robust, both directly applicable for non-constant boundaries, one may consult \cite[Lemma~I.1.4]{lerche} and \cite[Section~3]{daniels}, respectively. For a general discussion of parabolic measure on the boundary of a space-time domain and its relation to hitting times, see \cite[Section 2.IX.13]{doob}. Lemma~\ref{lem:heat_flow} can also be deduced from the formula given in Section 1.XV.7 of that book. A more readable, but non-rigorous discussion in the time-homogeneous case can be found in \cite[Section 5.2.1]{gardiner}.

\begin{lemma}\label{lem:Rexp}
% Suppose that $(t-r)/(t-s) = e^{O(\theta)}$, with $\theta$ depending on $A$. If $\theta \ll A^{-1}$, then, as first $t\to\infty$ and then $A\to\infty$,
%For all fixed $\eps > 0$ and $A \geq 0$, there is a positive constant $C$ depending on $\eps$ and $A$ such that if $s \leq (1 - \eps) t$, then
We have for $r\le s\le s_A$ and $x\in [0,L_{t,A}(r)]$,
\begin{equation}\label{Rexp1}
\begin{split}
\E_{(x,r)}[R_{t,A}(r,s)] &\leq \pi e^{A + O((1 \vee A^2)(t - s)^{-1/3})} \left(\frac{\tau_0(r,s)}{L_t(r)} z_{t,A}(x,r) + O(y_{t,A}(x,r))\right)\\ 
&\leq \left(\frac{t-r}{t-s}\right)^{\frac{2A} 3 + \frac 1 6}\E_{(x,r)}[R_{t,A}(r,s)].
\end{split}
\end{equation}
%Also, we have
%\begin{equation}\label{Rexp2}
%\E_{(x,r)}[R_{t,A}(r,s)] \sim \pi e^A\left(\frac{\tau_A(r,s)}{L_t(r)} z_{t,A}(x,r) + O(y_{t,A}(x,r))\right).
%\end{equation}
\end{lemma}
\begin{proof}
From Lemma~\ref{lem:heat_flow} together with the many-to-one lemma, we get
\begin{equation}
\E_{(x,r)}[R_{t,A}(r,s)] =\int_r^s \left(-\frac 1 2 \frac{d}{dy}p^A_{r,u}(x,y)\Big|_{y=L_{t,A}(u)}\right)\,du.
\label{eq:heat_flow}
\end{equation}
Equation (\ref{eq:heat_flow}) implies
\begin{align*}
\E_{(x,r)}[R_{t,A}(r,s)] &=\int_r^s \left(-\frac 1 2 \frac{d}{dy}e^{x-y}q^A_{r,u}(x,y)\Big|_{y=L_{t,A}(u)}\right)\,du\\
&= \int_r^s e^{x-L_{t,A}(u)}\left(-\frac 1 2 \partial_y q^A_{r,u}(x,L_{t,A}(u))\right)\,du.
\end{align*}
Because $\partial_y q^A_{r,u}(x,L_{t,A}(u)) = \lim_{y\uparrow L_{t,A}(u)} q^A_{r,u}(x,y)/(L_{t,A}(u)-y)$, the uniform bounds on $q^A_{r,u}(x,y)$ in Proposition~\ref{prop:qrs} directly turn into uniform bounds on its derivative at $y=L_{t,A}(u)$.  Therefore,
$$\E_{(x,r)}[R_{t,A}(r,s)] = e^A e^{O((t - s)^{-1/3})} \int_r^s \frac 1 {L_{t,A}(r)^{1/2}L_{t,A}(u)^{3/2}} e^{x-L_t(u)} \left(-\frac 1 2 \partial_y w_{\tau_A(r,u)} \big( \tfrac{x}{L_{t,A}(r)}, 1 \big)\right)\,du.$$
Now (\ref{tauLexpasym}) and \eqref{eq:LtA_Lt} give
\begin{align}
\nonumber
&\E_{(x,r)}[R_{t,A}(r,s)] = e^A e^{O((1 \vee A^2)(t - s)^{-1/3})} e^{x-L_t(r)} \\
\label{Rexpeqtoshow1}
&\hspace{.3in} \times \int_r^s \frac{1}{L_{t,A}(u)^2} \bigg( \frac{t-u}{t-r} \bigg)^{\frac{2A} 3 + \frac 1 6} e^{\frac{\pi^2}2\tau_A(r,u)} \left(-\frac 1 2 \partial_y w_{\tau_A(r,u)} \big( \tfrac{x}{L_{t,A}(r)}, 1 \big)\right)\,du.
\end{align}
We claim that
\begin{align}
\nonumber
T &:= e^{x-L_t(r)}\int_r^s \frac{1}{L_{t,A}(u)^2} e^{\frac{\pi^2}2\tau_A(r,u)} \left(-\frac 1 2 \partial_y w_{\tau_A(r,u)} \big( \tfrac{x}{L_{t,A}(r)}, 1 \big)\right)\,du \\
\label{Rexpeqtoshow}
&= \pi \left(\frac{\tau_A(r,s)}{L_{t,A}(r)} z_{t,A}(x,r) + O(y_{t,A}(x,r))\right).
\end{align}
Then \eqref{Rexpeqtoshow1} and \eqref{Rexpeqtoshow}, along with \eqref{eq:LtA_Lt} and \eqref{taucomparisonweak}, 
imply the lemma because $\frac{t-u}{t-r} \le 1$ for every $u\in [r,s]$.
To prove the claim, we transform the integral in \eqref{Rexpeqtoshow} using the change of variables $\tau_A(r,u) = u'$ along with \eqref{taursdef}, to get
\begin{align*}%\label{Rexpeq1}
T= e^{x-L_t(r)} \int_0^{\tau_A(r,s)} e^{\frac{\pi^2}2 u'} \left(-\frac 1 2 \partial_y w_{u'} \big( \tfrac{x}{L_{t,A}(r)}, 1 \big)\right)\,du'.
\end{align*}
Equation~\eqref{wprimeint} now gives
\begin{align*}%\label{Rexpeq2}
T = \pi e^{x-L_t(r)} \left(\tau_A(r,s) \sin\big( \tfrac{\pi x}{L_{t,A}(r)}\big) + O(\tfrac{x}{L_{t,A}(r)})\right),
\end{align*}
which is exactly \eqref{Rexpeqtoshow}.
%which gives (\ref{Rexp2}).  To prove (\ref{Rexp1}), note that when $A > 0$ is fixed and $s \leq (1 - \eps) t$, 
%we have $e^{O((1 \vee A^2)(t - s)^{-1/3})} ((t - u)/(t - r))^{2A/3 + 1/6} \leq C$ for $u \in [r, s]$.  Therefore, we can proceed as in the previous case, except that we get constants in front of the expressions in (\ref{Rexpeq1}) and (\ref{Rexpeq2}).
\end{proof}

\subsection{Second moment estimates}
\label{sec:second_moment_estimates}

\begin{lemma}
\label{lem:Zvar}
Let $\eps$, $\gamma_1$, and $\gamma_2$ be positive numbers.  Suppose $r \le s \leq (1 - \eps) t\wedge s_A$.  Suppose also that $\tau_A(r,s) \geq \gamma_1$ and $(1 \vee A^2)(t - s)^{-1/3} \leq \gamma_2$.  Then
%and $((t-s)/(t-r))^{2A/3 + 1/2} \leq \gamma_3$.  (TODO: we don't need $\gamma_3$ if $2A/3+1/2 \ge0$. But maybe we want to extend to negative $A$ ?) 
there exists a positive constant $C$, depending on $\eps$, $\gamma_1$, and $\gamma_2$, such that
% Suppose that $(t-r)/(t-s) = e^{O(\theta)}$, with $\theta$ depending on $A$. If $\theta \ll A^{-1}$, then, as first $t\to\infty$ and then $A\to\infty$,
\[
\E_{(x,r)}[Z_{t,A}(s)^2] \leq C e^{-A} \left(\frac{\tau_0(r,s)}{L_t(r)} z_{t,A}(x,r) + y_{t,A}(x,r)\right).
\]
%where $C$ is a constant depending only on $\beta$ and the second factorial moment $m_2$ of the offspring distribution.
\end{lemma}

\begin{proof}
Let $m_2$ be the second factorial moment of the offspring distribution. Standard second moment calculations (see, for example, p. 146 of \cite{inw}) give
\begin{align}
\E_{(x,r)}[Z_{t,A}(s)^2] &= \E_{(x,r)}\left[\sum_{u\in N_s} z_{t,A}(X_u(s),s)^2\right] \nonumber \\
&\hspace{.3in}+ \beta m_2 \int_r^s \int_0^{L_{t,A}(u)} e^{x-y}q^A_{r,u}(x,y)\E_{(y,u)}[Z_{t,A}(s)]^2\,dy\, du \nonumber \\
\label{T1T2}
&=: T_1 + T_2.
\end{align}
We first bound the first term in \eqref{T1T2}.  By Proposition \ref{prop:qrs},
$$T_1 \leq \frac{C}{(L_{t,A}(r) L_{t,A}(s))^{1/2}} \int_0^{L_{t,A}(s)} e^{x-y} w_{\tau_A(r,s)} \big( \tfrac{x}{L_{t,A}(r)}, \tfrac{y}{L_{t,A}(s)} \big) L_{t,A}(s)^2 \sin \big( \tfrac{\pi y}{L_{t,A}(s)} \big)^2 e^{2(y - L_t(s))} \: dy.$$
Now using (\ref{vdef}), (\ref{Ddef}), and (\ref{udef}), along with the fact that $\tau_A(r,s) \geq \gamma_1$, we get
$$T_1 \leq \frac{C L_{t,A}(s)^{3/2} e^{x}}{L_{t,A}(r)^{1/2}} \int_0^{L_{t,A}(s)} e^{-\frac{\pi^2}{2} \tau_A(r,s)} e^{y - 2 L_t(s)} \sin \big( \tfrac{\pi x}{L_{t,A}(r)} \big) \sin \big( \tfrac{\pi y}{L_{t,A}(s)} \big)^3 \: dy.$$
Using \eqref{taucomparison}, we get
\begin{align}\label{T1}
T_1 &\leq \frac{C L_{t,A}(s)^{3/2} e^{x - L_t(r)}}{L_{t,A}(r)^{1/2}} \sin \bigg( \frac{\pi x}{L_{t,A}(r)} \bigg) \int_0^{L_{t,A}(s)} e^{y-L_t(s)} \sin \big( \tfrac{\pi y}{L_{t,A}(s)} \big)^3 \: dy \nonumber \\
&\leq \frac{C e^{-A} z_{t,A}(x,r)}{L_{t,A}(r)^{3/2} L_{t,A}(s)^{3/2}}.
\end{align}

We now bound the term $T_2$ in \eqref{T1T2}.
%From now on, $C$ denotes a constant which may depend on $\beta$ and $m_2$ and which may change from line to line.
By Proposition~\ref{prop:qrs} and Lemma~\ref{lem:Zexp},
 \begin{align*}
  T_2 &\leq C \int_r^s \int_0^{L_{t,A}(u)} \frac{e^{x-y}}{L_{t,A}(r)^{1/2}L_{t,A}(u)^{1/2}}w_{\tau_A(r,u)}\big( \tfrac{x}{L_{t,A}(r)}, \tfrac{y}{L_{t,A}(u)} \big)z_{t,A}(y,u)^2\,dy\,du.
 \end{align*}
Applying the inequality $z_{t,A}(y,u) \le \pi(L_{t,A}(u)-y)e^{y-L_t(u)}$ and using that $L_{t,A}$ is decreasing and that $L_{t,A}\le L_t$ gives
\begin{align*}
T_2 \leq CL_t(r)\int_r^s \int_0^{L_{t,A}(u)} \frac{e^{x-L_t(u)+y-L_{t,A}(u) - A}}{L_{t,A}(u)^2}w_{\tau_A(r,u)}\big( \tfrac{x}{L_{t,A}(r)}, \tfrac{y}{L_{t,A}(u)} \big)(L_{t,A}(u)-y)^2\,dy\,du.
\end{align*}
Changing variables $y \mapsto L_{t,A}(u) - y$, and using the equality $w_u(x',y') = w_u(1-x',1-y')$ for all $x',y'\in[0,1]$ together with \eqref{taucomparison} gives
\begin{equation}\label{T2prelim}
T_2 \leq CL_t(r)e^{x-L_t(r)-A} \int_r^s \frac{e^{\frac{\pi^2}{2}\tau_A(r,u)}}{L_{t,A}(u)^2} \int_0^{L_{t,A}(u)} y^2e^{-y}  w_{\tau_A(r,u)}\big(1-\tfrac{x}{L_{t,A}(r)}, \tfrac{y}{L_{t,A}(u)}\big)\,dy \,du.
\end{equation}
Now making the additional change of variables $\tau_A(r,u) \mapsto u$, using \eqref{taursdef}, and letting $h(u)$ be the number such that $\tau_A(r, h(u)) = u$, we get
$$T_2 \leq CL_t(r) e^{x-L_t(r)-A} \int_0^{\tau_A(r,s)} e^{\pi^2 u/2} \int_0^{L_{t,A}(h(u))} y^2 e^{-y} w_u \big(1-\tfrac{x}{L_{t,A}(r)}, \tfrac{y}{L_{t,A}(h(u))}\big)\,dy \,du.$$
% &= C\frac {e^{x-L_t(r)-A}} {L_t(r)}\int_0^\infty y^2e^{-y}dy \int_r^s e^{\frac{\pi^2}{2}\tau_A(r,u)}w_{\tau_A(r,u)}\big(1-\tfrac{x}{L_{t,A}(r)}, \tfrac{y}{L_{t,A}(u)}\big)\,du,
%\end{align*}
We now split the inner integral into two pieces and use Tonelli's Theorem and the fact that $L_{t,A}$ is decreasing for the first piece to get
\begin{align}\label{T3T4}
T_2 &\leq Ce^{x-L_t(r)-A} L_t(r) \int_0^{\tau_A(r,s)} e^{\pi^2 u/2} \int_0^{\frac{1}{2}L_{t,A}(s)} y^2 e^{-y} w_u \big(1-\tfrac{x}{L_{t,A}(r)}, \tfrac{y}{L_{t,A}(h(u))}\big)\,dy \,du \nonumber \\
&\hspace{.2in}+ Ce^{x-L_t(r)-A} L_t(r) \int_0^{\tau_A(r,s)} e^{\pi^2 u/2} \int_{\frac{1}{2} L_{t,A}(s)}^{L_{t,A}(h(u))} y^2 e^{-y} w_u \big(1-\tfrac{x}{L_{t,A}(r)}, \tfrac{y}{L_{t,A}(h(u))}\big)\,dy \,du \nonumber \\
&\leq Ce^{x-L_t(r)-A} L_t(r) \int_0^{\frac{1}{2} L_{t,A}(s)} y^2 e^{-y} \int_0^{\tau_A(r,s)} e^{\pi^2 u/2} \sup_{y' \in [0, y/L_{t,A}(s)]} w_u\big(1-\tfrac{x}{L_{t,A}(r)}, y' \big)\,du \,dy \nonumber \\
&\hspace{.2in} + Ce^{x-L_t(r)-A} L_t(r)^3 e^{-\frac{1}{2} L_{t,A}(s)} \int_0^{\tau_A(r,s)} e^{\pi^2 u/2} \int_0^{L_{t,A}(h(u))} w_u \big(1-\tfrac{x}{L_{t,A}(r)}, \tfrac{y}{L_{t,A}(h(u))}\big)\,dy \,du \nonumber \\
&=: T_3 + T_4.
\end{align}
By Lemma \ref{lem:wintunif}, and then using \eqref{eq:LtA_Lt} and the assumptions on $s$ (in particular that $s \leq (1 - \eps)t$),
\begin{align}\label{T3}
T_3 &\leq \frac{C e^{x - L_t(r) - A} L_t(r)}{L_{t,A}(s)} \left[\tau_A(r,s) \sin\left(\tfrac{x}{L_{t,A}(r)}\right) + \tfrac x {L_{t,A}(r)}\right] \int_0^{\infty} y^3 e^{-y} \: dy \nonumber \\
&\leq C e^{-A} \left(\frac{\tau_A(r,s)}{L_t(r)} z_{t,A}(x,r) + y_{t,A}(x,r)\right).
\end{align}
By Lemma \ref{wint2}, and using again \eqref{eq:LtA_Lt} and the assumptions on $s$,
\begin{align}\label{T4}
T_4 &\leq Ce^{x-L_t(r)-A} L_t(r)^4 e^{-\frac{1}{2} L_{t,A}(s)} \left[\tau_A(r,s) \sin\left(\tfrac{x}{L_{t,A}(r)}\right) + \tfrac x {L_{t,A}(r)}\right] \nonumber \\
&\leq C e^{-A} \left(\frac{\tau_A(r,s)}{L_t(r)} z_{t,A}(x,r) + y_{t,A}(x,r)\right).
\end{align}
The lemma now follows from \eqref{T1T2}, \eqref{T1}, \eqref{T3T4}, \eqref{T3}, and \eqref{T4}, together with \eqref{taucomparisonweak}.
\end{proof}

\begin{lemma}
 \label{lem:Rvar}
Let $\eps$, $\gamma_1$, and $\gamma_2$ be positive numbers.  Suppose $r \le s \leq (1 - \eps) t\wedge s_A$.  Suppose also that $\tau_A(r,s) \geq \gamma_1$ and $(1 \vee A^2)(t - s)^{-1/3} \leq \gamma_2$.  Then
there exists a positive constant $C$, depending on $\eps$, $\gamma_1$, and $\gamma_2$, such that
 \[
  \E_{(x,r)}[R_{t,A}(r,s)^2] \le Ce^A\left(\frac{\tau_0(r,s)}{L_t(r)} z_{t,A}(x,r) + y_{t,A}(x,r)\right).
 \]
\end{lemma}

\begin{proof}
%As in the proof of Lemma~\ref{lem:Zvar}, we have
% \begin{align}
% \nonumber
%  \E_{(x,r)}[R_{t,A}(r,s)^2] &= \E_{(x,r)}[R_{t,A}(r,s)] + \beta m_2 \int_r^s du\int_0^{L_{t,A}(u)} e^{x-y}q^A_{r,u}(x,y)\E_{(y,u)}[R_{t,A}(u,s)]^2\,dy\\
%  \label{R_T1T2}
%  &=: T_1 + T_2.
% \end{align}
%It only remains to bound $T_2$, as $T_1$ is bounded by Lemma~\ref{lem:Rexp} and the hypothesis. By Lemma~\ref{lem:Rexp}, we have for every $u\in[r,s]$ and $y\in [0,L_{t,A}(u)]$,
%\begin{align*}
%\E_{(y,u)}[R_{t,A}(u,s)]^2 &\le C e^{2A}\left(\frac{\tau_A(u,s)}{L_t(u)} z_{t,A}(y,u)+y_{t,A}(y,u)\right)^2 \\
%&\le C e^{2A}((L_{t,A}(u)-y)^2 + 1)e^{2(y-L_{t,A}(u))},
%\end{align*}
%using in the last inequality that $\tau_A(r,s)/L_t(r)\le C$ by \eqref{tauasym} and the hypotheses.
%Plugging this into \eqref{R_T1T2}, the statement readily follows by proceeding as in the proof of Lemma~\ref{lem:Zvar}.
As in the proof of Lemma~\ref{lem:Zvar}, we have
 \begin{align}
 \nonumber
  \E_{(x,r)}[R_{t,A}(r,s)^2] &= \E_{(x,r)}[R_{t,A}(r,s)] + \beta m_2 \int_r^s \int_0^{L_{t,A}(u)} e^{x-y}q^A_{r,u}(x,y)(\E_{(y,u)}[R_{t,A}(u,s)])^2\,dy\,du\\
  \label{R_T1T2}
  &=: T_1 + T_2.
 \end{align}
In view of (\ref{Rexp1}), it only remains to bound $T_2$.  For every $u\in[r,s]$ and $y\in [0,L_{t,A}(u)]$, we get, using Lemma~\ref{lem:Rexp} and the fact that $\tau_0(u,s) \leq C L_t(u)$ when $s \leq (1 - \eps)t$,
\begin{align*}
(\E_{(y,u)}[R_{t,A}(u,s)])^2 &\leq C e^{2A}\left(\frac{\tau_0(u,s)}{L_t(u)} z_{t,A}(y,u)+y_{t,A}(y,u)\right)^2 \\
&\leq C e^{2A} \big( z_{t,A}(y,u)^2+y_{t,A}(y,u)^2 \big) \\
&\leq C e^{2A} \big( (L_{t,A}(u) - y)^2 e^{2(y - L_t(u))} + e^{2(y - L_t(u))} \big) \\
&= C e^{-2(L_{t,A}(u) - y)} \big((L_{t,A}(u) - y)^2 + 1 \big).
\end{align*}
Plugging this into \eqref{R_T1T2} and using Proposition \ref{prop:qrs}, we get
\begin{align*}
T_2 &\leq C \int_r^s \int_0^{L_{t,A}(u)} \frac{e^{x-y}}{L_{t,A}(r)^{1/2} L_{t,A}(u)^{1/2}} w_{\tau_A(r,u)}\big( \tfrac{x}{L_{t,A}(r)}, \tfrac{y}{L_{t,A}(u)} \big) \\
&\hspace{2in} \times e^{-2(L_{t,A}(u) - y)} \big((L_{t,A}(u) - y)^2 + 1 \big) \: dy \: du.
\end{align*}
Now making the change of variables $y \mapsto L_{t,A}(u) - y$, using that $w_u(x',y') = w_u(1 - x', 1 - y')$, and then using (\ref{tauLexpasym}) as in the proof of Lemma~\ref{lem:Zvar}, we get
\begin{align*}
T_2 &\leq C L_t(r) \int_r^s \int_0^{L_{t,A}(u)} \frac{e^{x + y - L_{t,A}(u)}}{L_{t,A}(u)^2} w_{\tau_A(r,u)}\big( 1 - \tfrac{x}{L_{t,A}(r)}, \tfrac{y}{L_{t,A}(u)} \big) e^{-2y}(y^2 + 1) \: dy \: du \\
&\leq C L_t(r) e^{x - L_t(r) + A} \int_r^s \frac{e^{\frac{\pi^2}{2} \tau_A(r,u)}}{L_{t,A}(u)^2} \int_0^{L_{t,A}(u)} w_{\tau_A(r,u)}\big(1 - \tfrac{x}{L_{t,A}(r)}, \tfrac{y}{L_{t,A}(u)} \big) e^{-y}(y^2 + 1) \: dy \: du.
\end{align*}
Note that this expression is identical to the expression in (\ref{T2prelim}) except that the sign of $A$ in the exponential in front of the integral is reversed, and we have $y^2 + 1$ in place of $y^2$ in the integrand.  Consequently, we can follow the same steps as in the proof of Lemma~\ref{lem:Zvar} to obtain
$$T_2 \leq  C e^A \left(\frac{\tau_0(r,s)}{L_t(r)} z_{t,A}(x,r) + y_{t,A}(x,r)\right),$$
which completes the proof of the lemma.
\end{proof}

\section{Particle configurations}\label{configsec}

Our goal in this section is to deduce Proposition \ref{configpropnew} from results in \cite{bbs3}.  
The strategy of the proofs in \cite{bbs3} is to show that if at time zero there is a single particle at $x > 0$, then for all $\kappa > 0$, the configuration of particles at time $\kappa t^{2/3}$ will satisfy certain conditions.  The rest of the proofs then use only what has been established about the configuration of particles at time $\kappa t^{2/3}$.  Consequently, the results in \cite{bbs3} immediately extend to any initial configuration of particles for which these conditions hold at time $\kappa t^{2/3}$.  This observation yields Lemma \ref{configlemma} below.  We define 
$${\tilde Y}_t(s) = \sum_{u\in N_s} {\tilde y}_t(X_u(s),s),$$
which is similar to ${\tilde Y}_{t,A}(s)$ defined at the beginning of Section \ref{sec:first_moment_estimates}, except that here particles are only killed at the origin and not at the curve $L_{t,A}$.

\begin{lemma}\label{configlemma}
Suppose we have a sequence of possibly random initial configurations $(\nu_n)_{n=1}^{\infty}$ such that the following conditions hold for a corresponding sequence of times $(t_n)_{n=1}^{\infty}$:
\begin{enumerate}
\item The times $t_n$ do not depend on the evolution of the branching Brownian motion after time zero, and $t_n \rightarrow_p \infty$ as $n \rightarrow \infty$.

\item For all $\eps > 0$ and $\kappa > 0$, there is a positive constant $C_{13}$, depending on $\eps$ and $\kappa$, such that for sufficiently large $n$,
\begin{equation}\label{configasm2}
\P_{\nu_n} \bigg( {\tilde Y}_{t_n}(\kappa t_n^{2/3}) \leq \frac{C_{13}}{L_{t_n}(\kappa t_n^{2/3})} \bigg) > 1 - \eps.
\end{equation}

\item For all $\eps > 0$ and $\kappa > 0$, there are positive constants $C_{14}$ and $C_{15}$, depending on $\eps$ and $\kappa$, such that for sufficiently large $n$,
\begin{equation}\label{configasm3}
\P_{\nu_n}(C_{14} \leq Z_{t_n}(\kappa t_n^{2/3}) \leq C_{15}) > 1 - \eps.
\end{equation}

\item For all $\kappa > 0$ and $A \geq 0$, we have
\begin{equation}\label{configasm4}
\lim_{n \rightarrow \infty} \P_{\nu_n}\big(R(\kappa t_n^{2/3}) < L_{t_n}(\kappa t_n^{2/3}) - A \big) = 1.
\end{equation}
\end{enumerate}
Let $0 < \delta < 1/2$.  Then the three conclusions of Proposition \ref{configpropnew} hold.
\end{lemma}

\begin{proof}
This proposition essentially restates the results of \cite{bbs3} in the context of the present paper.  The second, third, and fourth conditions that we require for the sequence $(t_n)_{n=1}^{\infty}$ are the three conclusions of Lemma 15 of \cite{bbs3}, while the first condition that $t_n \rightarrow \infty$ in probability corresponds to the condition in \cite{bbs3} that the position $x$ of the initial particle tends to infinity.  The first conclusion of Proposition \ref{configpropnew} is Theorem 1 of \cite{bbs3}.  The second conclusion of Proposition \ref{configpropnew} is Theorem 2 in \cite{bbs3}.  The third conclusion of Proposition \ref{configpropnew} is a combination of Theorems 3 and 4 in \cite{bbs3}.  Proposition \ref{configpropnew} holds because these four theorems in \cite{bbs3} are deduced from Lemma 15 in \cite{bbs3}.  When $q = 0$, the following adaptations are required to obtain the result in the present context:
\begin{itemize}
\item In \cite{bbs3}, the branching rate is $1$ and the drift is $-\sqrt{2}$.  However, it is straightforward to translate results into our setting by a simple scaling.  

\item Lemma 15 of \cite{bbs3} includes a stronger form (\ref{configasm3}), in which the bounds are proved when the term $\sin(\pi x/L_t(s))$ in the definition of $Z_t(s)$ from (\ref{Zdef}) is replaced by $\sin(\pi x/(L_t(0) + \alpha))$ for any $\alpha \in \R$.  However, we have $$|(L_t(0) + \alpha) - L_t(\kappa t^{2/3})| \leq C(\kappa + |\alpha|)$$ for some positive constant $C$, so the ratio of the two sine terms will be bounded above and below by positive constants with high probability as long as (\ref{configasm4}) holds and $t_n \rightarrow \infty$ in probability.  Therefore, establishing (\ref{configasm3}) is sufficient.  

\item Theorems 2, 3, and 4 of \cite{bbs3} are stated for the case when $s = ut$ for some $u \in (0, 1)$.  However, it is not hard to see that the proof extends to the case where $s \sim ut$ as $t \rightarrow \infty$, with the constants being uniform over $u \in [\delta, 1 - \delta]$, and then a subsequence argument gives the results in the form stated here.

\item The results in \cite{bbs3} are stated for a fixed initial configuration of particles.  However, because the proof in \cite{bbs3} ultimately works from the random configuration at time $\kappa t^{2/3}$, the only possible complication comes from the randomness of the times $t_n$.  In \cite{bbs3}, Theorems 1 and 2 are probability statements that hold when the position $x$ of the initial particle tends to infinity, while Theorems 3 and 4 establish convergence in distribution as $x \rightarrow \infty$.  The requirement that the random times $t_n$ tend to infinity in probability is therefore sufficient for these results to carry over to the present context.

\item In \cite{bbs3}, it is assumed that at the time of a birth event, a particle splits into two other particles.  However, as long as $q = 0$, the only change that results from considering a general offspring distribution is that a different constant appears in front of the second moment estimates, which does not affect the results.  Results of Bramson \cite{bram83} are needed to prove Theorem 2 in \cite{bbs3}, but those results hold under the more general offspring distributions considered here when $q = 0$.  Note in particular that equation (1.2$'$) on page 5 of \cite{bram83} is satisfied when the offspring distribution has finite variance.
\end{itemize}
The claim that Proposition \ref{configpropnew} holds even when $q > 0$ requires a bit more care.  Indeed, the initial configuration with a single particle at $x_n$, with $x_n \rightarrow \infty$, does not fulfill the four conditions in the lemma when $q > 0$ because of the possibility that all descendants of the initial particle could die out.  Nevertheless, once these four conditions, which correspond to Lemma 15 of \cite{bbs3}, are established, one deduces Theorems 1, 3, and 4 in \cite{bbs3} using moment estimates, which change only by a constant factor when $q > 0$.  Therefore, the first and third conclusions of Proposition \ref{configpropnew} follow from the arguments in \cite{bbs3} without change.  Some additional argument is needed, however, to obtain the second conclusion of Proposition \ref{configpropnew} because the proof of Theorem 2 in \cite{bbs3} uses a result of Bramson \cite{bbs3} which is valid only when $q = 0$.

To extend the second conclusion of Proposition \ref{configpropnew} to the case $q > 0$, we modify the process as follows.  First, we construct the original branching Brownian motion in two stages.  In the first stage, we construct the process without absorption at zero.  At the second stage, we truncate any particle trajectories that hit zero.  Now we can construct a modified process by deleting all particles that do not have an infinite line of descent in the first stage of this construction.  This yields a new branching Brownian motion with $q = 0$ that includes a subset of the particles in the original branching Brownian motion.  In particular, for any fixed $s > 0$, the law of the new process at time $s$, conditioned on the original branching Brownian motion at time $s$, is obtained by independently retaining each particle of the original process with probability $1 - q$.

We check that the four conditions of the lemma hold for the new process.  Condition 1 is immediate because we will use the same times $t_n$ as in the original process, while conditions 2 and 4 and the upper bound in (\ref{configasm3}) hold because the particles in the new process are a subset of the particles in the original process.  To establish the lower bound in (\ref{configasm3}), note that (\ref{configasm4}) implies that for all $\theta > 0$, with probability tending to one as $n \rightarrow \infty$, no individual particle in the original process contributes more than $\theta$ to $Z_{t_n}(\kappa t_n^{2/3})$.  Now, suppose $z_1, \dots, z_m$ is a sequence of numbers such that $z_1 + \dots + z_m = z$ and $z_i \leq \theta$ for all $i$.  Let $\xi_1, \dots, \xi_m$ be independent Bernoulli$(1-q)$ random variables, and let $Z = z_1 \xi_1 + \dots + z_m \xi_m$.  Then $E[Z] = (1 - q) z$ and $\mbox{Var}(Z) = q(1-q)(z_1^2 + \dots + z_m^2) \leq q(1 - q) \theta z$.  By applying this observation to the numbers $z_{t_n}(X_u(\kappa t_n^{2/3}), 0)$ for $u \in N_{t_n}$ and $\theta$ sufficiently small, and then using Chebyshev's Inequality, we obtain the lower bound in (\ref{configasm3}).

It now follows from the result when $q = 0$ that the conclusion (\ref{configconc2}) holds for the new process.  Because the particles in the new process are a subset of the particles in the original process, we immediately get the lower bound in (\ref{configconc2}) for the original process.  Finally, recall that for any time $s$, the position of the right-most particle is the same in the new process as in the original process with probability $1 - q$.  Therefore, the upper bound in (\ref{configconc2}) for the original process holds with probability at least $1 - \eps/(1 - q)$, which is sufficient.
\end{proof}

We are now able to prove Proposition \ref{configpropnew} by showing that the hypotheses of Proposition~\ref{configpropnew} imply those of Lemma \ref{configlemma}.

\begin{proof}[Proof of Proposition \ref{configpropnew}]
Suppose that the hypotheses of Proposition \ref{configpropnew} are satisfied.
The first condition of Lemma \ref{configlemma} holds by assumption.

Using that $\sin(x) \geq 2x/\pi$ and $\sin(\pi - x) \geq 2x/\pi$ for all $x \in [0, \pi/2]$, we have for all $x \in [0, L_{t_n}(0) - A]$,
\begin{equation}\label{yzrat}
\frac{y_{t_n,0}(x,0)}{z_{t_n}(x,0)} = \frac{x}{L_{t_n}(0)^2 \sin( \frac{\pi x}{L_{t_n}(0)})} \leq \frac{1}{2A}.
\end{equation}
Because $A$ is arbitrary and $(Z_{t_n}(0))_{n=1}^{\infty}$ is tight, the assumption that $L_{t_n}(0) - R(0) \rightarrow_p \infty$ implies that $Y_{t_n}(0) \rightarrow_p 0$ as $n \rightarrow \infty$.

Let $\eps > 0$ and $\kappa > 0$.  To establish the second, third, and fourth conditions in Lemma \ref{configlemma}, we consider the branching Brownian motion with particles killed when they reach either the origin or the curve $s \mapsto L_{t_n}(s)$, run for time $\kappa t_n^{2/3}$.
We will need to make some moment calculations, conditional on the initial configuration of particles.  By Markov's Inequality,
Lemma~\ref{lem:Rexp} with $A = 0$, and equation (\ref{tauasym}), there is a positive constant $C$, depending on $\kappa$, such that
$$\P_{\nu_n}(R_{t_n}(0, \kappa t_n^{2/3}) \geq 1|{\cal F}_0) \leq \E_{\nu_n}[R_{t_n}(0, \kappa t_n^{2/3})|{\cal F}_0] \leq C \bigg( \frac{Z_{t_n}(0)}{L_{t_n}(0)} + Y_{t_n}(0) \bigg).$$  Because $L_{t_n}(0) \rightarrow_p \infty$ and $Y_{t_n}(0) \rightarrow_p 0$ as $n \rightarrow \infty$, and $(Z_{t_n}(0))_{n=1}^{\infty}$ is tight, we can deduce that
\begin{equation}\label{Ris0}
\lim_{n \rightarrow \infty} \P_{\nu_n}(R_{t_n}(0, \kappa t_n^{2/3}) \geq 1) = 0.
\end{equation}
Thus, we may disregard the possibility that particles are killed at $L_{t_n}(s)$ before time $\kappa t_n^{2/3}$.

By Lemma \ref{lem:Yexp} with $A = 0$,
\begin{equation}\label{Ytnexp}
\E_{\nu_n}[{\tilde Y}_{t_n,0}(\kappa t_n^{2/3})|{\cal F}_0] \leq \frac{C Z_{t_n}(0)}{L_{t_n}(0)},
\end{equation}
where the positive constant $C$ depends on $\kappa$.  Because the sequence $(Z_{t_n}(0))_{n=1}^{\infty}$ is tight and
$L_{t_n}(0) \geq L_{t_n}(\kappa t^{2/3})$, the second condition (\ref{configasm2}) in Lemma \ref{configlemma} follows from (\ref{Ytnexp}) and Markov's Inequality, along with (\ref{Ris0}).

From Lemma \ref{lem:Zexp} with $A = 0$, and the fact $(Z_{t_n}(0))_{n=1}^{\infty}$ is tight, we conclude that for all $\eps > 0$ and $\delta > 0$, for sufficiently large $n$ we have, on an event of probability at least $1 - \eps/2$,
$$\delta \leq \E_{\nu_n}[Z_{t_n,0}(\kappa t_n^{2/3})|{\cal F}_0] \leq \frac{1}{\delta}.$$
By Lemma \ref{lem:Zvar} with $A = 0$, there is a positive constant $C$ such that
$$\Var_{\nu_n}(Z_{t_n,0}(\kappa t_n^{2/3})|{\cal F}_0) \leq C \bigg( \frac{Z_{t_n}(0)}{L_{t_n}(0)} + Y_{t_n}(0) \bigg),$$ 
and the right-hand side tends to zero in probability as $n \rightarrow \infty$ by the argument before (\ref{Ris0}).  In view of our assumptions on the initial configurations as well as (\ref{Ris0}), the third condition (\ref{configasm3}) in Lemma \ref{configlemma} now follows from an application of Chebyshev's Inequality.

Because ${\tilde y}_{t_n,0}(L_{t_n}(\kappa t_n^{2/3}) - A, \kappa t_n^{2/3}) = e^{-A},$ the fourth condition (\ref{configasm4}) in Lemma \ref{configlemma} follows immediately from (\ref{configasm2}).
\end{proof}

\section{Convergence to the CSBP: small time steps}\label{CSBPsec}

In this section we state and prove a result (Proposition~\ref{prop:csbp_small_step}) which will be at the heart of the proof of Theorem~\ref{CSBPthm} in Section~\ref{sec:csbp_proof}.
% To that end we follow the strategy and arguments laid out in \cite{bbs2}, but with several changes.
%
%%Recall that we want to show that  under hypothesis (H), we have as $t\to\infty$
%% \[
%%  Z_t \circ f_t^{-1}\quad \stackrel{\text{fidis}}{\Longrightarrow} \quad\text{Neveu's CSBP: } \Xi
%% \]
%%where $f_t(s) = \log\left(\frac t {t-s}\right)$.
%
%Let us define $V_t(s):= Z_{f_t^{-1}(s)}, s\ge 0$, where $f_t(s) = \log\left(\frac t {t-s}\right)$, such that $f_t^{-1}(s) = (1-e^{-s})t$. 
%
%The convergence will be established through convergence of Laplace transforms.
%Proposition 43 \cite{bbs2} (which is just Theorem 8.2 in Chapter 4 of \cite{ek86} in the present context) states that
%
%\begin{proposition}
%Under (H)  the finite-dimensional distributions of $(V_t(s),s \geq 0)$ converge to those of $(\Xi(s), t \geq 0)$ as $t \rightarrow \infty$ if and only if for all $j \geq 0$, all $0 \leq s_1 < s_2 < \dots < s_j \leq u < u + s$, all bounded continuous functions $h_1, \dots, h_j: [0, \infty) \rightarrow \R$, and all $f \in {\cal E}$, we have
%\begin{equation}\label{mainVt}
%\lim_{t \rightarrow \infty} E \bigg[ \bigg( f(V_t(u + s)) - f(V_t(u)) - \int_u^{u + s} Af(V_t(r)) \: dr \bigg) \prod_{i=1}^j h_i(V_t(s_i)) \bigg] = 0.
%\end{equation}
%\end{proposition}
%
%\nota{If we do like that, the time change will give us intervals of different length?}

\subsection{Notation in this section}

We will make heavy use of the results in Sections~\ref{sec:first_moment_estimates} and \ref{sec:second_moment_estimates}. In particular, we use all the notation introduced in Section~\ref{sec:first_moment_estimates}. Whenever the symbol $A$ appears, we will always tacitly assume that $A\ge1$.

In what follows, it will be necessary for us to let both $t$ and $A$ go to infinity. To this end, \emph{we will always first let $t$, then $A$ go to infinity}. We therefore introduce the following two symbols:

\begin{itemize}
\item $\eps_t$: denotes a quantity which is bounded in absolute value by a function $h(A,t)$ satisfying:
\[
\forall A\ge 1: \lim_{t\to\infty} h(A,t) = 0.
\]
\item $\eps_{A,t}$: denotes a quantity which is bounded in absolute value by a function $h(A,t)$ satisfying:
\[
\lim_{A\to\infty} \limsup_{t\to\infty} h(A,t) = 0.
\]
\end{itemize}
Note that the first condition is stronger than the second one.

Furthermore, as above, the symbol $O(\cdot)$ denotes a quantity bounded in absolute value by a constant times the quantity inside the parentheses. Also, throughout the section, we fix $\Lambda > 1$ and a positive function $\bar\theta$ such that $\bar\theta(A)A^2\to 0$ as $A\to\infty$. The functions $h$ above and the constant in the definition of $O(\cdot)$ may only depend on the offspring distribution of the branching Brownian motion and on $\Lambda$ and $\bar{\theta}$. 

Throughout the section, let $r\le s$ such that $s\le (1-\Lambda^{-1})t$ and $t-s=e^{-\theta}(t-r)$, for some $\theta\in [\bar\theta(A)/2,\bar\theta(A)]$. All estimates are meant to be uniform in $r$ and $s$ respecting these constraints.

Note that with this notation, we have
\begin{align}\label{26}
\frac{\tau_0(r,s) }{ L_t(r)} =2\pi^{-2} (1-e^{-\theta/3}) = \frac2{3\pi^{2}}\theta (  1  +  O(\theta)).
\end{align}
In particular, for all $r\le r'\le s'\le s$,
\begin{align}\label{27}
\frac{\tau_0(r',s') }{ L_t(r')} &=O(\theta).
\end{align}

The main step in the proof of Theorem \ref{CSBPthm} will be to show the following proposition.

\begin{proposition}
\label{prop:csbp_small_step}
Set $a=\frac{2}{3}(a_{\ref{eq:abelian}}+\log \pi)+\frac{1}{2}$. Then, uniformly in $\lambda\in [\Lambda^{-1},\Lambda]$, on the event $\{\forall u\in N_r: X_u(r) \le L_{t,A}(r)\}$, we have
\[
\E[e^{-\lambda Z_t(s)}\,|\,\F_r] = \exp\{(-\lambda + \theta (\Psi_{a,2/3}(\lambda)+\eps_{A,t}))Z_t(r) + O(AY_t(r))\}.
\]
\end{proposition}

The proof of this proposition will be decomposed into several steps. Inspired by \cite{bbs2}, we decompose the particles into those crossing the curve $L_{t,A}$ and those staying below it. The particles crossing the curve are exactly the ones causing the jumps in the CSBP. In Section~\ref{sec:csbp_jump}, we give an asymptotic result for the Laplace transform of such a jump. In Section~\ref{sec:csbp_small_step}, we use this result to prove Proposition~\ref{prop:csbp_small_step}.

\subsection{One particle at \texorpdfstring{$L_{t,A}$}{L\_\{t,A\}}}
\label{sec:csbp_jump}

\begin{lemma}\label{lem:jump}
Uniformly in $\lambda \in [\Lambda^{-1},\Lambda]$  and  $q\in[r,s-t^{2/3}],$ 
\begin{equation}
\label{eq:jump}
\E_{(L_{t,A}(q),q)} [e^{-\lambda Z_t(s)}] =\exp \left\{\pi e^{-A}(\Psi_{a_{\ref{eq:jump}},1}(\lambda) - A\lambda + \eps_{A,t}) \right\},
\end{equation}
with $a_{\ref{eq:jump}} = a_{\ref{eq:abelian}}+\log \pi$.
\end{lemma}

The following lemma will be needed for the proof of Lemma~\ref{lem:jump}.

\begin{lemma}\label{lem:deriv_appears}
Let $y: (0, \infty) \rightarrow (0, \infty)$ be a function such that $y(t)\to\infty$ and $y(t) = o(t^{1/3})$ as $t\to\infty$.
Let $f: (0, \infty) \rightarrow (0, \infty)$ be a function such that $f(t) = o(t^{2/3})$ as $t\to\infty$.
Then uniformly in $q\in[r,s-t^{2/3}]$, $q'\in [q,q+f(t)]$, and $\lambda \in [\Lambda^{-1},\Lambda]$, as $t\to\infty$, we have
\begin{equation}
\E_{(L_t(q)-y(t),q')}[e^{-\lambda Z_t(s)}] = \exp\left\{-(\lambda+\eps_t+O(\theta))\pi y(t)e^{-y(t)}\right\}.
\end{equation}
\end{lemma}

\begin{proof}
Write $x' = L_t(q) - y(t)$. Under $\P_{(x',q')}$, we have $Z_t(s) = Z_{t,0}(s)$ on the event $\{R_{t,0}(q',s) = 0\}$. Hence,
\begin{align}
\label{eq:701}
\left|\E_{(x',q')}[e^{-\lambda Z_t(s)}] - \E_{(x',q')}[e^{-\lambda Z_{t,0}(s)}]\right| \le \P_{(x',q')}(R_{t,0}(q',s) \ge 1) \le \E_{(x',q')}[R_{t,0}(q',s)].
\end{align}
By Lemma~\ref{lem:Rexp} and \eqref{27},
\begin{align}
\label{eq:702}
\E_{(x',q')}[R_{t,0}(q',s)] \le C(\theta z_t(x',q') + y_t(x',q')).
\end{align}
Furthermore, using that $e^{-z} = 1-z+O(z^2)$ for $z\ge 0$, we have
\begin{align}
\label{eq:703}
\E_{(x',q')}[e^{-\lambda Z_{t,0}(s)}] = 1 - \lambda \E_{(x',q')}[Z_{t,0}(s)] + O( \E_{(x',q')}[Z_{t,0}(s)^2]).
\end{align}
By Lemma~\ref{lem:Zexp},
\begin{align}
\label{eq:704}
 \E_{(x',q')}[Z_{t,0}(s)] = (1+O(\theta) + \eps_t) z_t(x',q').
\end{align}
As for the second moment, to apply Lemma~\ref{lem:Zvar}, note that $\tau_0(q',s) \ge \gamma_1$ for some $\gamma_1>0$, since $q' \le s-t^{2/3} + f(t)$ and $f(t) = o(t^{2/3})$ by assumption. Hence, for $t$ large enough, by Lemma~\ref{lem:Zvar} and \eqref{27},
\begin{align}
\label{eq:705}
 \E_{(x',q')}[Z_{t,0}(s)^2] \le C\left(\theta z_t(x',q') + y_t(x',q')\right).
\end{align}
Combining \eqref{eq:701}, \eqref{eq:702}, \eqref{eq:703}, \eqref{eq:704} and \eqref{eq:705}, we have for large enough $t$,
\begin{align}
\label{eq:706}
\E_{(x',q')}[e^{-\lambda Z_t(s)}] = 1 - (\lambda +\eps_t+O(\theta)) z_t(x',q') + O(y_t(x',q')).
\end{align}
Now using that $x' = L_t(q) - y(t)$ and $y(t) = o(t^{1/3}) = o(L_t(q'))$, along with the fact that $L_t(q) - L_t(q') \rightarrow 0$ as $t \rightarrow \infty$ because $q' \in [q, q + f(t)]$, we get
$$z_t(x',q') = L_t(q') \sin\left(\frac{\pi (y(t) - (L_t(q) - L_t(q')))}{L_t(q')}\right) e^{L_t(q) - L_t(q') - y(t)} = (1+\eps_t)\pi y(t)e^{-y(t)}.$$
Furthermore,
\[
y_t(x',q') = \frac{x'}{L_t(q')} e^{L_t(q) - L_t(q') - y(t)} \le (1 + \eps_t) e^{-y(t)}.
\]
It is also easy to check that
\[
z_t(x',q')^2 + y_t(x',q')^2 = O(y_t(x',q')).
\]
It follows from the above that the RHS of \eqref{eq:706} is at least $1/2$ for $t$ large enough, since $y(t) \to\infty$ as $t\to\infty$ by assumption. Using the equality $1-x = e^{-x + O(x^2)}$ for $x\in [0,1/2]$, equation~\eqref{eq:706} together with the above equations gives
\begin{align}
\E_{(x',q')}[e^{-\lambda Z_t(s)}]
&= \exp\left(-(\lambda +\eps_t+O(\theta)) z_t(x',q') + O(y_t(x',q')) + O(z_t(x',q')^2+y_t(x',q')^2)\right) \nonumber\\
&=  \exp\left(-[(\lambda +\eps_t+O(\theta)) \pi y(t) + O(1)]e^{-y(t)}\right),\nonumber
\end{align}
which implies the statement of the lemma, since $y(t)\to\infty$ as $t\to\infty$.
\end{proof}

\begin{proof}[Proof of Lemma~\ref{lem:jump}]
We start by proceeding as in the proof of Theorem~\ref{survivex}.  Let $g: (0, \infty) \rightarrow (0, \infty)$ be an increasing function that satisfies (\ref{Tygy}).  Let $y: (0, \infty) \rightarrow (0, \infty)$ be defined so that, similarly to (\ref{ycond}), we have
$$\lim_{t \rightarrow \infty} y(t) = \infty, \hspace{.5in} \lim_{t \rightarrow \infty} \frac{y(t)}{L_t(0)} = 0, \hspace{.5in} \lim_{t \rightarrow \infty} t^{-2/3} g(A + y(t)) = 0.$$
Starting with one particle at $L_{t,A}(q)$ at time $q$, we stop particles as soon as they hit the point $L_{t,A}(q)-y(t) = L_t(q) - A - y(t)$. We denote again by $K_t$ the number of particles hitting that point and by $w_1,\ldots,w_{K_t}$ the times they hit it. Then $w_i \in [q,q+g(A + y(t))]$ for all $i=1,\ldots,K_t$ with probability $1-\eps_t$ by (\ref{Tygy}).  We can apply Lemma~\ref{lem:deriv_appears} with $A + y(t)$ in place of $y(t)$ and $f(t) = g(A + y(t))$ to get
\begin{align}
\nonumber
\E_{(L_{t,A}(q),q)}[e^{-\lambda Z_t(s)}] 
&= \E_{(L_{t,A}(q),q)}\left[\prod_{i=1}^{K_t} \E_{(L_t(q) - A - y(t),w_i)}[e^{-\lambda Z_t(s)}]\right]\\
&= \E_{(L_{t,A}(q),q)}\left[\exp\left(-K_t (\lambda +\eps_t+O(\theta)) \pi (A+y(t))e^{-A-y(t)}\right)\right]+ \eps_t.
\end{align}
Recall that $y(t)e^{-y(t)} K_t$ converges in law to $W$, the random variable from Lemma~\ref{neveuW}. It follows that 
\begin{align}
\label{eq:714}
\E_{(L_{t,A}(q),q)}[e^{-\lambda Z_t(s)}] = 
 \E[\exp(-\pi e^{-A}(\lambda +O(\theta)) W)] + \eps_t.
\end{align}
Note that $\lambda + O(\theta) =\lambda(1+O(\theta))$, uniformly in $\lambda\ge \Lambda^{-1}$. Hence, by \eqref{eq:714}, combined with Lemma~\ref{neveuW}, as $A\to\infty$, we have
\begin{align}\label{eq:715}
\E_{(L_{t,A}(q),q)}[e^{-\lambda Z_t(s)}]
\nonumber
&= \exp\left\{\Psi_{a_{\ref{eq:abelian}},1}(\pi e^{-A}(1 +O(\theta))\lambda)+o(e^{-A}) + \eps_t\right\}\\
&= \exp\{\pi e^{-A} (1+O(\theta))\lambda(\log \lambda + a_{\ref{eq:abelian}} + \log \pi -A + O(\theta) + \eps_{A,t})\}.
\end{align}
Setting $a_{\ref{eq:jump}} = a_{\ref{eq:abelian}} + \log \pi$ and using the fact that $\theta A \le \bar\theta(A) A \to 0$ as $A\to\infty$, equation \eqref{eq:715} implies
\begin{align}
\label{eq:716}
\E_{(L_{t,A}(q),q)}[e^{-\lambda Z_t(s)}] = \exp\{\pi e^{-A}(\Psi_{a_{\ref{eq:jump}},1}(\lambda) -A\lambda + \eps_{A,t})\},
\end{align}
which finishes the proof of the lemma.
\end{proof}

\subsection{Proof of Proposition~\ref{prop:csbp_small_step}}
\label{sec:csbp_small_step}

Decomposing into the descendants of the particles living at time $r$, it is enough to show that for every $x\in[0,L_{t,A}(r)]$, we have
\begin{align}
\label{eq:720}
\E_{(x,r)}[e^{-\lambda Z_t(s)}] = \exp\{(-\lambda + \theta (\Psi_{a,2/3}(\lambda)+\eps_{A,t}))z_t(x,r) + O(Ay_t(x,r))\}.
\end{align}
Fix $x\in[0,L_{t,A}(r)]$ throughout the section. We adapt an idea from \cite{bbs2} and stop the particles the moment they hit the curve $L_{t,A}$ during the time interval $[r,s]$. We denote by $\L_{t,A}$ the set of those particles, identifying a particle with the time it hits the curve (one can do this more formally using the concept of \emph{stopping lines} from \cite{chauvin91}).  For every particle hitting the curve at time $u$, we denote by $Z_t^{(u)}(s)$ the contribution to $Z_t(s)$ of the descendants of $u$. We then have the following decomposition:

\begin{equation}\label{master decomposition}
Z_t(s) = Z'_{t,A}(s)+\sum_{u \in \L_{t,A}}Z_t^{(u)}(s),
\end{equation}
where 
\begin{align*}
Z'_{t,A}(s) = \sum_{u\in N_{s,A}} z_t(X_u(s),s),
\end{align*}
with $N_{s,A}$ defined in Section~\ref{sec:first_moment_estimates}. In what follows, we will also make use of the quantities $Z_{t,A}$, $Y_{t,A}$ etc.~defined in that section.

By the (strong) branching property, conditionally on $\L_{t,A}$, the $Z^{(u)}$ are independent and independent of $Z'_{t,A}(s)$.  Therefore, we can write
\begin{align*}
\E_{(x,r)}[e^{-\lambda Z_t(s)}] &= \E_{(x,r)}\left[  e^{-\lambda Z'_{t,A}(s)}  \prod_{u\in\L_{t,A}} e^{-\lambda Z^{(u)}_t(s)}\right]\\
&= \E_{(x,r)}\left[  e^{-\lambda Z'_{t,A}(s)}  \prod_{u\in\L_{t,A}} \E_{(L_{t,A}(u),u)}[e^{-\lambda Z_t(s)}]\right].
\end{align*}
Define $s'=s-t^{2/3}$.
Using Markov's inequality and conditioning on $\F_{s'}$, then applying Lemma~\ref{lem:Rexp}, Lemma~\ref{lem:Zexp}, and Lemma~\ref{lem:Yexp}, we have
\begin{align*}
\P_{(x,r)}(\L_{t,A} \cap [s',s]\ne \emptyset) &\le \E_{(x,r)}[R_{t,A}(s',s)] \\
&\le \E_{(x,r)}[Z_{t,A}(s')\eps_t + O(e^A Y_{t,A}(s'))(1 + \eps_t)]\\
&= z_{t,A}(x,r)\eps_t.
\end{align*}
Hence,
\begin{align}
\label{eq:721}
\E_{(x,r)}[e^{-\lambda Z_t(s)}] = \E_{(x,r)}\left[  e^{-\lambda Z'_{t,A}(s)}  \prod_{u\in\L_{t,A}\cap[r,s']} \E_{(L_{t,A}(u),u)}[e^{-\lambda Z_t(s)}]\right] + z_{t,A}(x,r)\eps_t.
\end{align}
Equation \eqref{eq:721} and Lemma~\ref{lem:jump} now give
\begin{align}
\label{eq:722}
\E_{(x,r)}[e^{-\lambda Z_t(s)}] = \E_{(x,r)}\left[  e^{-\lambda Z'_{t,A}(s) + R_{t,A}(r,s')\pi e^{-A}(\Psi_{a_{\ref{eq:jump}},1}(\lambda) - A\lambda + \eps_{A,t})}\right]  + z_{t,A}(x,r)\eps_t.
\end{align}

We next claim that \eqref{eq:722} implies
\begin{align}
\nonumber
\E_{(x,r)}[e^{-\lambda Z_t(s)}] &= 1-\lambda \E_{(x,r)}[Z_{t,A}(s)] + \pi e^{-A}(\Psi_{a_{\ref{eq:jump}},1}(\lambda) - A\lambda + \eps_{A,t})\E_{(x,r)}[R_{t,A}(r,s')]\\
\label{eq:723}
&+ O(\E_{(x,r)}[Z_{t,A}(s)^2 + (Ae^{-A}R_{t,A}(r,s'))^2 + AY_{t,A}(s)]) +  z_{t,A}(x,r)\eps_t.
\end{align}
Indeed, the upper bound follows using first the fact that $Z'_{t,A}(s)\ge Z_{t,A}(s)$, which can be seen by observing that $z_t(x,s) \ge z_{t,A}(x,s)$ for every $x\ge0$ because the function $L \mapsto L \sin(\pi x/L)$ is increasing on $[x, \infty)$, and then using that $e^{-x} = 1-x+O(x^2)$ for $x\ge0$.  Note that the second summand in the exponent on the RHS of \eqref{eq:722} is always negative, because the product in the expectation on the RHS of \eqref{eq:721} is bounded by $1$. The lower bound, on the other hand, follows from the equality $Z'_{t,A}(s) = Z_{t,A}(s) + O(A Y_{t,A}(s))$, which is a consequence of the fact that $z_t(x,s) = z_{t,A}(x,s) + O(Ay_{t,A}(x,s))$), together with the inequality $e^{-x} \ge 1-x$ for every $x\ge 0$.

We now gather the following estimates:
\begin{align*}
\E_{(x,r)}[Z_{t,A}(s)] &= e^{-\theta(\frac{2}{3}A+\frac{1}{2})+\eps_t} z_{t,A}(x,r) && \text{by Lemma~\ref{lem:Zexp}}\\
\E_{(x,r)}[R_{t,A}(r,s')] &= \pi e^{A+O(\theta A)+\eps_t}\\
&\hspace{-.5in}\times \left( \left(\frac{2}{3\pi^2}\theta(1+ O(\theta)) + \eps_t \right) z_{t,A}(x,r) + O(y_{t,A}(x,r))\right) && \text{by Lemma~\ref{lem:Rexp} and \eqref{26}}\\
\E_{(x,r)}[Z_{t,A}(s)^2] &\le Ce^{-A}(\theta z_{t,A}(x,r)+y_{t,A}(x,r)) && \text{by Lemma~\ref{lem:Zvar}}\\
\E_{(x,r)}[R_{t,A}(s)^2] &\le Ce^A(\theta z_{t,A}(x,r)+y_{t,A}(x,r)) && \text{by Lemma~\ref{lem:Rvar}}\\
\E_{(x,r)}[Y_{t,A}(s)] &\le z_{t,A}(x,r)\eps_t  && \text{by Lemma~\ref{lem:Yexp}}
\end{align*}
Using that $\theta A^2 \le \bar\theta(A)A^2 \to 0$ as $A\to\infty$, equation~\eqref{eq:723} together with the above estimates gives after some calculation, with $a_{\ref{eq:724}} = \frac{2}{3}a_{\ref{eq:jump}}+\frac{1}{2}$,
\begin{align}
\label{eq:724}
\E_{(x,r)}[e^{-\lambda Z_t(s)}] = 1 + (-\lambda + \theta(\Psi_{a_{\ref{eq:724}},2/3}(\lambda)+\eps_{A,t}))z_{t,A}(x,r) + O(Ay_{t,A}(x,r)).
\end{align}
Using that $z_{t,A}(x,r) =O(A e^{-A})$ and $y_{t,A}(x,r) \le e^{-A}$ for $x\le L_{t,A}(r)$, as well as $z_{t,A}(x,r)^2 = O(y_{t,A}(x,r))$, we get 
\begin{align}
\label{eq:725}
\E_{(x,r)}[e^{-\lambda Z_t(s)}] = \exp\left(-\lambda + \theta(\Psi_{a_{\ref{eq:724}},2/3}(\lambda)+\eps_{A,t}))z_{t,A}(x,r) + O(Ay_{t,A}(x,r))\right).
\end{align}
Using again the equality $z_t(x,r) = z_{t,A}(x,r) + O(Ay_{t,A}(x,r))$, equation \eqref{eq:725} implies \eqref{eq:720} with $a = a_{\ref{eq:724}}$ and concludes the proof of Proposition~\ref{prop:csbp_small_step}.

\section{Convergence to the CSBP: proof of Theorem~\ref{CSBPthm}}
\label{sec:csbp_proof}

Before getting to the heart of the proof, we perform a series of reductions. First, it is enough to consider initial conditions such that $Z$ is positive almost surely. For, suppose that, under $\P_{\nu_t}$, we have $Z_t(0)\to_p 0$ as $t\to\infty$. If we superpose $\lfloor Z_t(0)\rfloor$ independent copies of the system, we can reduce this case to the case where $Z_t(0)\to_p 1$ as $t\to\infty$. 
Indeed, once we have established that the finite-dimensional distributions of these superposed processes converge to the CSBP $(\Xi(u), u \geq 0)$ started from $1$, which almost surely stays finite for all times, it will follow that when $Z_t(0)\to0$ in probability as $t\to\infty$, the finite-dimensional distributions of the process converge to those of the process that is identically zero.
This argument is easily generalized to the general case where $Z$ has an atom at $0$ of arbitrary positive mass. 

Next, the finite-dimensional convergence can be easily deduced from the one-dimensional convergence result and the Markov property of the process. For this, it is enough to show that for every $u\in (0,1)$, with high probability, the configuration of particles at time $ut$ again satisfies the hypotheses, with $(1-u)t$ instead of $t$, i.e.~that $Z_t(ut) \Rightarrow Z$ for some random variable $Z>0$ and $L_t(ut) - R(ut) \to \infty$ in probability (note that $L_t(ut) = L_{(1-u)t}(0)$ and $z_t(x,ut) = z_{(1-u)t}(x,0)$). The first is precisely a consequence of the one-dimensional convergence result, together with the fact that Neveu's CSBP does not hit 0. The second on the other hand follows from the second part of Proposition~\ref{configpropnew}.

Finally, by a simple conditioning argument, it is enough for the one-dimensional convergence result to assume an initial condition such that, under $\P_{\nu_t}$, we have $Z_t(0) \to_p z_0$ as $t\to\infty$, for some constant $z_0>0$. We assume this for the rest of the section.  Also, all probabilities and expectations for the rest of this section will be taken under $\P_{\nu_t}$, so we will omit the subscript.

We now go on to prove the one-dimensional convergence. Fix $\tau > 0$. It is enough to show the following: for every $\lambda > 0$, we have
\begin{align}
\label{eq:750}
\lim_{t\to\infty} \E[e^{-\lambda Z_t(t(1-e^{-\tau}))}] = e^{-z_0u_\tau(\lambda)},
\end{align}
where $u_\tau(\lambda)$ is the function from \eqref{csbpLaplace} corresponding to the CSBP with branching mechanism $\Psi_{a,2/3}$, with $a$ being the number from Proposition~\ref{prop:csbp_small_step}.
We do this by discretizing time. As in Section~\ref{sec:csbp_small_step}, we introduce a parameter $A$ which goes slowly to $\infty$ with $t$. Recall the notation $\eps_t$ and $\eps_{A,t}$ from that section, as well as the function $\bar\theta$. Quantities denoted by $\eps_t$ and $\eps_{A,t}$ now may also depend on the initial condition and on $\tau$. For $A$ sufficiently large, choose $\theta\in [\bar{\theta}(A)/2,\bar{\theta}(A)]$ such that $\tau = K\theta$ for some $K\in\N$. Define $t_k = t(1-e^{-k\theta})$ for $k=0,\ldots,K$, so that $t_K = t(1-e^{-\tau})$.

Set $\F_k = \F_{t_k}$. By assumption, there exists a sequence $a_t\to \infty$ such that $L_t(0)-a_t-R(0) \to\infty$ and $a_tY_t(0)\to 0$ in probability as $t\to\infty$. We assume without loss of generality that $a_t \le t^{1/6}$ for every $t\ge0$. Define the events
\begin{align*}
G_k = \{\forall j\in\{0,\ldots,k\}: R(t_j) \le L_{t,A}(t_j),\ Y_t(t_j) \le Z_t(t_j)/a_t\},\quad k=0,\ldots,K,
\end{align*}
so that $G_k\in \F_k$ for all $k\in\{0,\ldots,K\}$. 
\begin{lemma}
\label{lem:Gk}
We have
\(
\P(G_K) \ge 1-\eps_t.
\)
\end{lemma}
\begin{proof}
We have $\P(R(0) \le L_{t,A}(0),\ Y_t(0) \le Z_t(0)/a_t)\ge 1-\eps_t$ by assumption. Let $k\in\{1,\ldots,K\}$. By part 2 of Proposition~\ref{configpropnew}, we have $L_{t,A}(t_k) - R(t_k) \to \infty$ in probability as $t\to\infty$. Furthermore, by part 3 of Proposition~\ref{configpropnew}, we have $L_t(t_k)Y_t(t_k)/Z_t(t_k)\to c$ in probability as $t\to\infty$, for some constant $c\in(0,\infty)$. Hence, since $a_t \le t^{1/6}$ by assumption, $a_t Y_t(t_k)/Z_t(t_k)\to 0$ in probability as $t\to\infty$. A union bound shows that $\P(G_K) = 1-\eps_t$.
\end{proof}

Now fix $\lambda > 0$. For every $\delta \in\R$, define recursively,
\begin{align*}
\lambda_K^{(\delta)} &= \lambda\\
\lambda_k^{(\delta)} &= \lambda_{k+1}^{(\delta)} - \theta(\Psi_{a,2/3}(\lambda_{k+1}^{(\delta)}) - \delta).
\end{align*}

%We now claim the following:
\begin{lemma}
\label{lem:csbp_lambda_delta}
Fix $\lambda > 0$.
\begin{enumerate}
\item There exists $\Lambda > 1$ such that for $|\delta|$ small enough and for $\theta$ small enough (a priori depending on $\delta$), we have $\lambda_k^{(\delta)} \in [\Lambda^{-1},\Lambda]$ for all $k=0,\ldots,K$.
\item For every $\eps>0$, there exists $\delta > 0$ such that for all $\theta$ sufficiently small, 
$$\lambda_0^{(\delta)},\lambda_0^{(-\delta)} \in [u_\tau(\lambda)-\eps,u_\tau(\lambda)+\eps].$$
\item For every $\delta>0$, we have for sufficiently large $A$ and $t$, for every $k=0,\ldots,K$,
\begin{align}
\label{eq:752}
\E[e^{-\lambda_k^{(\delta)}Z_t(t_k)}\Ind_{G_k}] -\P(G_K\backslash G_k) \le \E\left[e^{-\lambda Z_t(t_K)}\Ind_{G_K}\right] \le \E[e^{-\lambda_k^{(-\delta)}Z_t(t_k)}\Ind_{G_k}].
\end{align}
\end{enumerate}
\end{lemma}

\begin{proof}
Parts 1 and 2 follow from standard results on convergence of Euler schemes for ordinary differential equations, after suitable localization arguments. We provide the details for completeness.

Write $\Psi = \Psi_{a,2/3}$ for simplicity. Fix $\lambda > 0$. Choose $\Lambda > 1$ such that $u_t(\lambda) \in (\Lambda^{-1},\Lambda)$ for all $t\in [0,\tau]$. Define $\Psi^\Lambda:\R\to\R$ by
\[
\Psi^\Lambda(x) = 
\begin{cases}
\Psi(x) & \tif x\in [\Lambda^{-1},\Lambda]\\
\Psi(\Lambda^{-1}) & \tif x\le \Lambda^{-1}\\
\Psi(\Lambda) & \tif x \ge \Lambda.
\end{cases}
\]
Then $\Psi^\Lambda$ is a Lipschitz function. If we define $(\lambda_k^{(\delta,\Lambda)})_{k=0,\ldots,K}$ recursively by
\begin{align*}
\lambda_K^{(\delta, \Lambda)} &= \lambda\\
\lambda_k^{(\delta, \Lambda)} &= \lambda_{k+1}^{(\delta)} - \theta(\Psi^\Lambda(\lambda_{k+1}^{(\delta)}) - \delta),
\end{align*}
then $(\lambda_{K-k}^{(\delta,\Lambda)})_{k=0,\ldots,K}$ is the explicit Euler scheme for the ODE
\begin{align}
\label{eq:ode_delta_lambda}
y' = -(\Psi^\Lambda(y)-\delta),\quad y(0) = \lambda
\end{align}
on the interval $[0,\tau]$, with timestep $\theta$. The right-hand side being a Lipschitz function of $u$, it is well-known that the Euler scheme converges, i.e., if $y^{(\delta,\Lambda)}$ denotes the solution to the ODE \eqref{eq:ode_delta_lambda}, then as $\theta\to 0$,
\[
\max_{k=0,\ldots,K} |\lambda_{K-k}^{(\delta,\Lambda)} - y^{(\delta,\Lambda)}(k\theta)| \to 0.
\]
Furthermore, because the right-hand side of \eqref{eq:ode_delta_lambda} depends continuously on the parameter $\delta$, we have $y^{(\delta,\Lambda)} \to y^{(0,\Lambda)} \eqqcolon y^{(\Lambda)}$ as $\delta\to 0$, uniformly on $[0,\tau]$. Finally, since $\Psi^\Lambda = \Psi$ on $[\Lambda^{-1},\Lambda]$, and $(u_t(\lambda))_{t\in[0,\tau]}$ is the solution to the ODE \eqref{diffeq} and satisfies $u_t(\lambda)\in [\Lambda^{-1},\Lambda]$ for all $t\in[0,\tau]$, we have $y^{(\Lambda)}(t) = u_t(\lambda)$ for all $t\in[0,\tau]$. Altogether, the above arguments show
\begin{align}
\label{eq:756}
\lim_{\delta\to0} \lim_{\theta\to0} \max_{k=0,\ldots,K}  |\lambda_{K-k}^{(\delta,\Lambda)} - u_{k\theta}(\lambda)| = 0.
\end{align}
It remains to remove the localization: since $u_t(\lambda)$ is contained in the open interval $(\Lambda^{-1},\Lambda)$ for all $t\in [0,\tau]$, by \eqref{eq:756}, there exists $\delta_0>0$, such that for all $|\delta|\le \delta_0$, for $\theta$ sufficiently small, $\lambda_k^{(\delta,\Lambda)} \in [\Lambda^{-1},\Lambda]$ for all $k\in\{0,\ldots,K\}$. But since $\Psi^\Lambda = \Psi$ on $[\Lambda^{-1},\Lambda]$, a direct recurrence argument shows that $\lambda_k^{(\delta,\Lambda)} = \lambda_k^{(\delta)}$ for all $k=0,\ldots,K$. This proves part 1. Part~2 immediately follows, using again \eqref{eq:756}.

We now prove part 3 of the lemma. Fix $\delta > 0$. Choose $\Lambda>1$ such that $e^{-\tau} > \Lambda^{-1}$ and such that the first part of the lemma holds with this $\Lambda$. By Proposition~\ref{prop:csbp_small_step}, we have for $A$ and $t$ sufficiently large, for every $\lambda'\in [\Lambda^{-1},\Lambda]$, and every $k=0,\ldots,K-1$, almost surely,
\[
e^{(-\lambda' +\theta (\Psi_{a,2/3}(\lambda')-\delta)) Z_t(t_k)}\Ind_{G_k} \le \E\left[e^{-\lambda' Z_t(t_{k+1})}\,|\,\F_k\right]\Ind_{G_k} \le e^{(-\lambda' +\theta (\Psi_{a,2/3}(\lambda')+\delta)) Z_t(t_k)}\Ind_{G_k}.
\]
In particular, using the first part of the lemma, for every $\delta>0$ small enough, for $A$ and $t$ sufficiently large, we have for every $k=0,\ldots,K-1$, almost surely,
\begin{align}
\label{eq:753+}
\E\left[e^{-\lambda_{k+1}^{(\delta)} Z_t(t_{k+1})}\,|\,\F_k\right]\Ind_{G_k} &\ge e^{-\lambda_k^{(\delta)} Z_t(t_k)}\Ind_{G_k},\\
\label{eq:753-}
\E\left[e^{-\lambda_{k+1}^{(-\delta)} Z_t(t_{k+1})}\,|\,\F_k\right]\Ind_{G_k} &\le e^{-\lambda_k^{(-\delta)} Z_t(t_k)}\Ind_{G_k}.
\end{align}

We now prove \eqref{eq:752} by induction. For $k=K$, the inequalities trivially hold. Let $k<K$ and assume \eqref{eq:752} holds for $k+1$, i.e.
\begin{align}
\label{eq:752bis}
\E[e^{-\lambda_{k+1}^{(\delta)}Z_t(t_{k+1})}\Ind_{G_{k+1}}] -\P(G_K\backslash G_{k+1}) \le \E\left[e^{-\lambda Z_t(t_K)}\Ind_{G_K}\right] \le \E[e^{-\lambda_{k+1}^{(-\delta)}Z_t(t_{k+1})}\Ind_{G_{k+1}}].
\end{align}
Using that $G_{k+1}\subset G_k$, equation \eqref{eq:752bis} easily implies
\begin{align}
\label{eq:752ter}
\E[e^{-\lambda_{k+1}^{(\delta)}Z_t(t_{k+1})}\Ind_{G_k}] -\P(G_K\backslash G_k) \le \E\left[e^{-\lambda Z_t(t_K)}\Ind_{G_K}\right] \le \E[e^{-\lambda_{k+1}^{(-\delta)}Z_t(t_{k+1})}\Ind_{G_k}].
\end{align}
Equations \eqref{eq:753+}, \eqref{eq:753-} and \eqref{eq:752ter} now show that \eqref{eq:752} holds for $k$. This finishes the induction.
\end{proof}

We can now wrap up the proof of \eqref{eq:750}. By Lemma~\ref{lem:Gk}, we have $\P(G_K) = 1-\eps_t$, and so
\begin{equation}\label{Gbound}
\E[e^{-\lambda Z_t(t_K)}] = \E[e^{-\lambda Z_t(t_K)}\Ind_{G_K}] + \eps_t.
\end{equation}
Now fix $\ep>0$ and choose $\delta>0$ as in the second part of Lemma~\ref{lem:csbp_lambda_delta}. We then have by the third part of that lemma and (\ref{Gbound}), for $A$ and $t$ sufficiently large,
\[
\E[e^{-(u_\tau(\lambda)+\ep)Z_t(0)}\Ind_{G_0}]-\eps_t\le \E[e^{-\lambda Z_t(t_K)}] \le \E[e^{-(u_\tau(\lambda)-\ep)Z_t(0)}\Ind_{G_0}]+\eps_t.
\]
Hence, letting $t\to\infty$, and using the assumption on the initial configuration, we have
\[
e^{-(u_\tau(\lambda)+\ep)z_0}\le \liminf_{t\to\infty} \E[e^{-\lambda Z_t(t(1-e^{-\tau}))}] \le \limsup_{t\to\infty} \E[e^{-\lambda Z_t(t(1-e^{-\tau}))}] \le e^{-(u_\tau(\lambda)-\ep)z_0}.
\]
Letting $\ep \to0$ proves \eqref{eq:750} and thus finishes the proof of Theorem~\ref{CSBPthm}.

\subsection*{Acknowledgments}  The authors warmly thank Julien Berestycki for a number of productive discussions while the ideas of this project were being formulated, along with some further discussions throughout the course of the project.  They also thank Louigi Addario-Berry, Nathana\"el Berestycki, \'Eric Brunet, Simon Harris, and Matt Roberts for helpful discussions at an early stage of the project, and they thank a referee for comments which improved the exposition of the paper.  Pascal Maillard was supported in part by grants ANR-20-CE92-0010-01 and ANR-11-LABX-0040 (ANR program ``Investissements d'Avenir''). Jason Schweinsberg was supported in part by NSF Grants DMS-1206195 and DMS-1707953.

\end{document}